\documentclass[12pt,leqno]{amsart}
\usepackage[normalem]{ulem}
\usepackage{verbatim}
\usepackage{amsthm}
\usepackage{amstext}
\usepackage{amssymb}
\usepackage[dvipdfmx]{graphicx}
\usepackage{mathrsfs}
\usepackage{amscd}
\usepackage[dvipsnames]{xcolor}

\makeatletter
\numberwithin{equation}{section}
\numberwithin{figure}{section}

{
  \color{brown}%
}%
{}

\theoremstyle{plain}
\newtheorem{main theorem}{Main Theorem}
\newtheorem{theorem}{Theorem}[section]
\newtheorem{lemma}[theorem]{Lemma}

\newtheorem{corollary}[theorem]{Corollary}
\newtheorem{proposition}[theorem]{Proposition}
\newtheorem{claim}[theorem]{Claim}

\theoremstyle{definition}
\newtheorem{definition}[theorem]{Definition}
\newtheorem{remark}[theorem]{Remark}
\newtheorem{example}[theorem]{Example}

\newtheorem{problem}[theorem]{Problem}
\newtheorem{condition}[theorem]{Condition}
\newtheorem{toy model}[theorem]{Toy Model}
\newtheorem{data}[theorem]{Data}
\newcommand{\mdim}{\mathrm{mdim}}
\newcommand{\diam}{\mathrm{diam}}
\newcommand{\widim}{\mathrm{Widim}}
\newcommand{\dist}{\mathrm{dist}}
\newcommand{\supp}{\mathrm{supp}}
\newcommand{\norm}[1]{\left|\!\left|#1\right|\!\right|}

\newcommand{\rdim}{\mathrm{rdim}}
\newcommand{\ver}{\mathrm{Ver}}

\makeatother

\begin{document}

\title[Double variational principle for mean dimension]{Double variational principle for mean dimension}

\author{Elon Lindenstrauss, Masaki Tsukamoto}



\date{\today}

\thanks{E.L was partially supported by ISF grant 891/15. 
M.T. was partially supported by JSPS KAKENHI 18K03275. }

\maketitle

\begin{abstract}
We develop a variational principle between mean dimension theory and rate distortion theory.
We consider a minimax problem about the rate distortion dimension with respect to two variables 
(metrics and measures).
We prove that the minimax value is 
equal to the mean dimension for a dynamical system with the marker property.
The proof exhibits a new combination of ergodic theory, rate distortion theory and geometric measure theory.
Along the way of the proof, we also show that if a dynamical system has the marker property then it has a metric for which
the upper metric mean dimension is equal to the mean dimension.
\end{abstract}

\section{Introduction}  \label{section: introduction}

\subsection{Statement of the main result}  \label{subsection: statement of the main result}

The purpose of this paper is to develop a new variational principle in dynamical systems theory.
We first quickly prepare the terminologies and state the main result.
Backgrounds will be explained in \S \ref{subsection: backgrounds}.

A pair $(\mathcal{X},T)$ is called a \textbf{dynamical system} if $\mathcal{X}$ is a compact metrizable space and 
$T:\mathcal{X}\to \mathcal{X}$ is a homeomorphism.
We denote by $\mathscr{M}^T(\mathcal{X})$ the set of $T$-invariant Borel probability measures on $\mathcal{X}$.
The standard \textit{variational principle} (\cite{Goodwyn, Dinaburg, Goodman}) 
states that the topological entropy $h_{\mathrm{top}}(T)$ is equal to the 
supremum of the ergodic-theoretic entropy $h_\mu(T)$ over $\mu \in \mathscr{M}^T(\mathcal{X})$:
\begin{equation}  \label{eq: variational principle}
   h_{\mathrm{top}}(T) = \sup_{\mu\in \mathscr{M}^T(\mathcal{X})} h_\mu(T).
\end{equation}
Our main result below is an analogous formula in \textit{mean dimension theory}.

Mean dimension (denoted by $\mdim(\mathcal{X},T)$) is a topological invariant of dynamical systems 
introduced by Gromov \cite{Gromov}.
It counts how many parameters per iterate we need to describe an orbit in $(\mathcal{X},T)$.
We review its definition in \S \ref{subsection: topological and metric mean dimensions}.
We would like to connect mean dimension to some information-theoretic quantity as in 
(\ref{eq: variational principle}).
An appropriate notion turns out to be \textit{rate distortion dimension}, which was first introduced by 
Kawabata--Dembo \cite{Kawabata--Dembo}.

Let $\mathscr{D}(\mathcal{X})$ be the set of metrics (i.e. distance functions) on $\mathcal{X}$ compatible with the topology.
Take $d\in \mathscr{D}(\mathcal{X})$ and $\mu\in \mathscr{M}^T(\mathcal{X})$.
Consider a stochastic process $\{T^n x\}_{n\in \mathbb{Z}}$ where $x\in \mathcal{X}$ is chosen randomly according to $\mu$.
We denote by $R(d,\mu,\varepsilon)$, $\varepsilon >0$, the \textit{rate distortion function} of this process
with respect to the distortion measure $d$. 
This evaluates how many bits per iterate we need to describe the process within the distortion (w.r.t. $d$)
bound by $\varepsilon$.
We review its definition in \S \ref{subsection: rate distortion theory}.
We define the \textbf{upper/lower rate distortion dimensions} by\footnote{Throughout the paper 
we assume that the base of the logarithm is two. 
The natural logarithm (i.e. the logarithm of base $e$) is written as $\ln(\cdot)$.} 
\begin{equation}  \label{eq: definition of rate distortion dimension}
   \begin{split}
     & \overline{\rdim}(\mathcal{X},T,d,\mu) =  \limsup_{\varepsilon \to 0} \frac{R(d,\mu,\varepsilon)}{\log(1/\varepsilon)},  \\
     & \underline{\rdim}(\mathcal{X},T,d,\mu)  = \liminf_{\varepsilon\to 0} \frac{R(d,\mu,\varepsilon)}{\log (1/\varepsilon)}.
   \end{split}
\end{equation}  
When the upper and lower limits coincide, we denote their common value by $\rdim(\mathcal{X},T,d,\mu)$.

A dynamical system $(\mathcal{X},T)$ is said to have the \textbf{marker property} if for any $N>0$ there exists 
an open set $U\subset \mathcal{X}$ satisfying 
\[   \mathcal{X} = \bigcup_{n\in \mathbb{Z}} T^{-n} U, \quad U\cap T^{-n} U = \emptyset \quad (\forall 1\leq n\leq N). \]
This property implies that $(\mathcal{X},T)$ is free (i.e. it has no periodic points).
Free minimal systems and their extensions have the marker property.
The marker property has been intensively used in the context of the \textit{embedding problem}
(see \S \ref{subsection: backgrounds}) and related issues
\cite{Lindenstrauss, Gutman--Lindenstrauss--Tsukamoto, Gutman--Tsukamoto minimal, Gutman--Qiao--Tsukamoto}.

Now we can state our main result.

\begin{theorem}[\textbf{Double Variational Principle}]  \label{main theorem}
  If a dynamical system $(\mathcal{X},T)$ has the marker property, then
     \begin{equation}  \label{eq: double variational principle}
        \begin{split}
         \mdim(\mathcal{X},T)  & =  \min_{d\in \mathscr{D}(\mathcal{X})} \sup_{\mu\in \mathscr{M}^T(\mathcal{X})} 
         \overline{\rdim}(\mathcal{X},T,d,\mu) \\
       & =  \min_{d\in \mathscr{D}(\mathcal{X})} \sup_{\mu\in \mathscr{M}^T(\mathcal{X})} 
       \underline{\rdim}(\mathcal{X},T,d,\mu).   
       \end{split}
     \end{equation}
   Here ``$\min$'' indicates that the minimum is attained by some $d$.  
\end{theorem}

A fundamental difference between the standard variational principle (\ref{eq: variational principle})
and our new one (\ref{eq: double variational principle}) is that 
(\ref{eq: variational principle}) is a maximazation problem with respect to the single variable $\mu$ 
wheres (\ref{eq: double variational principle}) is a \textit{minimax} problem with respect to the two variables 
$d$ and $\mu$.
We have used the word ``double'' in order to emphasize that there exist two variables playing different roles.

\begin{remark}
    \begin{enumerate}
         \item We have adopted the minimax approach in (\ref{eq: double variational principle}).
      It might also look interesting to consider a \textit{maximin} approach:
       \begin{equation}  \label{eq: maximin approach}  
         \begin{split}
               \sup_{\mu \in \mathscr{M}^T(\mathcal{X})} \inf_{d\in \mathscr{D}(\mathcal{X})} \overline{\rdim}(\mathcal{X},T,d,\mu), \\
               \sup_{\mu \in \mathscr{M}^T(\mathcal{X})} \inf_{d\in \mathscr{D}(\mathcal{X})} \underline{\rdim}(\mathcal{X},T,d,\mu).
        \end{split}
       \end{equation}
     However this turns out to be fruitless.
     Indeed both the quantities in (\ref{eq: maximin approach}) are always zero.
     More strongly, we can prove that for any dynamical system $(\mathcal{X},T)$ and $\mu\in \mathscr{M}^T(\mathcal{X})$
     there exists $d\in \mathscr{D}(\mathcal{X})$ satisfying $\overline{\rdim}(\mathcal{X},T,d,\mu) = 0$.
     \item The rate distortion dimension depends on both metrics $d$ and measures $\mu$.
             It might look more satisfactory to define a certain ``ergodic-theoretic mean dimension'' in a purely measure-theoretic way
             and prove a corresponding ``variational principle'' for mean dimension.
             But this naive approach is impossible: 
             Let us consider an arbitrary ergodic measure-preserving system.
             By the Jewett--Krieger theorem \cite{Jewett, Krieger} we can find a dynamical system $(\mathcal{X},T)$ which has only one 
             invariant probability measure (say, $\mu$) and that $(\mathcal{X},\mu,T)$ is measure-theoretically isomorphic to the given 
             system.
             It is known that uniquely ergodic systems have zero mean dimension \cite[Theorem 5.4]{Lindenstrauss--Weiss}.
             So, if we have a ``variational principle'', the given measure-preserving system must have zero ``ergodic-theoretic 
             mean dimension''.
    \item We conjecture that the marker property assumption in Theorem \ref{main theorem} is, in fact, unnecessary.         
            The proof of Theorem \ref{main theorem} shows that the inequality
            \[  \mdim(\mathcal{X},T) \leq \inf_{d\in \mathscr{D}(\mathcal{X})} \sup_{\mu\in \mathscr{M}^T(\mathcal{X})}
                 \underline{\rdim}(\mathcal{X},T,d,\mu)  \]
            holds true for any dynamical system $(\mathcal{X},T)$.
            So the remaining problem is how to prove the reverse inequality. 
            See \S \ref{subsection: open problems and future directions} 
            for further discussions. 
    \end{enumerate}
\end{remark}

\subsection{Backgrounds}  \label{subsection: backgrounds}

Mean dimension provides a nontrivial information for infinite dimensional dynamical systems of 
infinite topological entropy.
It has several applications which cannot be touched within the framework of topological entropy
\cite{Lindenstrauss--Weiss, Lindenstrauss, Gutman--Lindenstrauss--Tsukamoto, Gutman--Tsukamoto minimal, 
Gutman--Qiao--Tsukamoto, Meyerovitch--Tsukamoto}.
As an illustration, we explain an application to the problem of \textit{embedding dynamical systems into shift actions}.

Consider the $N$-dimensional cube $C_N := [0,1]^N$ and let 
$\sigma: (C_N)^\mathbb{Z} \to (C_N)^\mathbb{Z}$ be the shift on the alphabet $C_N$ where $(C_N)^\mathbb{Z}$ is endowed with the 
standard product topology.
The mean dimension of $\left((C_N)^\mathbb{Z},\sigma\right)$ is $N$.
Given a dynamical system $(\mathcal{X},T)$, we are interested in whether we can 
embed\footnote{$f:\mathcal{X}\to (C_N)^\mathbb{Z}$ is called an \textbf{embedding of a dynamical system}
if it is a topological embedding and satisfies $f\circ T = \sigma\circ f$.} 
it into $\left((C_N)^\mathbb{Z},\sigma\right)$ or not.

Periodic points are an obvious obstruction: If $(\mathcal{X},T)$ has too many periodic points (e.g. if 
the set of fixed points has dimension greater than $N$) then it cannot be embedded into $(C_N)^\mathbb{Z}$.
Mean dimension provides another obstruction: If we can embed $(\mathcal{X},T)$ into $(C_N)^\mathbb{Z}$ then 
$\mdim(\mathcal{X},T) \leq \mdim\left((C_N)^\mathbb{Z},\sigma\right) = N$.
We can construct free (and, moreover, minimal) dynamical systems of arbitrary mean dimension
\cite[Proposition 3.5]{Lindenstrauss--Weiss}.
So there exist plenty of examples which are free but cannot be embedded into $(C_N)^\mathbb{Z}$.
(This observation by \cite{Lindenstrauss--Weiss} solved a question posed by Auslander in 1970s.)

Somehow surprisingly, a partial converse also holds.
Based on the work \cite{Lindenstrauss}, the papers 
\cite{Gutman--Tsukamoto minimal, Gutman--Qiao--Tsukamoto}\footnote{The papers
\cite{Gutman--Tsukamoto minimal, Gutman--Qiao--Tsukamoto} used the ideas of communication theory and signal processing. 
This is another manifestation of the intimate connections between mean dimension and information theory.}
proved that if $(\mathcal{X},T)$ has the marker property and satisfies $\mdim(\mathcal{X},T) < N/2$ then we 
can embed it into $(C_N)^\mathbb{Z}$.
The example in \cite{Lindenstrauss--Tsukamoto example}
shows that the condition $\mdim(\mathcal{X},T) < N/2$ is optimal.
These results demonstrate that mean dimension is certainly a reasonable measure of the ``size''
of dynamical systems.

It is classically known that the concepts of entropy and dimension are closely connected
(R\'{e}nyi \cite{Renyi}, Kolmogorov--Tihomirov \cite{Kolmogorov--Tihomirov} and 
Kawabata--Dembo \cite{Kawabata--Dembo})\footnote{It seems that these attract new interests of information theory 
researchers in the context of compressed sensing; see, e.g. \cite{Wu--Verdu} and \cite{Rezagah--Jalali--Erkip--Poor}.}.
So it is natural to expect that we can approach to mean dimension from the entropy theory viewpoint.
The first attempt of such an approach was made by Weiss and the first named author \cite{Lindenstrauss--Weiss}
by introducing the notion of \textit{metric mean dimension}.
This is a dynamical analogue of Minkowski dimension defined as follows.  
Let $(\mathcal{X},T)$ be a dynamical system with a metric $d$.
Let $S(\mathcal{X},T,d,\varepsilon)$ be its entropy detected at the resolution $\varepsilon>0$.
(See \S \ref{subsection: topological and metric mean dimensions} for the precise definition.)     
The topological entropy is given by $h_{\mathrm{top}}(T) = \lim_{\varepsilon \to 0} S(\mathcal{X},T,d,\varepsilon)$.       
We define the \textbf{upper/lower metric mean dimensions} by 
\begin{equation}  \label{eq: definition of metric mean dimension}
 \begin{split}
   & \overline{\mdim}_{\mathrm{M}}(\mathcal{X},T,d) 
    = \limsup_{\varepsilon \to 0} \frac{S(\mathcal{X},T,d,\varepsilon)}{\log(1/\varepsilon)}, \\
   &  \underline{\mdim}_\mathrm{M} (\mathcal{X}, T, d) 
    = \liminf_{\varepsilon \to 0} \frac{S(\mathcal{X}, T, d, \varepsilon)}{\log(1/\varepsilon)}.
 \end{split} 
\end{equation}            
When the upper and lower limits coincide, we denote their common value by $\mdim_{\mathrm{M}}(\mathcal{X},T,d)$.
In analogy with the well-known fact that Minkowski dimension bounds topological dimension, we have
\cite[Theorem 4.2]{Lindenstrauss--Weiss}
\begin{equation}  \label{eq: metric mean dimension bounds mean dimension}
    \mdim(\mathcal{X}, T) \leq  \underline{\mdim}_\mathrm{M} (\mathcal{X}, T, d)   \leq  \overline{\mdim}_{\mathrm{M}}(\mathcal{X},T,d).
\end{equation}    
It was also proved in \cite[Theorem 4.3]{Lindenstrauss} that
if $(\mathcal{X},T)$ has the marker property then
there exists a metric $d$ on $\mathcal{X}$ satisfying $\mdim(\mathcal{X}, T) = \underline{\mdim}_\mathrm{M}(\mathcal{X},T,d)$.
The same statement for upper metric mean dimension remained open since \cite{Lindenstrauss}.
We will establish it as a part of the proof of Theorem \ref{main theorem}. See Theorem \ref{theorem: existence of nice metrics} below.

Metric mean dimension seems to be quite useful. In particular,
it provides a powerful method to obtain upper bounds on mean dimension via (\ref{eq: metric mean dimension bounds mean dimension}).
This was used for example in \cite{Tsukamoto Brody} for solving a problem of Gromov \cite{Gromov} to estimate the mean dimension of 
a dynamical system in holomorphic curve theory.
It also has an application to the study of expansive group actions \cite{Meyerovitch--Tsukamoto}.

It seems desirable to inject ergodic theory and in particular invariant measures 
into mean dimension theory in order to broaden the scope of applications.
This motivated the authors to begin the study of our previous paper \cite{Lindenstrauss--Tsukamoto rate distortion}.
In \cite{Lindenstrauss--Tsukamoto rate distortion} we proved the following variational principle between 
metric mean dimension and rate distortion function
under a mild condition on $d$ (called
\textit{tame growth of covering numbers}; see Definition \ref{definition: tame growth of covering numbers}):
\begin{equation}   \label{eq: rate distortion variational principle}
   \begin{split}
    & \overline{\mdim}_{\mathrm{M}}(\mathcal{X}, T,d)
      = \limsup_{\varepsilon \to 0} \frac{\sup_{\mu\in \mathscr{M}^T(\mathcal{X})} R(d,\mu,\varepsilon)}{\log(1/\varepsilon)},  \\
    &  \underline{\mdim}_{\mathrm{M}}(\mathcal{X}, T, d) 
      =  \liminf_{\varepsilon \to 0} \frac{\sup_{\mu\in \mathscr{M}^T(\mathcal{X})} R(d,\mu,\varepsilon)}{\log(1/\varepsilon)}.   
   \end{split}
\end{equation}
We proved this by 
developing a rate distortion theory version of Misiurewicz's proof \cite{Misiurewicz}
of the standard variational principle (\ref{eq: variational principle}).
This is an initial step 
of our program to inject measure into mean dimension theory.
However it is still not completely satisfactory.
The equation (\ref{eq: rate distortion variational principle}) implies that we can construct $\mu\in \mathcal{M}^T(\mathcal{X})$
capturing (most of) dynamical complexity of $(\mathcal{X},T)$ at each \textit{fixed} resolution $\varepsilon >0$.
It would be nicer if we could find $\mu$ capturing the dynamical complexity over \textit{all} resolutions.
In other words, we would like to exchange the order of the limit and supremum in (\ref{eq: rate distortion variational principle}).
This naturally leads us to the following question
(this was also posed by Velozo--Velozo \cite[Section 6]{Velozo--Velozo}):

\begin{problem}  \label{problem: metric mean dimension is equal to rate distortion dimension}
  When do the following equalities hold?
  \begin{equation}  \label{eq: metric mean dimension is equal to rate distortion dimension}
      \begin{split}
       & \overline{\mdim}_{\mathrm{M}}(\mathcal{X},T,d) = \sup_{\mu\in \mathscr{M}^T(\mathcal{X})} \overline{\rdim}(\mathcal{X},T,d,\mu), \\
       &  \underline{\mdim}_{\mathrm{M}}(\mathcal{X},T,d) = \sup_{\mu \in \mathscr{M}^T(\mathcal{X})} \underline{\rdim}(\mathcal{X},T,d,\mu).
      \end{split}
  \end{equation}    
\end{problem}        

Of course, (\ref{eq:  metric mean dimension is equal to rate distortion dimension}) does not hold in general.
The following example clarifies the situation:

\begin{example}
Let $A = \{1,1/2,1/3, \dots\}\cup \{0\} \subset [0,1]$ and $\mathcal{X} = A^\mathbb{Z}$ with the shift $\sigma$.        
Define a metric $d$ on $\mathcal{X}$ by $d(x,y) = \sum_{n\in \mathbb{Z}} 2^{-|n|} |x_n-y_n|$.         
 Then it is straightforward to check that $\mdim_{\mathrm{M}}(\mathcal{X},\sigma,d)  = 1/2$ 
     and that $\rdim(\mathcal{X},\sigma,d,\mu)  = 0$ for any $\sigma$-invariant probability measure $\mu$
     (cf. \cite[Lemma 3.1]{Kawabata--Dembo}).
So (\ref{eq: metric mean dimension is equal to rate distortion dimension}) does not hold even for this simple example.
However we can push our consideration further.
Let $B=\{1,2^{-1},2^{-2},\dots\}\cup \{0\}$ and consider a homeomorphism $f:A\to B$ defined by 
$f(1/n) = 2^{-n}$ and $f(0)=0$.
Define a new metric $d'$ on $\mathcal{X}=A^\mathbb{Z}$ by 
$d'(x,y) = \sum_{n\in \mathbb{Z}} 2^{-n} |f(x_n)-f(y_n)|$.
Then we can check that
\[  \mdim_{\mathrm{M}}(\mathcal{X},\sigma,d') =
    \sup_{\mu\in\mathscr{M}^T(\mathcal{X})} \rdim(\mathcal{X},\sigma,d', \mu)
  = 0. \]
In particular (\ref{eq: metric mean dimension is equal to rate distortion dimension}) holds true for $d'$.
\end{example}

The above example shows that we can expect the equality (\ref{eq: metric mean dimension is equal to rate distortion dimension})
only for well-chosen metrics $d$.
This suggests a new viewpoint: 
\textit{We cannot stick to a fixed metric $d$.
We should regard $d$ as a variable and move both $d$ and $\mu$.}
The double variational principle (Theorem \ref{main theorem}) is a crystallization of this idea.

The proof of Theorem \ref{main theorem} also provides a partial answer to Problem 
\ref{problem: metric mean dimension is equal to rate distortion dimension} (see Corollary \ref{cor: all the dimensions coincide}):
If $(\mathcal{X},T)$ has the marker property, then there exists a metric $d$ on $\mathcal{X}$
  such that all the following quantities are equal to each other:
  \begin{equation*}
     \begin{split} 
       &\mdim(\mathcal{X},T), \quad
        \overline{\mdim}_{\mathrm{M}}(\mathcal{X},T,d) , \quad \underline{\mdim}_{\mathrm{M}}(\mathcal{X},T,d), \\
       & \sup_{\mu\in \mathscr{M}^T(\mathcal{X})} \overline{\rdim}(\mathcal{X},T,d,\mu), \quad 
        \sup_{\mu\in \mathscr{M}^T(\mathcal{X})}   \underline{\rdim}(\mathcal{X},T,d,\mu).      
     \end{split}
  \end{equation*}

\subsection{Outline of the proof of Theorem \ref{main theorem}} \label{subsection: outline of the proof of the main theorem}

Let $(\mathcal{X},T)$ be a dynamical system.
The proof of Theorem \ref{main theorem} consists of the following three steps:

\vspace{0.2cm}

\textbf{Step 1 (Metric mean dimension bounds rate distortion dimension):}
\textit{For all $d\in \mathscr{D}(\mathcal{X})$ and $\mu\in \mathscr{M}^T(\mathcal{X})$}
\begin{equation*}
   \begin{split}
     \overline{\rdim}(\mathcal{X},T,d,\mu) \leq \overline{\mdim}_{\mathrm{M}}(\mathcal{X},T,d), \\
     \underline{\rdim}(\mathcal{X},T,d,\mu) \leq \underline{\mdim}_{\mathrm{M}}(\mathcal{X},T,d). 
   \end{split}
\end{equation*}      

\textbf{Step 2 (Constructing invariant measures encoding dynamical complexity):}
\textit{For all $d\in \mathscr{D}(\mathcal{X})$}
\[  \mdim(\mathcal{X},T) \leq \sup_{\mu\in \mathscr{M}^T(\mathcal{X})} \underline{\rdim}(\mathcal{X},T,d,\mu). \]

\textbf{Step 3 (Constructing nice metrics):}
\textit{Under the marker property assumption 
\[ \exists d\in \mathscr{D}(\mathcal{X}): \quad 
   \mdim(\mathcal{X},T) = \overline{\mdim}_{\mathrm{M}}(\mathcal{X},T,d). \]
We emphasize that the marker property is used only in this step.}

\vspace{0.2cm}

Combining the above three steps, we get Theorem \ref{main theorem}.
Step 1 is easy to prove. (See \S \ref{subsection: comparison between various dynamical dimensions}.)
So the main issues are Steps 2 and 3.

\vspace{0.2cm}

\textbf{About Step 2:}
Let $d$ be a metric on $\mathcal{X}$.
In the proof of Step 2, we introduce a new notion called \textit{mean Hausdorff dimension} (denoted by 
$\mathrm{\mdim}_{\mathrm{H}}(\mathcal{X},T,d)$).
This is a dynamical version of Hausdorff dimension.
As is well known in geometric measure theory, Hausdorff dimension is more closely related to measure theory than 
Minkowski dimension. So it is natural to expect that its dynamical analogue is helpful to connect measure theory to
mean dimension\footnote{The idea of introducing mean Hausdorff dimension was partly motivated by the study of 
Kawabata--Dembo \cite[Proposition 3.2]{Kawabata--Dembo}. 
Roughly speaking, their result \cite[Proposition 3.2]{Kawabata--Dembo} corresponds to Step 2.2 for 
$(\mathcal{X}, T) = (A^\mathbb{Z},\mathrm{shift})$ with $A\subset \mathbb{R}^n$.
In other words, Step 2.2 is a generalization of their result to arbitrary dynamical systems.}.
We decompose Step 2 into two smaller steps\footnote{There also exists a small issue about the tame growth of covering numbers condition.
But we ignore it here}:

\vspace{0.2cm}

\textbf{Step 2.1 (Mean Hausdorff dimension bounds mean dimension):}
\[  \mdim(\mathcal{X},T) \leq \mdim_{\mathrm{H}}(\mathcal{X},T,d). \]

\textbf{Step 2.2 (Dynamical analogue of Frostman's lemma):}
\textit{Under a mild condition on $d$ (the tame growth of covering numbers; see Definition \ref{definition: tame growth of covering numbers})}
\[  \mdim_{\mathrm{H}}(\mathcal{X},T,d) \leq \sup_{\mu\in \mathscr{M}^T(\mathcal{X})} \underline{\rdim}(\mathcal{X},T,d,\mu). \]

\vspace{0.2cm}

Step 2.1 is a dynamical analogue of the fact that Hausdorff dimension bounds topological dimension.
Its proof is given in \S \ref{subsection: comparison between various dynamical dimensions}.
Step 2.2 is the main part of Step 2.
Frostman's lemma is a classical result in geometric measure theory.
Roughly speaking, it claims that we can construct a probability measure which obeys the \textit{scaling law}
corresponding to the Hausdorff dimension.
We establish Step 2.2 by combining Frostman's lemma with the techniques of our previous variational 
principle (\ref{eq: rate distortion variational principle}).
It roughly goes as follows.
For $n\geq 1$ we set $d_n(x,y) = \max_{0\leq k< n} d(T^k x, T^k y)$.
By using the geometric measure theory around Frostman's lemma, for each $n\geq 1$, 
we construct a (non-invariant) probability measure $\nu_n$
on $\mathcal{X}$ which captures the \textit{geometric} complexity of $(\mathcal{X},d_n)$ over all resolutions.
Consider 
\[  \mu_n  = \frac{1}{n}  \sum_{k=0}^{n-1} T^k_* \nu_n. \]
From the compactness we can choose a subsequence $\mu_{n_k}$
which converges to some invariant probability measure (say, $\mu$).
We apply to $\mu$ the rate distortion theory version of Misiurewicz's technique \cite{Misiurewicz} (developed in 
\cite{Lindenstrauss--Tsukamoto rate distortion}) and prove that 
$\mu$ captures most of the \textit{dynamical} complexity of $(\mathcal{X},T)$ over all resolutions.

\vspace{0.2cm}

\textbf{About Step 3:}
Step 3 is technically hard.
As we briefly noted in \S \ref{subsection: backgrounds}, it was already proved in \cite[Theorem 4.3]{Lindenstrauss}
that if $(\mathcal{X},T)$ has the marker property then 
\begin{equation} \label{eq: case of lower metric mean dimension: outline of the proof}
  \exists d\in \mathscr{D}(\mathcal{X}): \quad 
  \mdim(\mathcal{X},T) = \underline{\mdim}_{\mathrm{M}}(\mathcal{X},T,d).
\end{equation}
The claim of Step 3 looks very similar.
But, in fact, it is much subtler and remained to be an open problem for about 20 years since \cite{Lindenstrauss}.
It is difficult to briefly explain the ideas of the proof. (See \S \ref{subsection: background: PS theorem} for more background.)
Here we just remark that the above (\ref{eq: case of lower metric mean dimension: outline of the proof})
(with Steps 1 and 2)
are already enough for proving the equality for the lower rate distortion dimension:
\begin{equation} \label{eq: lower case of double variational principle}
   \mdim(\mathcal{X},T) = \min_{d\in \mathscr{D}(\mathcal{X})} \sup_{\mu\in \mathscr{M}^T(\mathcal{X})} 
                                  \underline{\rdim}(\mathcal{X},T,d,\mu).
\end{equation}

\subsection{Open problems and future directions} \label{subsection: open problems and future directions}

The most important open problem is to remove the marker property assumption in Theorem \ref{main theorem}:

\begin{problem} \label{problem: removing the marker property in double variational principle}
Prove the double variational principle (\ref{eq: double variational principle}) for 
all dynamical systems.
\end{problem}

As we explained in \S \ref{subsection: outline of the proof of the main theorem}, the marker property 
is used only in Step 3 of the proof of Theorem \ref{main theorem}.
So Problem \ref{problem: removing the marker property in double variational principle} reduces to 

\begin{problem}  \label{problem: removing the marker property in Step 3}
Prove that for any dynamical system $(\mathcal{X},T)$
\[ \exists d\in \mathscr{D}(\mathcal{X}): \quad 
  \mdim(\mathcal{X},T) = \overline{\mdim}_{\mathrm{M}}(\mathcal{X},T,d). \]
  We emphasize that the same problem for lower metric mean dimension is also open.
\end{problem}

Problems \ref{problem: removing the marker property in double variational principle} and 
\ref{problem: removing the marker property in Step 3} are certainly the central open problems.
But there also exists a different interesting direction.
Step 2 of the proof of Theorem \ref{main theorem} does not use the marker property assumption.
So we always have the inequality 
\[  \mdim(\mathcal{X},T) \leq \inf_{d\in \mathscr{D}(\mathcal{X})} \sup_{\mu\in \mathscr{M}^T(\mathcal{X})}
     \underline{\rdim}(\mathcal{X},T,d,\mu), \]
although we don't know whether the equality holds or not.
This implies that we can always find a ``sufficiently rich'' invariant measure $\mu$.
Study of these measures for concrete examples seems very interesting.
We begin such study in \S \ref{section: example: algebraic actions}.
Although our investigation in this direction has just started,
the result in \S \ref{section: example: algebraic actions}
seems to suggest a high potential of this research direction.
It is desirable to study geometric examples in \cite{Gromov, Tsukamoto Brody, Tsukamoto YG}
from the viewpoint of the double variational principle.

\subsection{Organization of the paper and how to read it}

\S \ref{section: preliminaries}
is a preparation of basics of mean dimension and rate distortion function.
In \S \ref{section: mean Hausdorff dimension} we introduce mean Hausdorff dimension and establish 
Step 1 and Step 2.1 of the proof of Theorem \ref{main theorem}.
In \S \ref{section: proof of existence of nice measures}
we prepare some basics of geometric measure theory and establish Step 2.2.
We establish Step 3 and complete the proof of Theorem \ref{main theorem}
in \S \ref{section: proof of existence of nice metrics}.
We study a concrete example in \S \ref{section: example: algebraic actions}.
Although the result in \S \ref{section: example: algebraic actions} is not used in the proof of Theorem \ref{main theorem}, 
hopefully it will help readers to understand various concepts in the paper.

This paper is rather lengthy. We would like to suggest readers how to read it.
\S \ref{section: proof of existence of nice metrics} is technically hard.
So it may be reasonable to concentrate on \S \ref{section: mean Hausdorff dimension} and \S \ref{section: proof of existence of nice measures}
at the first reading.
\S \ref{section: preliminaries} is a preparation for these two sections.
So, after reading only the main definitions in \S \ref{section: preliminaries} 
(topological/metric mean dimensions, mutual information and rate distortion function),
readers may skip to \S \ref{section: mean Hausdorff dimension}
and return to \S \ref{section: preliminaries} when they need the results there.
\S \ref{section: example: algebraic actions}
might help readers to improve the understanding.
So it may be nice to briefly look at it in the midst of reading \S \ref{section: mean Hausdorff dimension} and 
\S \ref{section: proof of existence of nice measures}.

\subsection*{Acknowledgment} This project was initiated at the Banff International Research Station meeting ``Mean Dimension and Sofic Entropy Meet Dynamical Systems, Geometric Analysis and Information Theory'' in 2017. We thank BIRS for hosting this workshop, and for providing ideal conditions for collaborations.

\section{Preliminaries}  \label{section: preliminaries}

\subsection{Topological and metric mean dimensions} \label{subsection: topological and metric mean dimensions}

We review basics of topological and metric mean dimensions in this subsection 
\cite{Gromov, Lindenstrauss--Weiss}.
Throughout this paper we assume that all simplicial complexes are finite (i.e. they have only 
finitely many faces).

Let $(\mathcal{X},d)$ be a compact metric space.
We introduce some metric invariants of $(\mathcal{X},d)$.
Take a positive number $\varepsilon$.  Let $f:\mathcal{X}\to \mathcal{Y}$ be a continuous map 
from $\mathcal{X}$ to some topological space $\mathcal{Y}$.
The map $f$ is said to be an \textbf{$\varepsilon$-embedding}
if $\diam f^{-1}(y) < \varepsilon$ for every $y\in \mathcal{Y}$.
We define the \textbf{$\varepsilon$-width dimension} $\widim_\varepsilon (\mathcal{X},d)$ 
as the minimum $n\geq 0$ such that there exists an $\varepsilon$-embedding 
$f:\mathcal{X}\to P$ from $\mathcal{X}$ to some $n$-dimensional simplicial complex $P$.
The \textbf{topological dimension} of $\mathcal{X}$ is given by 
$\dim \mathcal{X} = \lim_{\varepsilon\to 0} \widim_\varepsilon(\mathcal{X},d)$. 

We define the \textbf{$\varepsilon$-covering number} $\#(\mathcal{X},d,\varepsilon)$
as the minimum $n\geq 1$ such that there exists an open cover $\{U_1,\dots, U_n\}$ of $\mathcal{X}$
satisfying $\diam \, U_i < \varepsilon$ for all $1\leq i\leq n$.
We also define the \textbf{$\varepsilon$-separating number} $\#_{\mathrm{sep}}(\mathcal{X},d,\varepsilon)$
as the maximum $n\geq 1$ such that there exist $x_1,\dots, x_n\in \mathcal{X}$ satisfying 
$d(x_i,x_j) \geq \varepsilon$ for all $i\neq j$.
For $0<\delta<\varepsilon/2$
\begin{equation}   \label{eq: covering and separating numbers}
   \#_{\mathrm{sep}}(\mathcal{X},d,\varepsilon)  \leq \#(\mathcal{X},d,\varepsilon)
    \leq \#_{\mathrm{sep}}(\mathcal{X},d,\delta).
\end{equation}
The \textbf{upper and lower Minkowski dimensions (or box dimensions)} of $(\mathcal{X},d)$ are given by 
\begin{equation*}
    \begin{split}
     \overline{\dim}_{\mathrm{M}} (\mathcal{X},d) &= \limsup_{\varepsilon \to 0} \frac{\log \#(\mathcal{X},d,\varepsilon)}{\log (1/\varepsilon)},\\
     \underline{\dim}_{\mathrm{M}}(\mathcal{X},d) &= \liminf_{\varepsilon\to 0} \frac{\log \#(\mathcal{X},d,\varepsilon)}{\log (1/\varepsilon)}. 
    \end{split}
\end{equation*}

\begin{example} \label{example: widim and covering number of ball}
   Let $(V,\norm{\cdot})$ be a finite dimensional Banach space and $B_r(V)$ the closed $r$-ball around the origin ($r>0$).
   Then for $0<\varepsilon <r$
   \begin{equation} \label{eq: widim of ball}
      \widim_\varepsilon (B_r(V), \norm{\cdot}) = \dim V,   
   \end{equation}
   \begin{equation}  \label{eq: covering number of ball}
      \#(B_r(V),\norm{\cdot},\varepsilon) \geq (r/\varepsilon)^{\dim V}.
   \end{equation}
  (\ref{eq: widim of ball}) is due to Gromov \cite[\S 1.1.2]{Gromov}.
  See \cite[Appendix]{Tsukamoto deformation} for a simple proof.
  The proof of (\ref{eq: covering number of ball}) is easy: Take the Lebesgue measure $\mu$ on $V$ normalized by $\mu(B_r(V))=1$.
  Let $B_r(V) = U_1\cup \dots\cup U_n$ with $\diam\, U_i < \varepsilon$.
  Pick $x_i\in U_i$. Then $B_r(V) \subset B_\varepsilon (x_1)\cup\dots \cup B_\varepsilon (x_n)$ 
  ($B_\varepsilon(x_i)$ is the closed $\varepsilon$-ball 
  centered at $x_i$). It follows that 
  \[ 1 = \mu(B_r(V)) \leq \sum_{i=1}^n \mu(B_\varepsilon(x_i)) = n (\varepsilon/r)^{\dim V}. \]
   This shows $n\geq (r/\varepsilon)^{\dim V}$.
\end{example}

Let $(\mathcal{X},T)$ be a dynamical system with a metric $d$.
For $N\geq 1$ we define a new metric on $\mathcal{X}$ by 
\[  d_N(x,y) = \max_{0\leq n < N} d(T^n x, T^n y). \]
We define the (topological) \textbf{mean dimension} by 
\[ \mdim(\mathcal{X},T) = \lim_{\varepsilon \to 0} \left(\lim_{N\to \infty} \frac{\widim_\varepsilon(\mathcal{X},d_N)}{N}\right). \]
The limit always exists because $\widim_\varepsilon (\mathcal{X}, d_N)$ is subadditive in $N$. 
The value of $\mdim(\mathcal{X},T)$ is independent of the choice of $d$, namely it becomes a topological invariant of 
$(\mathcal{X},T)$.
We define the \textbf{entropy at the resolution $\varepsilon>0$} by 
\[ S(\mathcal{X}, T, d, \varepsilon) = \lim_{N\to \infty} \frac{\log \#(\mathcal{X},d_N,\varepsilon)}{N}, \]
where the limit exists because $\log \#(\mathcal{X},d_N,\varepsilon)$ is subadditive in $N$.
We define the upper and lower metric mean dimensions by (\ref{eq: definition of metric mean dimension})
in \S \ref{subsection: backgrounds}.

The following two theorems were proved in \cite[Theorem 4.2]{Lindenstrauss--Weiss} and \cite[Theorem 4.3]{Lindenstrauss}
respectively.

\begin{theorem} \label{theorem: metric mean dimension dominates mean dimension}
\[  \mdim(\mathcal{X},T) \leq  \underline{\mdim}_\mathrm{M} (\mathcal{X}, T, d) \leq  \overline{\mdim}_{\mathrm{M}}(\mathcal{X},T,d).\]
\end{theorem}

\begin{theorem} \label{theorem: lower metric mean dimension coincides with mean dimension}
If $(\mathcal{X},T)$ has the marker property then 
there exists a metric $d$ on $\mathcal{X}$ compatible with the topology satisfying 
\[    \underline{\mdim}_{\mathrm{M}} (\mathcal{X},T, d) = \mdim(\mathcal{X},T). \]
\end{theorem}

\begin{example}  \label{example: mean dimension of Hilbert cube}
 Let $\sigma: [0,1]^\mathbb{Z}\to [0,1]^\mathbb{Z}$ be the shift on the alphabet $[0,1]$ (the unit interval).
 We define a metric $d$ on it by $d(x,y) = \sum_{n\in \mathbb{Z}} 2^{-|n|} |x_n-y_n|$.
 Then 
 \[ \mdim\left([0,1]^\mathbb{Z},\sigma\right) = \mdim_{\mathrm{M}}\left([0,1]^\mathbb{Z},\sigma, d\right) = 1.\]
 The only nontrivial point is the lower bound $\mdim\left([0,1]^\mathbb{Z},\sigma\right) \geq 1$, 
 which follows from \eqref{eq: widim of ball}; cf.\ also \cite[Proposition 3.3]{Lindenstrauss--Weiss}. 
\end{example}

\subsection{Mutual information}  \label{subsection: mutual information}

Here we prepare some basics of mutual information 
\cite[Chapter 2]{Cover--Thomas}.
Throughout this subsection we fix a probability space $(\Omega, \mathbb{P})$ and assume that 
all random variables are defined on it.

Let $X$ and $Y$ be two random variables taking values in some measurable spaces $\mathcal{X}$ and $\mathcal{Y}$
respectively.
We want to define their \textbf{mutual information} $I(X;Y)$, which measures the amount of information shared by 
both $X$ and $Y$.
If $\mathcal{X}$ and $\mathcal{Y}$ are finite sets\footnote{We always assume that the $\sigma$-algebra of 
a finite set is the largest one (the set of all subsets).}, then we set 
\begin{equation} \label{eq: mutual information}
    I(X;Y) = H(X)+H(Y) -H(X,Y) = H(X) - H(X|Y),
\end{equation}    
where $H(X|Y)$ is the conditional entropy of $X$ given $Y$.
With the convention that $0\log(0/a)= 0$ for all $a\geq 0$, we can also write this as 
\begin{equation}  \label{eq: mutual information 2}
   I(X;Y)    = \sum_{x\in \mathcal{X}, y\in \mathcal{Y}} \mathbb{P}(X=x,Y=y) 
           \log \frac{\mathbb{P}(X=x, Y=y)}{\mathbb{P}(X=x) \mathbb{P}(Y=y)}.
\end{equation}
In general we proceed as follows.
Take finite measurable partitions $\mathcal{P} = \{P_1,\dots, P_M\}$ and 
$\mathcal{Q} = \{Q_1,\dots, Q_N\}$ of $\mathcal{X}$ and $\mathcal{Y}$ respectively.
For $x\in \mathcal{X}$ and $y\in \mathcal{Y}$ we set 
$\mathcal{P}(x) = P_m$ and $\mathcal{Q}(y) = Q_n$ where $x\in P_m$ and $y\in Q_n$
Then we can consider the mutual information $I(\mathcal{P}\circ X; \mathcal{Q}\circ Y)$
defined by (\ref{eq: mutual information}) because $\mathcal{P}\circ X$ and $\mathcal{Q}\circ Y$
take only finitely many values.
We define $I(X;Y)$ as the supremum of $I(\mathcal{P}\circ X; \mathcal{Q}\circ Y)$
over all finite measurable partitions $\mathcal{P}$ and $\mathcal{Q}$ of $\mathcal{X}$ and $\mathcal{Y}$.
This definition is compatible with (\ref{eq: mutual information}) when $\mathcal{X}$ and $\mathcal{Y}$ are finite sets
\footnote{We can show this by proving the data-processing inequality (Lemma \ref{lemma: data-processing inequality})
for the quantity defined by (\ref{eq: mutual information}) in the case that $\mathcal{X}$ and $\mathcal{Y}$ are finite sets.
See \cite[Section 2.8]{Cover--Thomas}.}.

We gather properties of mutual information required in the proof of the double variational principle
(Theorem \ref{main theorem}) below.
They are not used in \S \ref{section: mean Hausdorff dimension}.
So readers may postpone to read the rest of this subsection until they come to \S \ref{section: proof of existence of nice measures}.

\begin{lemma}[Data-Processing inequality]  \label{lemma: data-processing inequality}
Let $X$ and $Y$ be random variables taking values in measurable spaces $\mathcal{X}$ and $\mathcal{Y}$.
If $f:\mathcal{Y}\to \mathcal{Z}$ is a measurable map, then $I(X; f(Y)) \leq I(X;Y)$.
\end{lemma}

\begin{proof}
This immediately follows from the definition.
A nontrivial point is that the above definition is compatible with (\ref{eq: mutual information}) for discrete random 
variables.
\end{proof}

\begin{lemma}  \label{lemma: law convergence implies the convergence of mutual information}
Let $\mathcal{X}$ and $\mathcal{Y}$ be finite sets and $(X_n,Y_n)$ a sequence of random variables taking values in 
$\mathcal{X}\times \mathcal{Y}$. 
If $(X_n,Y_n)$ converges to some $(X,Y)$ in law, then $I(X_n;Y_n)$ converges to $I(X;Y)$.
\end{lemma}

\begin{proof}
This follows from (\ref{eq: mutual information}).
\end{proof}

\begin{lemma}[Subadditivity of mutual information] \label{lemma: subadditivity of mutual information}
Let $X,Y,Z$ be random variables taking values in finite sets $\mathcal{X},\mathcal{Y},\mathcal{Z}$ respectively.
Suppose $X$ and $Y$ are conditionally independent given $Z$, namely for every $x\in \mathcal{X}$, $y\in \mathcal{Y}$ and
$z\in \mathcal{Z}$ with $\mathbb{P}(Z=z)\neq 0$ we have 
\[  \mathbb{P}(X=x, Y=y|Z=z) = \mathbb{P}(X=x|Z=z) \mathbb{P}(Y=y|Z=z). \]
Then $I(X, Y; Z) \leq I(X;Z) + I(Y;Z)$.
\end{lemma}

\begin{proof}
$I(X,Y;Z) = H(X,Y)-H(X,Y|Z)$. From the conditional independence 
$H(X,Y|Z) = H(X|Z) + H(Y|Z)$. Hence 
\begin{equation*}
   \begin{split}
     I(X,Y;Z) & = H(X,Y)-H(X|Z) - H(Y|Z) \\
               &\leq H(X)+H(Y) - H(X|Z)-H(Y|Z) \\
               &=   I(X;Z) + I(Y;Z).   
   \end{split}
\end{equation*}   
Here we have used $H(X,Y) \leq H(X)+H(Y)$.
\end{proof}

Let $X$ and $Y$ be random variables taking values in finite sets $\mathcal{X}$ and $\mathcal{Y}$.
We set $\mu(x) = \mathbb{P}(X=x)$ and $\nu(y|x) = \mathbb{P}(Y=y|X=x)$ for 
$x\in \mathcal{X}$ and $y\in \mathcal{Y}$.
(The conditional probability mass function $\nu(y|x)$ is defined only for $x\in \mathcal{X}$ with $\mathbb{P}(X=x)\neq 0$.)
The mutual information $I(X;Y)$ is determined by the distribution of $(X,Y)$, which is given by $\mu(x)\nu(y|x)$.
It will be convenient for us to write $I(X;Y)$ sometimes as $I(\mu, \nu)$.

\begin{lemma}[Concavity/convexity of mutual information] \label{lemma: concavity/convexity of mutual information}
In the above setting
$I(\mu,\nu)$ is a concave function of $\mu(x)$ for fixed $\nu(y|x)$ and a convex function of $\nu(y|x)$ for fixed $\mu(x)$.
Namely for $0\leq t\leq 1$
\begin{equation*}
    \begin{split}
      I\left((1-t)\mu_1+t\mu_2, \nu\right) &\geq (1-t) I(\mu_1,\nu) + t I(\mu_2,\nu), \\
     I\left(\mu, (1-t)\nu_1+t\nu_2\right)   &\leq (1-t) I(\mu,\nu_1) + t I(\mu,\nu_2). 
    \end{split}
\end{equation*}    
\end{lemma}

\begin{proof}
See \cite[Theorem 2.7.4]{Cover--Thomas} for the detailed proof.
First we prove the concavity.
\[ I(\mu,\nu) = I(X;Y) = H(Y) -H(Y|X). \]
$H(Y)$ is a concave function of $\mu(x)$ for fixed $\nu(y|x)$
(since the Shannon entropy is a concave function of distribution) and $H(Y|X)$ is a linear function of $\mu(x)$.
So $I(\mu,\nu)$ is a concave function of $\mu(x)$.

Next we prove the convexity.
From the convexity of $\phi(t) := t\log t$ 
\[ \phi\left(\frac{a+a'}{b+b'}\right) \leq \frac{b}{b+b'}\phi\left(\frac{a}{b}\right) + \frac{b'}{b+b'}\phi\left(\frac{a'}{b'}\right) \] 
for positive $a,a',b,b'$. This leads to the \textit{log sum inequality}:
\begin{equation}  \label{eq: log sum inequality}
   (a+a') \log \frac{a+a'}{b+b'} \leq a\log \frac{a}{b}  + a' \log \frac{a'}{b'}. 
\end{equation} 
Set $\sigma_i (y) = \sum_{x\in \mathcal{X}}\mu(x)\nu_i(y|x)$ for $i=1,2$.
\begin{equation*}
   \begin{split}
  I\left(\mu, (1-t)\nu_1+ t\nu_2\right) =   \sum_{x,y} &\left\{(1-t)\mu(x)\nu_1(y|x) + t\mu(x) \nu_2(y|x)\right\}  \\
     &\times \log \frac{(1-t)\mu(x)\nu_1(y|x) + t\mu(x) \nu_2(y|x)}{(1-t)\mu(x)\sigma_1(y) + t\mu(x)\sigma_2(y)}. 
   \end{split}
\end{equation*}   
Apply (\ref{eq: log sum inequality}) to each summand: This is bounded by
\[ \sum_{x,y} (1-t)\mu(x)\nu_1(y|x) \log \frac{\mu(x)\nu_1(y|x)}{\mu(x)\sigma_1(y)}
   + \sum_{x,y} t\mu(x)\nu_2(y|x) \log\frac{\mu(x)\nu_2(y|x)}{\mu(x)\sigma_2(y)},   \]
  which is equal to $(1-t) I(\mu,\nu_1) + t I(\mu,\nu_2)$.
\end{proof}

We borrow the next lemma from \cite[Lemma A.1]{Kawabata--Dembo}.
This is a duality of convex programming.
(See Section 2.5 of \cite{Berger}, specifically \cite[Theorem 2.5.3]{Berger}, for further information.)
Recall that the base of the logarithm is two and the natural logarithm is written as $\ln(\cdot)$.

\begin{lemma}   \label{lemma: duality in rate distortion theory}
Let $\mathcal{X}$ and $\mathcal{Y}$ be compact metric spaces and 
$\rho:\mathcal{X}\times \mathcal{Y}\to [0,\infty)$ 
a continuous function.
Let $\mu$ be a Borel probability measure on $\mathcal{X}$, $\varepsilon >0$ and $a\geq 0$ real numbers.
Suppose a continuous function\footnote{The continuity of $\rho$ and $\lambda$ is inessential. 
But we assume it for simplicity.
Indeed in our applications, $\mathcal{X}= \mathcal{Y}$, 
$\rho$ is a distance function and $\lambda$ is a constant. } $\lambda:\mathcal{X}\to [0,\infty)$ satisfies
\begin{equation} \label{eq: assumption in duality in rate distortion theory}
    \forall y\in \mathcal{Y}:  \quad \int_{\mathcal{X}} \lambda(x)\, 2^{-a\rho(x,y)} d\mu(x) \leq 1. 
\end{equation}    
If $X$ and $Y$ are random variables taking values in $\mathcal{X}$ and $\mathcal{Y}$ respectively and satisfying 
$\mathrm{Law}(X) = \mu$ and $\mathbb{E} \rho(X,Y) < \varepsilon$ then 
\begin{equation} \label{eq: duality in rate distortion theory}
      I(X;Y) \geq -a \varepsilon + \int_{\mathcal{X}} \log \lambda(x) d\mu(x). 
\end{equation}      
\end{lemma}

\begin{proof}
We divide the proof into two steps.

\textbf{Step 1: Assume $\mathcal{X}$ and $\mathcal{Y}$ are finite sets.}
Let $\nu= \mathrm{Law}(Y)$ be the distribution of $Y$.
We define a function $f(x,y)$ by $\mathbb{P}(X=x,Y=y) = f(x,y) \mathbb{P}(X=x)\mathbb{P}(Y=y)$.
(We do not need to define the value $f(x,y)$ if $\mathbb{P}(X=x)\mathbb{P}(Y=y)=0$.)
It follows from (\ref{eq: mutual information 2}) that
\begin{equation*}
   \begin{split}
    I(X;Y) &=   \sum_{x\in \mathcal{X}, y\in \mathcal{Y}}  \mathbb{P}(X=x) \mathbb{P}(Y=y) f(x,y) \log f(x,y)  \\
    & = \int_{\mathcal{X}\times \mathcal{Y}} f(x,y) \log f(x,y) d\mu(x) d\nu(y). 
   \end{split}
\end{equation*}    
Set $g(x,y) = \lambda(x)\, 2^{-a \rho(x,y)}$.
The right-hand side of (\ref{eq: duality in rate distortion theory}) is equal to 
\[  -a\varepsilon + \int_{\mathcal{X}\times \mathcal{Y}} \log\lambda\, d \mathrm{Law}(X,Y)
     =  -a\varepsilon + \int_{\mathcal{X}\times \mathcal{Y}} f(x,y) \log\lambda(x)  d\mu(x) d\nu(y). \] 
Since $-\varepsilon < -\mathbb{E}\rho(X,Y) = -\int_{\mathcal{X}\times \mathcal{Y}}\rho(x,y) f(x,y) d\mu(x)d\nu(y)$,
this is less than 
\[  \int_{\mathcal{X}\times \mathcal{Y}} f(x,y) \log g(x,y) d\mu(x) d\nu(y). \]
\[  I(X;Y) -  \int_{\mathcal{X}\times \mathcal{Y}} f(x,y) \log g(x,y) d\mu(x) d\nu(y)
     = \int_{\mathcal{X}\times \mathcal{Y}} f \log (f/g) d\mu d\nu. \]
As $\ln 2 \cdot \log(1/u) = \ln (1/u) \geq 1-u$, we have $\ln 2 \cdot f\log (f/g) \geq  f(1-g/f) = f-g$ and hence 
\begin{equation*}
   \begin{split}
        \ln 2 \cdot \int_{\mathcal{X}\times \mathcal{Y}} f \log (f/g) d\mu d\nu & 
         \geq \int_{\mathcal{X}\times \mathcal{Y}} (f-g)d\mu d\nu \\
        & = 1 - \int_{\mathcal{Y}}\left(\int_{\mathcal{X}} g(x,y) d\mu(x)\right) d\nu(y) \geq 0.
   \end{split}
\end{equation*}      
Here we have used the assumption (\ref{eq: assumption in duality in rate distortion theory}) in the last inequality.

\textbf{Step 2: General case.}
Let $\delta>0$.
Take finite partitions $\mathcal{P} = \{P_1,\dots, P_M\}$ and $\mathcal{Q} = \{Q_1,\dots,Q_N\}$ 
of $\mathcal{X}$ and $\mathcal{Y}$ respectively.
For each $P_m$ we take a point $x_m\in P_m$ satisfying $\lambda(x_m) \geq (1+\delta)^{-1} \sup_{P_m} \lambda$.
We pick arbitrary $y_n\in Q_n$ for each $Q_n$.
We set $\mathcal{X}' = \{x_1,\dots,x_M\}$ and $\mathcal{Y}' = \{y_1,\dots,y_N\}$ and define maps 
$\mathcal{P}:\mathcal{X}\to \mathcal{X}'$ and $\mathcal{Q}:\mathcal{Y}\to \mathcal{Y}'$ by 
$\mathcal{P}(P_m) = \{x_m\}$ and $\mathcal{Q}(Q_n)= \{y_n\}$.
Set $X' = \mathcal{P}\circ X$,$Y'= \mathcal{Q}\circ Y$ and 
$\mu' = \mathcal{P}_*\mu = \mathrm{Law}(X')$.

From the continuity of $\rho$ and $\lambda$, by taking $\mathcal{P}$ and $\mathcal{Q}$ sufficiently fine, we can assume 
$\mathbb{E}\rho(X',Y') < \varepsilon$ and
\[ \forall y\in \mathcal{Y'}, \quad \int_{\mathcal{X}'} \lambda(x) \, 2^{-a\rho(x,y)} d\mu'(x) \leq 1+\delta. \]
Then we can apply Step 1 to $X', Y'$ and the function $\lambda'(x_m) := (1+\delta)^{-1} \lambda(x_m)$. 
This yields
\[  I(X';Y') \geq -a\varepsilon + \int_{\mathcal{X'}} \log \lambda' d\mu'
     = -a\varepsilon -\log (1+\delta) + \int_{\mathcal{X}'} \log\lambda d\mu'. \]
It follows from the choice of $x_m$ that 
\[   \int_{\mathcal{X}'} \log \lambda d\mu' \geq - \log (1+\delta) + \int_{\mathcal{X}} \log \lambda d\mu. \]
As $I(X;Y) \geq I(X';Y')$ by the definition of mutual information,
\[ I(X;Y) \geq -a\varepsilon -2\log(1+\delta) + \int_{\mathcal{X}} \log \lambda d\mu. \]
Let $\delta\to 0$. This shows the statement.
\end{proof}

The next lemma is essentially due to \cite[Proposition 3.2]{Kawabata--Dembo}.
This is a key to connect geometric measure theory to rate distortion theory.

\begin{lemma}\label{lemma: from geometric measure theory to rate distortion theory}
Let $\varepsilon$ and $\delta$ be positive numbers with $2\varepsilon \log(1/\varepsilon) \leq \delta$.
Let $0\leq \tau \leq \min(\varepsilon/3, \delta/2)$ and $s\geq 0$ be real numbers.
Let $(\mathcal{X}, d)$ be a compact metric space with a Borel probability measure $\mu$ satisfying 
\begin{equation} \label{eq: scaling law in preliminary lemma}
    \mu(E) \leq (\tau + \diam E)^s, \quad \forall E\subset \mathcal{X} \text{ with } \diam E < \delta. 
\end{equation}    
Let $X$ and $Y$ be random variables taking values in $\mathcal{X}$ 
with $\mathrm{Law}(X) = \mu$ and $\mathbb{E} d(X,Y) < \varepsilon$.
Then 
\[  I(X;Y) \geq s\log(1/\varepsilon) - C(s+1), \]
where $C$ is a universal positive constant independent of $\varepsilon, \delta, \tau, s, (\mathcal{X},d), \mu$.
\end{lemma}

\begin{proof}
We apply Lemma \ref{lemma: duality in rate distortion theory} with $a = s/\varepsilon$.
Set $b=a\ln 2$ and 
estimate $\int_{\mathcal{X}} 2^{-ad(x,y)} d\mu(x) = \int_{\mathcal{X}} e^{-b d(x,y)} d\mu(x)$ for each $y\in \mathcal{X}$:
\begin{equation*}
   \begin{split}
     \int_{\mathcal{X}} e^{-b d(x,y)} d\mu(x) & = \int_0^1 \mu\{x|\, e^{-b d(x,y)} \geq u\} du \\
     & = \int_0^\infty \mu\{x|\, d(x,y) \leq v\} b e^{-b v} dv \quad (\text{set $u = e^{-b v}$}) \\
     & = \left(\int_0^\tau + \int_\tau^{\delta/2} + \int_{\delta/2}^\infty\right) \mu\{x|\, d(x,y) \leq v\} b e^{-bv} dv. 
   \end{split}
\end{equation*}
In the last line we have used $\tau \leq \delta/2$.
From (\ref{eq: scaling law in preliminary lemma})
\[    \int_0^\tau \mu\{x|\, d(x,y) \leq v\} b e^{-bv} dv \leq (3\tau)^s \int_0^\infty b e^{-bv} dv = (3\tau)^s
      \leq \varepsilon^s,  \]
where we have used $\tau \leq \varepsilon/3$.      
\begin{equation*}
   \begin{split}
    \int_\tau^{\delta/2} \mu\{x|\, d(x,y) \leq v\} b e^{-bv} dv & \leq \int_\tau^{\delta/2} (\tau+2v)^s b e^{-bv}dv \\
    & \leq 3^s \int_{\tau}^{\delta/2} v^s b e^{-bv} dv       \\
    & \leq (3/b)^s \int_0^\infty t^s e^{-t} dt \quad (\text{set $t=bv$}) \\
    &= \varepsilon^s (3\log e)^s s^{-s} \Gamma(s+1).
   \end{split}
\end{equation*}            
In the last step we have used $b = s\ln 2/\varepsilon = s/(\varepsilon \log e)$.
\begin{equation*}
   \begin{split}
     \int_{\delta/2}^\infty \mu\{x|\, d(x,y) \leq v\} b e^{-bv} dv & \leq \int_{\delta/2}^\infty b e^{-bv} dv \\
     &  = e^{-b\delta/2} = \left(2^{-\delta/(2\varepsilon)}\right)^s \leq \varepsilon^s. 
   \end{split}   
\end{equation*}      
In the last inequality we have used $2\varepsilon \log(1/\varepsilon) \leq \delta$.
Summing the above estimates, we get 
\[  \int_{\mathcal{X}} 2^{-ad(x,y)} d\mu(x) \leq \varepsilon^s \left\{2+ (3\log e)^s s^{-s} \Gamma(s+1)\right\}. \]
Thus the constant function $\lambda(x) := \varepsilon^{-s}  \left\{2+ (3\log e)^s s^{-s} \Gamma(s+1)\right\}^{-1}$ satisfies 
\[  \forall y\in \mathcal{X}: \quad \int_{\mathcal{X}} \lambda(x) 2^{-ad(x,y)} d\mu(x) \leq 1. \]
From Lemma \ref{lemma: duality in rate distortion theory} 
\begin{equation*}
   \begin{split}
      I(X;Y) &\geq -a\varepsilon + \int_{\mathcal{X}} \log \lambda d\mu \\
           & = s\log(1/\varepsilon) -s - \log  \left\{2+ (3\log e)^s s^{-s} \Gamma(s+1)\right\}. 
   \end{split}
\end{equation*}           
Recalling Stirling's formula $\Gamma(s+1) \sim  s^s e^{-s} \sqrt{2\pi s}$, we can find a universal constant $C>0$
satisfying 
\[  s + \log \left\{2+ (3\log e)^s s^{-s} \Gamma(s+1)\right\}  \leq C + Cs. \]
This proves the statement.
\end{proof}

\subsection{Rate distortion theory}   \label{subsection: rate distortion theory}

Here we review rate distortion function (\cite{Shannon, Shannon59}, \cite[Chapter 10]{Cover--Thomas}).
The Shannon entropy is the fundamental limit in \textit{lossless} data compression of \textit{discrete} random variables and 
processes.
For a stationary stochastic process $X_1,X_2,\dots$, its entropy 
is equal to the minimum expected number of bits per symbol for describing the process.
But if random variables $X_n$ take \textit{continuously} many values, the entropy is simply infinite
(namely, we cannot describe continuous variables perfectly within finitely many bits).
For continuous random variables and processes (e.g. audio signals, images, etc.) we have to consider 
\textit{lossy} data compression method\footnote{E.g. expanding signals in a wavelet basis,
discarding small terms and quantizing the remaining terms.} achieving some \textit{distortion constraint}.
This is the primary object of rate distortion theory. 
Rate distortion function is the fundamental limit of 
data compression in this context.

Let $(\mathcal{X},T)$ be a dynamical system with a metric $d$ and an invariant probability measure $\mu$.
For $\varepsilon >0$ we define the \textbf{rate distortion function} $R(d,\mu,\varepsilon)$ as the infimum of 
\[  \frac{I(X;Y)}{N}, \]
where $N>0$ is a natural number,
 $X$ and $Y = (Y_0,\dots, Y_{N-1})$ are random variables defined on some probability space $(\Omega, \mathbb{P})$
such that all $X$ and $Y_n$ take values in $\mathcal{X}$ and satisfy 
\[   \mathrm{Law}(X) = \mu, \quad  \mathbb{E}\left( \frac{1}{N}\sum_{n=0}^{N-1} d(T^n X, Y_n)\right) < \varepsilon. \]

We define the lower and upper rate distortion dimensions by (\ref{eq: definition of rate distortion dimension}) in 
\S \ref{subsection: statement of the main result}.

\begin{remark} \label{remark: Y takes only finitely many values in the definition of R}
In the above definition of rate distortion function we can assume that $Y$ takes only finitely many values,
namely its distribution is supported on a finite set:
Take a finite partition $\mathcal{P}$ of $\mathcal{X}$ and pick a point $x_P\in P$
for each $P\in \mathcal{P}$.
Define $f:\mathcal{X}\to \mathcal{X}$ by $f(P)=\{x_P\}$ for $P\in \mathcal{P}$
and set $Z = (Z_0,\dots, Z_{N-1}) = (f(Y_0),\dots,f(Y_{N-1}))$.
If $\mathcal{P}$ is sufficiently fine then 
\[   \mathbb{E}\left( \frac{1}{N}\sum_{n=0}^{N-1} d(T^n X, Z_n)\right) < \varepsilon. \]
From the definition of mutual information 
(or the data-processing inequality; Lemma \ref{lemma: data-processing inequality}),
\[  I(X;Z) \leq I(X;Y). \]
The random variable $Z$ takes only finitely many values.
\end{remark}

The rate distortion function $R(d,\mu,\varepsilon)$ is the minimum rate when we try to quantize the process 
$\{T^n X\}_{n\in \mathbb{Z}}$ within the average distortion bound by $\varepsilon$ 
\cite[Chapter 11]{Gray}:
For simplicity\footnote{Although the ``operational meaning'' of rate distortion function is 
important for the understanding, we do not use it in the paper.
So we do not give a complete explanation. See \cite{LDN, ECG, Gray} for the non-ergodic case.},
suppose $\mu$ is ergodic.
For any $\delta>0$, if $N$ is sufficiently large, there exists a map $f = (f_0,\dots, f_{N-1}): \mathcal{X} \to \mathcal{X}^N$
which has a finite range (i.e. it takes only finitely many values) and satisfies  
\[  \frac{\log |f(\mathcal{X})|}{N} < R(d,\mu,\varepsilon) + \delta, \quad 
      \mathbb{E}\left( \frac{1}{N}\sum_{n=0}^{N-1} d(T^n X, f_n(X))\right) < \varepsilon. \]
Namely we can approximate the process $X, TX, \dots, T^{N-1}X$ by the quantization $f_0(X), f_1(X), \dots, f_{N-1}(X)$
within the average distortion bound by $\varepsilon$.
The bits per iterate required for this description is less than 
$R(d,\mu,\varepsilon) + \delta$.

\begin{example}
 Consider the shift $\sigma: [0,1]^\mathbb{Z}\to [0,1]^\mathbb{Z}$ with a metric 
 $d(x,y) =\sum_{n\in \mathbb{Z}} 2^{-|n|} |x_n-y_n|$ and an invariant probability measure 
 $\mu = \left(\text{Lebesgue measure}\right)^{\otimes \mathbb{Z}}$.
 Then \cite[Example 22]{Lindenstrauss--Tsukamoto rate distortion}
 \[  \rdim\left([0,1]^\mathbb{Z}, \sigma, d, \mu\right) = 1. \]
\end{example}

\section{Mean Hausdorff dimension and the proof of the double variational principle}  \label{section: mean Hausdorff dimension}

In this section we introduce the key concept of the paper -- \textit{mean Hausdorff dimension}.
 We develop various comparison estimates between topological/metric mean dimensions, mean Hausdorff dimension and 
 rate distortion dimension.
 Some of the proofs are postponed to later sections.
 We prove the double variational principle (Theorem \ref{main theorem}) 
 by using these comparison estimates at the end of \S \ref{subsection: comparison between various dynamical dimensions}.

\subsection{Definition of mean Hausdorff dimension}  \label{subsection: definition of mean Hausdorff dimension}

Let $(\mathcal{X},d)$ be a compact metric space.
For $s\geq 0$ and $\varepsilon >0$ we define $\mathcal{H}_\varepsilon^s(\mathcal{X},d)$ as 
\[ \inf\left\{ \sum_{n=1}^\infty (\diam \, E_n)^s \middle| \,
    \mathcal{X} = \bigcup_{n=1}^\infty E_n \text{ with $\diam\, E_n < \varepsilon$ for all $n\geq 1$}\right\}. \]
Here we use the convention that $0^0 = 1$ and $\diam(\emptyset)^s = 0$.    
We also define 
\[ \mathcal{H}_\infty^s(\mathcal{X},d) = 
   \inf  \left \{\sum_{n=1}^\infty (\diam E_n)^s \middle|\,      \mathcal{X} = \bigcup_{n=1}^\infty E_n\right\}. \]
We set 
\[ \dim_{\mathrm{H}}(\mathcal{X},d,\varepsilon) 
    = \sup \{s\geq 0|\, \mathcal{H}^s_\varepsilon (\mathcal{X},d) \geq 1\}. \]    
The \textbf{Hausdorff dimension} $\dim_{\mathrm{H}}(\mathcal{X},d)$ is given by 
\[  \dim_{\mathrm{H}}(\mathcal{X},d) = \lim_{\varepsilon \to 0} \dim_{\mathrm{H}}(\mathcal{X},d,\varepsilon). \]

Let $(\mathcal{X},T)$ be a dynamical system with a metric $d$.
As in \S \ref{subsection: topological and metric mean dimensions}
we set $d_N(x,y) = \max_{0\leq n<N} d(T^n x, T^n y)$.
We define the \textbf{mean Hausdorff dimension} by 
\begin{equation}  \label{eq: definition of mean Hausdorff dimension}
   \mdim_{\mathrm{H}} (\mathcal{X},T,d) = \lim_{\varepsilon \to 0}
    \left(\limsup_{N\to \infty} \frac{1}{N} \dim_{\mathrm{H}}(\mathcal{X},d_N,\varepsilon)\right).
\end{equation}

\begin{remark}  \label{remark: lower mean Hausdorff dimension}
We can also define the \textbf{lower mean Hausdorff dimension}
$\underline{\mdim}_{\mathrm{H}}(\mathcal{X},T,d)$ by replacing $\limsup_N$ in 
(\ref{eq: definition of mean Hausdorff dimension}) with $\liminf_N$.
But we do not seriously use this concept in the paper.
\end{remark}

\subsection{Comparison between various dynamical dimensions}  \label{subsection: comparison between various dynamical dimensions}

The following proposition extends Theorem \ref{theorem: metric mean dimension dominates mean dimension}
to mean Hausdorff dimension and rate distortion dimension.

\begin{proposition}  \label{prop: mean Hausdorff dimension dominates mean dimension}
  Let $(\mathcal{X},T)$ be a dynamical system with a metric $d$ and an invariant probability measure $\mu$.
   \begin{equation}   \label{eq: mean dimension and mean Hausdorff dimension}
       \mdim(\mathcal{X},T) \leq \mdim_{\mathrm{H}}(\mathcal{X},T,d) \leq \underline{\mdim}_{\mathrm{M}}(\mathcal{X},T,d).
   \end{equation}
  \begin{equation}  \label{eq: metric mean dimension dominates rate distortion dimension}
     \begin{split}
       \overline{\rdim}(\mathcal{X},T,d,\mu) &\leq \overline{\mdim}_{\mathrm{M}}(\mathcal{X},T,d), \\
       \underline{\rdim}(\mathcal{X},T,d,\mu) &\leq \underline{\mdim}_{\mathrm{M}}(\mathcal{X},T,d).
     \end{split}
  \end{equation}
\end{proposition}

\begin{proof}
The nontrivial result is only $\mdim(\mathcal{X},T) \leq \mdim_{\mathrm{H}}(\mathcal{X},T,d)$.
The rest of the statement is easy.
We first prove easy estimates.
Let $N\geq 1$ and $\varepsilon >0$.
Consider an open cover $\mathcal{X} = U_1\cup\dots \cup U_n$ with 
$\diam (U_i, d_N) < \varepsilon$ and $n= \#(\mathcal{X},d_N,\varepsilon)$.

We have $\mathcal{H}_\varepsilon^s (\mathcal{X},d_N) \leq n \varepsilon^s$.
If $s > \log n / \log(1/\varepsilon)$ then $\mathcal{H}^s(\mathcal{X},d_N) < 1$.
This shows 
\[  \dim_{\mathrm{H}}(\mathcal{X},d_N,\varepsilon) \leq \frac{\log \#(\mathcal{X},d_N,\varepsilon)}{\log (1/\varepsilon)}. \]
Divide this by $N$ and take limits with respect to $N$ and then $\varepsilon$. 
It follows that $\mdim_{\mathrm{H}}(\mathcal{X},T,d) \leq \underline{\mdim}_{\mathrm{M}}(\mathcal{X},T,d)$.

Next we consider (\ref{eq: metric mean dimension dominates rate distortion dimension}).
Let $X$ be a random variable obeying $\mu$.
Choose a point $x_i$ from each $U_i$.
We define $f:\mathcal{X}\to \{x_1,\dots,x_n\}$ by $f(x)=x_i$ where
$i$ is the smallest number with $x\in U_i$.
Set $Y=(f(X), T f(X), \dots, T^{N-1} f(X))$.
Since $d(T^k x, T^k f(x)) < \varepsilon$ for all $x\in \mathcal{X}$ and $0\leq k <N$,
\[  \frac{1}{N}\sum_{k=0}^{N-1} d(T^k X, T^k f(X)) < \varepsilon. \]
Since $Y$ takes at most $n =  \#(\mathcal{X},d_N,\varepsilon)$ values,
\[  I(X;Y) \leq H(Y) \leq \log n . \]
This shows 
\[  R(d,\mu,\varepsilon) \leq \frac{\log  \#(\mathcal{X},d_N,\varepsilon)}{N}. \]
Letting $N\to \infty$, we get $R(d,\mu,\varepsilon) \leq S(\mathcal{X},T,d,\varepsilon)$.
Divide this by $\log(1/\varepsilon)$ and take the upper/lower limits with respect to $\varepsilon$.
This proves (\ref{eq: metric mean dimension dominates rate distortion dimension}).

Now we come to the main point; the comparison between mean dimension and mean Hausdorff dimension.
We use the idea of the proof of Theorem \ref{theorem: metric mean dimension dominates mean dimension}
(comparison between topological/metric mean dimensions).
We need some preliminary claims.
In the sequel we denote by $\nu_N$ and $\norm{\cdot}_\infty$ the standard Lebesgue measure and $\ell^\infty$-norm
on $\mathbb{R}^N$. 
For $A\subset \{1,2,\dots,N\}$ we denote by $\pi_A:\mathbb{R}^N \to \mathbb{R}^A$ the projection to the $A$-coordinates.
For $0\leq n\leq N$ we define $P_n$ as the $n$-skeleton of the cube $[0,1]^N$, i.e. the set of $x\in [0,1]^N$
satisfying $|\{k | x_k = \text{$0$ or $1$}\}| \geq N-n$.

\begin{claim} \label{claim: contraction to skeleton}
   Let $K\subset [0,1]^N$ be a closed subset and $1 \leq n\leq N$.
    \begin{enumerate}
       \item $\nu_N(K) \leq 2^N \mathcal{H}^N_\infty\left(K, \norm{\cdot}_\infty\right)$.
       \item $\nu_N\left(\bigcup_{|A|\geq n} \pi_A^{-1}(\pi_A K)\right) \leq 4^N \mathcal{H}^n_\infty\left(K,\norm{\cdot}_\infty\right)$.
       \item If $\mathcal{H}^{n+1}_\infty\left(K,\norm{\cdot}_\infty\right) < 4^{-N}$ then 
               there exists a $1$-embedding $f:K\to P_n$, i.e. a continuous map satisfying $f(x)\neq f(y)$ for any $x,y\in K$
               with $\norm{x-y}_\infty =1$.
    \end{enumerate}
\end{claim}

\begin{proof}
(1) Let $K = \bigcup_{k \geq 1} E_k$ and set $l_k = \diam \left(E_k, \norm{\cdot}_\infty\right)$.
    Take $x_k\in E_k$.
    Since $E_k\subset x_k + [-l_k,l_k]^N$,
    \[  \nu_N(K) \leq \sum_{k=1}^\infty (2 l_k)^N = 2^N \sum_{k=1}^\infty l_k^N. \]

(2)  $\nu_N\left(\bigcup_{|A|\geq n} \pi_A^{-1}(\pi_A K)\right)$ is bounded by 
      \[  \sum_{|A|\geq n} \nu_N\left(\pi_A^{-1}(\pi_A K)\right) = \sum_{|A|\geq n} \nu_{|A|} (\pi_A K). \]
     Apply the above (1) to $\pi_A K\subset [0,1]^A$: 
     \[   \nu_{|A|}(\pi_A K) \leq 2^{|A|} \mathcal{H}_\infty^{|A|}\left(\pi_A K,\norm{\cdot}_\infty\right) 
      \leq 2^N  \mathcal{H}_\infty^{|A|} \left(\pi_A K, \norm{\cdot}_\infty \right). \]
      Since $\pi_A$ is one-Lipschitz, 
      $\mathcal{H}_\infty^{|A|}\left(\pi_A K,\norm{\cdot}_\infty\right)  \leq  \mathcal{H}^{|A|}_\infty\left(K,\norm{\cdot}_\infty\right)$.
      Thus 
      \begin{equation*}
          \begin{split}
           \nu_N\left(\bigcup_{|A|\geq n} \pi_A^{-1}(\pi_A K)\right)  & 
           \leq 2^N \sum_{|A|\geq n} \mathcal{H}^{|A|}_\infty\left(K,\norm{\cdot}_\infty \right) \\
           & \leq  2^N \sum_{|A|\geq n} \mathcal{H}^n_\infty\left(K,\norm{\cdot}_\infty \right) \\
           & \leq  4^N \mathcal{H}_\infty^n \left(K,\norm{\cdot}_\infty \right).           
          \end{split}
     \end{equation*}     

(3)   If $n=N$ then the statement is trivial. So we assume $n<N$.
      It follows from the above (2) that 
      \[  \nu_N\left(\bigcup_{|A|\geq n+1} \pi_A^{-1}(\pi_A K)\right) < 1. \]
     In particular we can find $q \in (0,1)^N$ outside of $\bigcup_{|A|\geq n+1} \pi_A^{-1}(\pi_A K)$.
     For $1\leq m\leq N$ we set 
     \[  C_m = P_m \cap \bigcup_{|A|=m} \pi_A^{-1}(\pi_A (q)). \]
     This is a finite set. (Each facet of $P_m$ contains exactly one point of $C_m$.)
     By using the central projection from each point of $C_m$, we define a continuous map 
     $g_m: P_m \setminus C_m\to P_{m-1}$.
     This map has the following properties:
     \begin{itemize}
        \item $\norm{g_m(x)-g_m(y)}_\infty = 1$ for $x,y\in P_m\setminus C_m$ with $\norm{x-y}_\infty = 1$. 
        \item For $1\leq l <m$     
                \[   g_m\left(P_m\setminus \bigcup_{|A|=l} \pi_A^{-1}(\pi_A(q))\right) = 
                      P_{m-1}\setminus \bigcup_{|A|=l} \pi_A^{-1}(\pi_A(q)). \]             
     \end{itemize}
     Since $K\cap \bigcup_{|A|\geq n+1} \pi_A^{-1}(\pi_A(q)) = \emptyset$,
       we can define $f = g_{n+1}\circ g_{n+2}\circ \dots \circ g_N: K\to P_n$.
     If $x,y\in K$ satisfy $\norm{x-y}_\infty =1$ then $\norm{f(x)-f(y)}_\infty = 1$.
     In particular $f$ is a $1$-embedding.
\end{proof}

\begin{claim} \label{claim: widim and  Hausdorff dimension}
  Let $N$ be a positive integer and $\varepsilon,\delta, s, \tau, L$ positive numbers with $4^N L^{s+\tau} \delta^\tau < 1$.
  Let $(K,d)$ be a compact metric space with $\dim_{\mathrm{H}}(K,d,\delta) < s$.
  Suppose there exists an $L$-Lipschitz map $\varphi: (K,d) \to \left([0,1]^N, \norm{\cdot}_\infty\right)$ such that 
  if $x,y\in K$ satisfy $\norm{\varphi(x)-\varphi(y)}_\infty < 1$ then $d(x,y) < \varepsilon$.
  Then $\widim_\varepsilon (K,d) \leq s+\tau$.
\end{claim}

\begin{proof}
It follows from $\dim_{\mathrm{H}}(K,d,\delta) < s$ that there exists a covering $K=\bigcup_{n=1}^\infty E_n$ satisfying 
$\diam\, E_n < \delta$ and $\sum_{n=1}^\infty (\diam\, E_n)^s < 1$.
Then 
\[  \mathcal{H}_\infty^{s+\tau}(K,d) \leq \sum_{n=1}^\infty (\diam \, E_n)^{s+\tau}  < \delta^\tau. \]
Since $\varphi$ is $L$-Lipschitz, 
\[  \mathcal{H}^{s+\tau}_\infty \left(\varphi(K),\norm{\cdot}_\infty \right)  < L^{s+\tau} \delta^\tau < 4^{-N}. \]
Hence $\mathcal{H}^{\lfloor s+\tau \rfloor +1}_\infty\left(\varphi(K),\norm{\cdot}_\infty\right) < 4^{-N}$.
Apply Claim \ref{claim: contraction to skeleton} (3) to $\varphi(K)$:
There exists a $1$-embedding $f:\varphi(K) \to P_{\lfloor s+\tau \rfloor}$.
Then $f\circ \varphi: K\to P_{\lfloor s+\tau \rfloor}$ becomes an $\varepsilon$-embedding.
The skeleton $P_{\lfloor s+\tau \rfloor}$ admits a structure of a $\lfloor s+\tau \rfloor$-dimensional simplicial complex.
So $\widim_\varepsilon(K,d) \leq \lfloor s+\tau \rfloor$.
\end{proof}

We start the proof of $\mdim(\mathcal{X},T) \leq \mdim_{\mathrm{H}}(\mathcal{X},T,d)$.
We can assume $\mdim_{\mathrm{H}}(\mathcal{X},T,d) < \infty$.
We take $\tau>0$ and $s>\mdim_{\mathrm{H}}(\mathcal{X},T,d)$.
Let $\varepsilon >0$.

\begin{claim}  \label{claim: construction of Lipschitz map}
  There exist a positive number $L$, a positive integer $M$ and an $L$-Lipschitz map 
  $\varphi:(\mathcal{X},d) \to \left([0,1]^M, \norm{\cdot}_\infty \right)$
  such that if $x,y\in \mathcal{X}$ satisfy $\norm{\varphi(x)-\varphi(y)}_\infty < 1$ then $d(x,y) < \varepsilon$.
\end{claim}

\begin{proof}
Choose a Lipschitz function $\psi:\mathbb{R}\to [0,1]$ satisfying $\psi(t) = 1$ for $t\leq \varepsilon/4$ and 
$\psi(t)= 0$ for $t\geq \varepsilon/2$.
Take an $\varepsilon/4$-spanning subset $\{x_1,\dots, x_M\}\subset \mathcal{X}$, i.e.~so that for any $x\in \mathcal{X}$ there exists 
$x_i$ with $d(x,x_i) < \varepsilon/4$.
We define $\varphi:\mathcal{X}\to [0,1]^M$ by 
\[ \varphi(x) = \left(\psi(d(x,x_1)),\dots, \psi(d(x,x_M))\right). \] 
\end{proof}

For $N\geq 1$ we define an $L$-Lipschitz map 
\[   \varphi_N:(\mathcal{X},d_N) \to \left(\left([0,1]^M\right)^N, \norm{\cdot}_\infty\right) \]
by $\varphi_N(x) = (\varphi(x), \varphi(Tx), \dots, \varphi(T^{N-1}x))$.
This has the property that if $x,y\in \mathcal{X}$ satisfy 
$\norm{\varphi_N(x)-\varphi_N(y)}_\infty < 1$ then $d_N(x,y) < \varepsilon$.

Choose a sufficiently small $\delta>0$ satisfying $4^M L^{s+\tau} \delta^{\tau} < 1$.
It follows from\footnote{Indeed here we use only $\underline{\mdim}_{\mathrm{H}}(\mathcal{X},T,d) < s$.} 
$\mdim_{\mathrm{H}}(\mathcal{X},T,d) < s$ that there exists $N_1<N_2<N_3<\dots \to \infty$
satisfying $\dim_{\mathrm{H}}(\mathcal{X},d_{N_k},\delta) < sN_k$.
Note 
\[  4^{M N_k} L^{sN_k + \tau N_k} \delta^{\tau N_k} = (4^M L^{s+\tau} \delta^{\tau})^{N_k} < 1. \]
Then we can apply Claim \ref{claim: widim and  Hausdorff dimension} to the space $(\mathcal{X},d_{N_k})$ and the map 
$\varphi_{N_k}$ with the parameters $MN_k, \varepsilon, \delta, s N_k, \tau N_k, L$.
This provides 
\[   \widim_\varepsilon(\mathcal{X},d_{N_k}) \leq s N_k+\tau N_k,   \]
and hence 
\[  \lim_{N\to \infty} \frac{1}{N} \widim_\varepsilon(\mathcal{X},d_N) \leq s+\tau. \]
The right-hand side is independent of $\varepsilon$. Thus $\mdim(\mathcal{X},T) \leq s+\tau$.
Let $\tau\to 0$ and $s\to \mdim_{\mathrm{H}}(\mathcal{X},T,d)$. This proves the statement.
\end{proof}

\begin{remark}
  The above proof actually shows 
  \[  \mdim(\mathcal{X},T) \leq \underline{\mdim}_{\mathrm{H}}(\mathcal{X},T,d), \]
  where the right-hand side is the lower mean Hausdorff dimension (Remark \ref{remark: lower mean Hausdorff dimension}).
\end{remark}

\begin{example}
   Let $\sigma: [0,1]^\mathbb{Z}\to [0,1]^\mathbb{Z}$ be the shift on the alphabet $[0,1]$ with 
  a metric $d(x,y) = \sum_{n\in \mathbb{Z}} 2^{-|n|} |x_n-y_n|$ as in Example \ref{example: mean dimension of Hilbert cube}.
  Then $\mdim_{\mathrm{H}}\left([0,1]^\mathbb{Z},\sigma,d\right) =1$ because
  \begin{equation*}
     \begin{split}
            1 & = \mdim\left([0,1]^\mathbb{Z},\sigma\right) \\
       &   \leq  \mdim_{\mathrm{H}}\left([0,1]^\mathbb{Z},\sigma,d\right) 
         \leq \mdim_{\mathrm{M}}\left([0,1]^\mathbb{Z},\sigma,d\right) = 1.
    \end{split}    
  \end{equation*}    
\end{example}

The next two theorems are the most crucial ingredients of the proof of the double variational principle.
Their proofs are postponed to later sections.
Before stating the results we need to introduce a concept expressing some regularity of metrics:

\begin{definition}  \label{definition: tame growth of covering numbers}
  Let $(\mathcal{X},d)$ be a compact metric space. It is said to have the \textbf{tame growth of covering numbers}
  if for every $\delta>0$ 
  \[  \lim_{\varepsilon\to 0} \varepsilon^\delta \log \#(\mathcal{X},d,\varepsilon) = 0. \] 
  Notice that this is purely a condition on metrics and does not involve dynamics.
\end{definition}

\begin{example}  \label{example: tame growth of covering numbers}
\begin{enumerate}
   \item If $\mathcal{X}$ is a compact subset of a finite dimensional Banach space $(V,\norm{\cdot})$,
           then $(\mathcal{X},\norm{\cdot})$ has the tame growth of covering numbers because 
           $\#(\mathcal{X},\norm{\cdot},\varepsilon) = O(\varepsilon^{-\dim V})$.
  \item  If a compact metric space $(K,\rho)$ has the tame growth of covering numbers, then the following metric $d$ on the shift 
           space $K^\mathbb{Z}$ also has the tame growth of covering numbers:
           \[  d(x,y) = \sum_{n\in \mathbb{Z}} 2^{-|n|} \rho(x_n,y_n). \]
  \item  It follows from (1) and (2) that the metric $d(x,y) = \sum_{n\in \mathbb{Z}} 2^{-|n|} |x_n-y_n|$ on $[0,1]^\mathbb{Z}$
           has the tame growth of covering numbers.          
\end{enumerate}
\end{example}

The next lemma shows that the tame growth of covering numbers is a fairly mild condition.

\begin{lemma} \label{lemma: tame growth condition is a mild condition}
Let $(\mathcal{X}, d)$ be a compact metric space.
There exists a metric $d'$ on $\mathcal{X}$ such that 
$d'(x,y) \leq d(x,y)$ and $(\mathcal{X},d')$ has the tame growth of covering numbers.
In particular every compact metrizable space admits a metric having the tame growth of covering numbers.
\end{lemma}

\begin{proof}
We can assume $\diam (\mathcal{X},d) \leq 1$.
Let $K=[0,1]^\mathbb{N}$ be the one-sided infinite product of the unit interval.
We define a metric $\rho$ on it by $\rho(x,y) = \sum_{i \geq 1} 2^{-i} |x_i-y_i|$.
As in Example \ref{example: tame growth of covering numbers},
$(K,\rho)$ has the tame growth of covering numbers.
Take a countable dense subset $\{x_i\}_{i=1}^\infty \subset \mathcal{X}$.
We define $f:\mathcal{X}\to K$ by $f(x) = (d(x,x_i))_{i=1}^\infty$.
$f$ is a topological embedding and it is one-Lipschitz:
\[ \rho(f(x),f(y)) = \sum_{i=1}^\infty 2^{-i}|d(x,x_i)-d(y,x_i)| \leq \sum_{i=1}^\infty 2^{-i} d(x,y) = d(x,y). \]
The metric $d'(x,y) := \rho(f(x), f(y))$ satisfies the requirements.
\end{proof}

Recall that we have denoted by $\mathscr{M}^T(\mathcal{X})$ the set of all $T$-invariant Borel probability measures on $\mathcal{X}$.

\begin{theorem}[Existence of nice measures]  \label{theorem: existence of nice measures}
  Let $(\mathcal{X},T)$ be a dynamical system with a metric $d$ such that 
  $(\mathcal{X},d)$ has the tame growth of covering numbers.
  Then
  \[ \mdim_{\mathrm{H}}(\mathcal{X},T,d) \leq \sup_{\mu\in \mathscr{M}^T(\mathcal{X})} \underline{\rdim}(\mathcal{X},T,d,\mu). \]
\end{theorem}

\begin{theorem}[Existence of nice metrics]  \label{theorem: existence of nice metrics}
   If a dynamical system $(\mathcal{X},T)$ has the marker property, 
   then there exists a metric $d$ on $\mathcal{X}$ compatible with 
   the topology such that
   \[  \overline{\mdim}_{\mathrm{M}}(\mathcal{X},T,d) = \mdim(\mathcal{X},T). \]
\end{theorem}

The inequalities $\mdim(\mathcal{X},T) \leq \underline{\mdim}_{\mathrm{M}}(\mathcal{X},T,d)
\leq \overline{\mdim}_{\mathrm{M}}(\mathcal{X},T,d)$ always hold true.
So 
$\mdim(\mathcal{X},T) = \overline{\mdim}_{\mathrm{M}}(\mathcal{X},T,d)$ implies that 
$\mdim_{\mathrm{M}}(\mathcal{X},T,d)$ exists and is equal to $\mdim(\mathcal{X},T)$.

\begin{corollary}   \label{corollary: mean dimension and rate distortion dimension}
  Let $(\mathcal{X},T)$ be a dynamical system with a metric $d$. Then
  \[  \mdim(\mathcal{X},T) \leq  \sup_{\mu\in \mathscr{M}^T(\mathcal{X})}  \underline{\rdim}(\mathcal{X},T,d,\mu).   \]
\end{corollary}

\begin{proof}
Notice that if $d$ has the tame growth of covering numbers then the statement immediately follows from 
Proposition \ref{prop: mean Hausdorff dimension dominates mean dimension} and 
Theorem \ref{theorem: existence of nice measures}.
Hence the problem is how to reduce the general case to this case.

Let $d'$ be a metric given by Lemma \ref{lemma: tame growth condition is a mild condition}.
It has the tame growth of covering numbers.
So for any $\varepsilon >0$ there exists an invariant probability measure $\mu$ on $\mathcal{X}$
satisfying 
\[  \mdim(\mathcal{X},\sigma) \leq \underline{\rdim}(\mathcal{X},T,d',\mu) + \varepsilon.\]
Since $d'\leq d$,
\begin{equation*}
      \underline{\rdim}(\mathcal{X},T,d',\mu) 
      \leq  \underline{\rdim}(\mathcal{X},T,d,\mu). 
\end{equation*}
Since $\varepsilon>0$ is arbitrary, this proves the claim.
\end{proof}

\begin{corollary}  \label{cor: all the dimensions coincide}
  If a dynamical system $(\mathcal{X},T)$ has the marker property, then there exists a metric $d$ on $\mathcal{X}$
  such that all the following quantities are equal to each other:
  \begin{equation}  \label{eq: various dynamical dimensions}
     \begin{split} 
       &\mdim(\mathcal{X},T),  \quad  \mdim_{\mathrm{H}}(\mathcal{X},T,d),  \\
       & \overline{\mdim}_{\mathrm{M}}(\mathcal{X},T,d) , \quad \underline{\mdim}_{\mathrm{M}}(\mathcal{X},T,d), \\
       & \sup_{\mu\in \mathscr{M}^T(\mathcal{X})} \overline{\rdim}(\mathcal{X},T,d,\mu), \quad 
        \sup_{\mu\in \mathscr{M}^T(\mathcal{X})}   \underline{\rdim}(\mathcal{X},T,d,\mu).      
     \end{split}
  \end{equation}
\end{corollary}

\begin{proof}
  All the quantities in (\ref{eq: various dynamical dimensions})
  are bounded between $\mdim(\mathcal{X}, T)$ and 
  $\overline{\mdim}_{\mathrm{M}}(\mathcal{X},T,d)$ by 
  Proposition \ref{prop: mean Hausdorff dimension dominates mean dimension} and Corollary 
  \ref{corollary: mean dimension and rate distortion dimension}.
  Take a metric $d$ given in Theorem \ref{theorem: existence of nice metrics}.
  Then $\mdim(\mathcal{X},T) = \overline{\mdim}_{\mathrm{M}}(\mathcal{X},T,d)$ and hence all the quantities in 
  (\ref{eq: various dynamical dimensions}) coincide with each other.
\end{proof}

Now we can prove the double variational principle (Theorem \ref{main theorem}).

\begin{proof}[Proof of Theorem \ref{main theorem}] 

Let $(\mathcal{X},T)$ be a dynamical system having the marker property.
From Corollary \ref{corollary: mean dimension and rate distortion dimension}
 \begin{equation*} 
        \begin{split}
         \mdim(\mathcal{X},T)   \leq  \inf_{d\in \mathscr{D}(\mathcal{X})} \sup_{\mu\in \mathscr{M}^T(\mathcal{X})} 
         \underline{\rdim}(\mathcal{X},T,d,\mu), \\
        \leq  \inf_{d\in \mathscr{D}(\mathcal{X})} \sup_{\mu\in \mathscr{M}^T(\mathcal{X})} 
       \overline{\rdim}(\mathcal{X},T,d,\mu).   
       \end{split}
     \end{equation*}
On the other hand we already know (Corollary \ref{cor: all the dimensions coincide}) that there exists $d\in \mathscr{D}(\mathcal{X})$
satisfying 
\[ \mdim(\mathcal{X},T) =  
    \sup_{\mu\in \mathscr{M}^T(\mathcal{X})}  \overline{\rdim}(\mathcal{X},T,d,\mu)
    = \sup_{\mu\in \mathscr{M}^T(\mathcal{X})}  \underline{\rdim}(\mathcal{X},T,d,\mu).   \]
\end{proof}

\begin{remark}
It follows from Proposition \ref{prop: mean Hausdorff dimension dominates mean dimension} and 
Theorem \ref{theorem: existence of nice measures} that for a metric $d$ having the tame growth of covering numbers
\begin{equation*}
   \begin{split}
     \mdim_{\mathrm{H}}(\mathcal{X}, T,d) \leq 
     \sup_{\mu\in \mathscr{M}^T(\mathcal{X})} \overline{\rdim}(\mathcal{X},T,d,\mu) 
     \leq \overline{\mdim}_{\mathrm{M}}(\mathcal{X},T,d), \\
      \mdim_{\mathrm{H}}(\mathcal{X}, T,d) \leq 
     \sup_{\mu\in \mathscr{M}^T(\mathcal{X})} \underline{\rdim}(\mathcal{X},T,d,\mu) 
     \leq \underline{\mdim}_{\mathrm{M}}(\mathcal{X},T,d).
  \end{split}
\end{equation*}     
Hence we have a sufficient criterion (under the assumption of the tame growth of covering numbers) for the equality (\ref{eq: metric mean dimension is equal to rate distortion dimension})
in Problem \ref{problem: metric mean dimension is equal to rate distortion dimension}:
If mean Hausdorff dimension is equal to metric mean dimension, then they also coincide with 
the supremum of rate distortion dimensions.
\end{remark}

\section{Proof of Theorem \ref{theorem: existence of nice measures}: Geometric measure theory and Misiurewicz's technique}
\label{section: proof of existence of nice measures}

We prove Theorem \ref{theorem: existence of nice measures} in this section.
The proof is a combination of geometric measure theory and the rate distortion theory version of 
Misiurewicz's technique \cite{Misiurewicz, Lindenstrauss--Tsukamoto rate distortion}.

\subsection{Geometric measure theory around Frostman's lemma}
\label{subsection: geometric measure theory around Frostman's lemma}

The purpose of this subsection is to prepare some basics of geometric measure theory 
around Frostman's lemma.
Frostman's lemma is a fundamental result in geometric measure theory.
It states that a Borel subset $A\subset \mathbb{R}^n$ has positive (possibly infinite) $s$-dimensional
Hausdorff measure if and only if there exists a nonzero Radon measure $\mu$ on $\mathbb{R}^n$ 
supported on $A$ and satisfying $\mu(B_r(x)) \leq r^s$ for all $x\in \mathbb{R}^n$ and $r>0$ (see \cite[8.8 Theorem]{Mattila}). 
We need a generalization of this result to compact metric spaces, which is due to 
Howroyd \cite{Howroyd}.
Our presentation follows the book of Mattila \cite[Sections 8.14-8.17]{Mattila}.

Let $(\mathcal{X},d)$ be a compact metric space.
For $\delta>0$ and $s\geq 0$ we define $\lambda_\delta^s(\mathcal{X},d)$ as 
\[  \inf \> \sum_{n=1}^\infty c_n (\diam E_n)^s \]
where the infimum is taken over all countable families $\{(E_n,c_n)\}$ such that 
$0<c_n<\infty$, $E_n\subset \mathcal{X}$ with $\diam E_n < \delta$ and 
\[  \forall  x\in \mathcal{X}: \>  \sum_{n=1}^\infty c_n 1_{E_n}(x) \geq 1. \]
Obviously $\lambda_\delta^s(\mathcal{X},d) \leq \mathcal{H}^s_\delta(\mathcal{X},d)$.

\begin{lemma} \label{lemma: weighted Hausdorff measure}
\[  \mathcal{H}_{6\delta}^s(\mathcal{X},d) \leq 6^s \lambda^s_\delta(\mathcal{X},d). \]
\end{lemma}

\begin{proof}
The proof is essentially the same as \cite[8.16 Lemma]{Mattila},
but the above statement is a bit different\footnote{An important point for us is that the statement is 
valid for each fixed $\delta$ (not only the limits of $\delta\to 0$).} 
(at least formally) from \cite[8.16 Lemma]{Mattila}.
So we include a proof.

\begin{claim}  \label{claim: covering lemma}
Let $a_1,\dots, a_N$ and $m$ be positive integers.
Let $B^\circ_n=B^\circ_{r_n}(x_n)$ $(1\leq n\leq N)$ be open balls in $\mathcal{X}$ of radius $r_n<\delta$.
If 
\[   \mathcal{X} = \left\{\sum_{n=1}^N a_n 1_{B^\circ_n} \geq m\right\} \]
then
\[   \mathcal{H}_{6\delta}^s(\mathcal{X},d) \leq m^{-1} 6^s \sum_{n=1}^N a_n r_n^s. \]
\end{claim}

\begin{proof}
The induction on $m$: If $m=1$ then $\mathcal{X} = \bigcup_{n=1}^N B^\circ_n$ and $\diam B^\circ_n < 2\delta$.
Hence 
\[ \mathcal{H}_{2\delta}^s(\mathcal{X},d) \leq \sum_{n=1}^N (\diam B^\circ_n)^s \leq 2^s \sum_{n=1}^N  r_n^s. \]
Suppose $m\geq 2$. By Finite Vitali's covering lemma (see \cite[Lemma 2.27]{Einsiedler--Ward}) there exists 
a disjoint family $\mathcal{B}\subset \{B^\circ_1,\dots, B^\circ_N\}$ satisfying 
$\mathcal{X} = \bigcup_{B^\circ_n\in \mathcal{B}} 3B^\circ_n$ where $3B^\circ_n := B_{3r_n}^\circ(x_n)$.
We have $\diam (3B^\circ_n) \leq 6r_n < 6\delta$ and hence 
\[  \mathcal{H}^s_{6\delta}(\mathcal{X},d) \leq 6^s \sum_{B^\circ_n\in \mathcal{B}} r_n^s. \]
We set 
\[  a'_n = \begin{cases}
              a_n & \text{if $B^\circ_n\not\in \mathcal{B}$} \\
              a_n-1 & \text{if $B^\circ_n\in \mathcal{B}$}.
             \end{cases} \] 
Since $\mathcal{B}$ is a disjoint family, we have $\sum a'_n 1_{B^\circ_n}(x) \geq m-1$ for all $x\in \mathcal{X}$.
By the induction hypothesis, 
\[  (m-1) \mathcal{H}_{6\delta}^s(\mathcal{X},d) \leq 6^s \sum_{n=1}^N a'_n r_n^s. \]
Thus 
\[ m \mathcal{H}^s_{6\delta}(\mathcal{X},d) \leq 6^s \sum_{n=1}^N a'_n r_n^s 
    + 6^s \sum_{B^\circ_n\in \mathcal{B}} r_n^s  = 6^s \sum_{n=1}^N a_n r_n^s. \]
\end{proof}

Let $0<c_n< \infty$, $E_n\subset \mathcal{X}$ such that $\diam E_n <\delta$ and 
$\sum_{n=1}^\infty c_n 1_{E_n}(x) \geq 1$ for all $x\in \mathcal{X}$.
Let $\varepsilon >0$ and $0<t<1$.
We choose $\diam E_n < r_n < \delta$ satisfying 
\begin{equation}  \label{eq: choice of r_n in GMT lemma}
     \sum_{n=1}^\infty c_n r_n^s < \varepsilon + \sum_{n=1}^\infty c_n (\diam E_n)^s. 
\end{equation}     
We pick $x_n\in E_n$. The open balls $B^\circ_n := B^\circ_{r_n}(x_n)$ contain $E_n$ and hence 
\[  \mathcal{X} = \bigcup_{N=1}^\infty \left\{\sum_{n=1}^N c_n 1_{B^\circ_n} > t\right\}. \]
Each set $\{\cdots\}$ here is open. 
Since $\mathcal{X}$ is compact, we can find $N$ such that
$\mathcal{X} = \left\{ \sum_{n=1}^N c_n 1_{B_n^\circ} > t \right\}$.
We choose rational numbers $0<b_n \leq c_n$ so that 
$\mathcal{X} =  \left\{ \sum_{n=1}^N b_n 1_{B_n^\circ} > t \right\}$.
Take a positive integer $p$ such that all $a_n := p b_n$ become integers.
Set $m = \lceil pt \rceil$. Then 
$\mathcal{X} = \left\{ \sum_{n=1}^N a_n 1_{B_n^\circ} \geq m \right\}$.
By Claim \ref{claim: covering lemma} 
\[  \mathcal{H}^s_{6\delta}(\mathcal{X},d) \leq m^{-1} 6^s \sum_{n=1}^N a_n r_n^s  
     \leq p m^{-1}  6^s \sum_{n=1}^N b_n r_n^s  \leq p m^{-1} 6^s \sum_{n=1}^N c_n r_n^s. 
      \]
It follows from $m\geq pt$ and (\ref{eq: choice of r_n in GMT lemma}) that 
\[    \mathcal{H}^s_{6\delta}(\mathcal{X},d) \leq t^{-1} 6^s \left(\varepsilon +  \sum_{n=1}^\infty c_n (\diam E_n)^s\right). \]
Let $\varepsilon \to 0$ and $t\to 1$. This proves the statement.
\end{proof}

\begin{lemma} \label{lemma: Howroyd--Frostman}
There exists a Borel measure $\mu$ on $\mathcal{X}$ satisfying $\mu(X)= \lambda^s_\delta(\mathcal{X},d)$ and 
\[  \mu(E) \leq (\diam E)^s \quad \text{for all $E\subset \mathcal{X}$ with $\diam E < \delta$}. \]
\end{lemma}

\begin{proof}[Sketch of the proof]
See \cite[8.17 Theorem]{Mattila} for the details.
We define a sublinear functional $p(f)$ for continuous functions $f:\mathcal{X}\to \mathbb{R}$ by 
\[  p(f) = \inf\, \sum_{n=1}^\infty c_n \left(\diam(E_n)\right)^s, \]
where the infimum is taken over all countable families $\{(E_n,c_n)\}$ such that 
$0<c_n<\infty$, $E_n\subset \mathcal{X}$ with $\diam E_n < \delta$ and 
\[  \forall  x\in \mathcal{X}: \>  \sum_{n=1}^\infty c_n 1_{E_n}(x) \geq f(x). \]
We have $p(1) = \lambda_\delta(\mathcal{X},d)$.
By using the Hahn--Banach theorem, we can find a linear functional $L$ defined on the space of continuous functions in $\mathcal{X}$
such that $L(1) = p(1)$ and for any continuous function $f$ on $\mathcal{X}$
\[ -p(-f) \leq L(f) \leq p(f). \]
If $f\geq 0$ then $L(f)\geq -p(-f) =0$. So $L$ is a positive functional.
It follows from the Riesz representation theorem that there exists a Borel measure $\mu$ on $\mathcal{X}$
satisfying $L(f) = \int_{\mathcal{X}} f d\mu$.
We can easily check that $\mu$ satisfies the statement.
\end{proof}

\begin{corollary}  \label{cor: quantitative Frostman's lemma}
Let $0<c<1$. 
We can choose $\delta_0 = \delta_0(c)\in (0,1)$ independent of $(\mathcal{X}, d)$ so that 
for any $0<\delta \leq \delta_0$ there exists a Borel probability measure $\mu$ on $\mathcal{X}$ satisfying 
\[  \mu(E) \leq (\diam E)^{c \dim_{\mathrm{H}}(\mathcal{X},d,\delta)}  \quad 
     \text{for all $E\subset \mathcal{X}$ with $\diam E < \delta/6$}. \]
\end{corollary}

\begin{proof}
Take $0<\delta_0<1$ satisfying
\[  \left(\frac{1}{\delta_0}\right)^{\frac{1-c}{2c}} \geq 6. \] 
Let $0< \delta\leq \delta_0$.
If $\dim_{\mathrm{H}}(\mathcal{X},d,\delta) = 0$ then the statement is trivial (the delta measure satisfies the claim; recall that we 
promised $0^0=1$).
So we assume $\dim_{\mathrm{H}}(\mathcal{X},d,\delta) >0$.
From Lemma \ref{lemma: Howroyd--Frostman} it is enough to prove $\lambda_{\delta/6}^s( \mathcal{X},d) \geq 1$
for $s := c \dim_{\mathrm{H}}(\mathcal{X},d,\delta)$.

Set $t= \frac{1+c}{2} \dim_{\mathrm{H}}(\mathcal{X},d,\delta)$.
We have $t-s = \frac{1-c}{2c}s$ and hence
\begin{equation*}
    \begin{split}
     \mathcal{H}_\delta^s(\mathcal{X},d) & \geq \left(\frac{1}{\delta}\right)^{t-s} \mathcal{H}^t_\delta(\mathcal{X},d) \\
     & \geq \left(\frac{1}{\delta}\right)^{t-s}   \quad (\text{by $t < \dim_{\mathrm{H}}(\mathcal{X},d,\delta)$})          \\
      & =  \left(\frac{1}{\delta}\right)^{\frac{1-c}{2c}s} \geq 6^s.  
    \end{split}
\end{equation*}       
Then $\lambda_{\delta/6}^s(\mathcal{X},d) \geq 1$
by Lemma \ref{lemma: weighted Hausdorff measure}.
\end{proof}

\subsection{$L^1$-mean Hausdorff dimension and the condition of tame growth of covering numbers}
\label{subsection: L^1 mean Hausdorff dimension}

We need a modification of mean Hausdorff dimension.
Let $(\mathcal{X},T)$ be a dynamical system with a metric $d$.
For $N\geq 1$ we define a new metric $\bar{d}_N$ on $\mathcal{X}$ by 
\[  \bar{d}_N(x,y) = \frac{1}{N} \sum_{n=0}^{N-1} d(T^n x, T^n y). \]
This metric is more closely connected to the distortion condition 
\[  \mathbb{E}\left( \frac{1}{N}\sum_{n=0}^{N-1} d(T^n X, Y_n)\right) < \varepsilon \]
in the definition of rate distortion function (see \S \ref{subsection: rate distortion theory}) 
than $d_N(x,y) = \max_{0\leq n<N} d(T^n x, T^n y)$.
The next lemma is a manifestation of this connection.
(This will be used only in \S \ref{section: example: algebraic actions}.
But it is conceptually a toy model of the proof of Theorem \ref{theorem: existence of nice measures}.)

\begin{lemma}\label{lemma: L^1 metric and rate distortion dimension}
Let $\mu$ be a $T$-invariant probability measure on $\mathcal{X}$.
Suppose that there exist $s\geq 0$ and $\delta>0$ such that for any $N\geq 1$
\[  \mu(E) \leq \left(\diam (E,\bar{d}_N)\right)^{sN} \quad \text{for all $E\subset \mathcal{X}$ with $\diam (E,\bar{d}_N) < \delta$}. \]
Then $\underline{\rdim}(\mathcal{X},T,d,\mu) \geq s$.
\end{lemma}

\begin{proof}
Define $\Delta: \mathcal{X}\to \mathcal{X}^N$ by $\Delta(x)= (x,Tx,\dots,T^{N-1}x)$.
We define a metric $\bar{d}_N$ on $\mathcal{X}^N$ by
\[ \bar{d}_N\left((x_0,\dots,x_{N-1}),(y_0,\dots,y_{N-1})\right) = \frac{1}{N} \sum_{n=0}^{N-1} d(x_n,y_n). \]
The push-forward measure $\Delta_*\mu$ on $\mathcal{X}^N$ satisfies 
\[   \Delta_*\mu(E) \leq \left(\diam (E,\bar{d}_N)\right)^{sN} \quad
       \text{for all $E\subset \mathcal{X}^N$ with $\diam (E,\bar{d}_N) <\delta$}. \]
 
 Let  $\varepsilon >0$ with $2\varepsilon \log (1/\varepsilon) \leq \delta$.
 Let $X$ and $Y = (Y_0,\dots, Y_{N-1})$ be random variables such that all $X$ and $Y_n$ take values in $\mathcal{X}$ and satisfy 
 \[    \mathrm{Law} X = \mu, \quad
      \mathbb{E}\left(\frac{1}{N} \sum_{n=0}^{N-1} d(T^n X, Y_n)\right) < \varepsilon. \]             
This condition is equivalent to $\mathrm{Law} \Delta(X) = \Delta_*\mu$ and 
 $\mathbb{E} \bar{d}_N(\Delta(X), Y) < \varepsilon$.
So we apply Lemma \ref{lemma: from geometric measure theory to rate distortion theory} to $(\Delta(X), Y)$ and get 
\[  \frac{I(X;Y)}{N} = \frac{I\left(\Delta(X);Y\right)}{N} \geq s \log(1/\varepsilon) - C(s+1),  \]
where $C>0$ is a universal constant.
Therefore
\[ R(d,\mu,\varepsilon) \geq s \log (1/\varepsilon) - C(s+1), \]
\[  \underline{\rdim}(\mathcal{X},T,d,\mu) = \liminf_{\varepsilon \to 0} \frac{R(d,\mu,\varepsilon)}{\log(1/\varepsilon)} \geq s.\]
\end{proof}

For a dynamical system $(\mathcal{X},T)$ with a metric $d$,
we define the \textbf{$L^1$-mean Hausdorff dimension} by 
\[ \mdim_{\mathrm{H},L^1} (\mathcal{X},T,d) = 
    \lim_{\varepsilon \to 0} \left(\limsup_{N\to \infty} \frac{\dim_{\mathrm{H}}(\mathcal{X},\bar{d}_N,\varepsilon)}{N}\right). \]
Since $\bar{d}_N \leq d_N$, this always satisfies 
\[  \mdim_{\mathrm{H},L^1}(\mathcal{X},T,d) \leq \mdim_{\mathrm{H}}(\mathcal{X},T,d). \]
These two quantities actually coincide under the tame growth of covering numbers condition
(this is the only place where we use the tame growth of covering numbers 
in the proof of Theorem \ref{theorem: existence of nice measures}):

\begin{lemma}  \label{lemma: mean Hausdorff dimension and L^1 mean Hausdorff dimension}
If $(\mathcal{X},d)$ has the tame growth of covering numbers, then 
$\mdim_{\mathrm{H},L^1}(\mathcal{X},T,d) = \mdim_{\mathrm{H}}(\mathcal{X},T,d)$.
\end{lemma}

\begin{proof}
We assume $\mdim_{\mathrm{H},L^1}(\mathcal{X},T,d) < \infty$ and prove 
$\mdim_{\mathrm{H}}(\mathcal{X},T,d) \leq \mdim_{\mathrm{H},L^1}(\mathcal{X},T,d)$.
We set $[N] = \{0,1,2,\dots,N-1\}$ and define a metric $d_A$ on $\mathcal{X}$ for $A\subset [N]$
by $d_A(x,y) = \max_{a\in A} d(T^a x, T^a y)$. In particular $d_N = d_{[N]}$.

For $\tau>0$ we set $M(\tau) = \#(\mathcal{X},d,\tau)$.
We can find a covering $\mathcal{X} = W_1^\tau\cup\dots\cup W_{M(\tau)}^\tau$ with $\diam (W_m^\tau, d) < \tau$ 
for all $1\leq m\leq M(\tau)$.
Take any $0<\delta<1/2$ and $s>\mdim_{\mathrm{H},L^1}(\mathcal{X},T,d)$.
By the tame growth of covering numbers, we can choose $0<\varepsilon_0<1$ satisfying 
\begin{itemize}
   \item $\tau^\delta \log M(\tau) < 1$ for $0<\tau < \varepsilon_0$.   
   \item  $4\cdot 2^{s/(1-2\delta)} \cdot \varepsilon_0^{\delta s/(1-2\delta)} < 1$.
\end{itemize}

Let $0<\varepsilon <\varepsilon_0$.
Let $N$ be a sufficiently large number.
From $\mdim_{\mathrm{H},L^1}(\mathcal{X},T,d) < s$ we can find a covering 
$\mathcal{X} = \bigcup_{n=1}^\infty E_n$ satisfying $\tau_n :=\diam(E_n, \bar{d}_N) < \varepsilon$ for all $n$ and 
\[  \sum_{n=1}^\infty \tau_n^{sN} < 1. \]
Set $L_n = (1/\tau_n)^\delta$.
Pick a point $x_n$ from each $E_n$.
Every point $x\in E_n$ satisfies $\bar{d}_N(x,x_n) \leq \tau_n$ and hence 
\[  |\{k\in [N]|\, d(T^k x, T^k x_n) \geq L_n \tau_n\}| \leq \frac{N}{L_n}. \]
Thus there exists $A\subset [N]$ (depending on $x\in E_n$) satisfying 
$|A|\leq N/L_n$ and $d_{[N]\setminus A}(x,x_n) < L_n\tau_n$.
This implies 
\[ E_n \subset \bigcup_{A\subset [N], |A|\leq N/L_n}  B^{\circ}_{L_n \tau_n}(x_n, d_{[N]\setminus A}). \]
Here $B^{\circ}_{L_n \tau_n}(x_n, d_{[N]\setminus A})$ is the open ball of radius $L_n\tau_n$ with respect to $d_{[N]\setminus A}$
around $x_n$, which for $A=\{a_1,\dots,a_r\}$ we can write as 
\[ \bigcup_{1\leq i_1, \dots, i_r\leq M(\tau_n)} 
    B^{\circ}_{L_n \tau_n}(x_n, d_{[N]\setminus A}) \cap T^{-a_1}W_{i_1}^{\tau_n} \cap \dots \cap T^{-a_r} W^{\tau_n}_{i_r}. \]
Therefore $\mathcal{X}$ can be written as a union of 
\begin{equation}  \label{eq: decomposition by tame growth}
  B^{\circ}_{L_n \tau_n}(x_n, d_{[N]\setminus A}) \cap T^{-a_1}W_{i_1}^{\tau_n} \cap \dots \cap T^{-a_r} W^{\tau_n}_{i_r},
\end{equation}
where $n\geq 1$, $A=\{a_1,\dots,a_r\} \subset [N]$ with $r \leq N/L_n$ and 
$1\leq i_1,\dots, i_r\leq M(\tau_n)$.
The diameter of (\ref{eq: decomposition by tame growth}) with respect to $d_N$ is bounded by 
$2L_n\tau_n   = 2 \tau_n^{1-\delta} < 2 \varepsilon^{1-\delta}$.
Thus 
\begin{equation*}
   \mathcal{H}^{sN/(1-2\delta)}_{2\varepsilon^{1-\delta}}(\mathcal{X}, d_N)  \leq 
    \sum_{n=1}^\infty 2^N M(\tau_n)^{N/L_n} (2\tau_n^{1-\delta})^{sN/(1-2\delta)}.
\end{equation*}  
 Here $2^N$ comes from the choice of $A\subset [N]$. 
\[  2^N M(\tau_n)^{N/L_n} (2\tau_n^{1-\delta})^{sN/(1-2\delta)} = 
    \left\{2^{1+\frac{s}{1-2\delta}} M(\tau_n)^{\tau_n^\delta} \tau_n^{\frac{s\delta}{1-2\delta}}\right\}^N \tau_n^{sN}. \]
Recall $\tau_n < \varepsilon < \varepsilon_0$ and the choice of $\varepsilon_0$ above. We have 
\[ 2^{1+\frac{s}{1-2\delta}} M(\tau_n)^{\tau_n^\delta} \tau_n^{\frac{s\delta}{1-2\delta}}  
    < 4 \cdot 2^{\frac{s}{1-2\delta}} \varepsilon_0^{\frac{s\delta}{1-2\delta}} < 1. \]
Therefore
\[   \mathcal{H}^{sN/(1-2\delta)}_{2\varepsilon^{1-\delta}}(\mathcal{X}, d_N)  < \sum_{n=1}^\infty \tau_n^{sN} < 1. \]
So we get 
\[ \dim_{\mathrm{H}}(\mathcal{X},d_N,2\varepsilon^{1-\delta}) \leq \frac{sN}{1-2\delta}. \]
Since this holds for any sufficiently large $N$,
\[  \limsup_{N\to \infty} \left(\frac{1}{N}  \dim_{\mathrm{H}}(\mathcal{X},d_N,2\varepsilon^{1-\delta}) \right) \leq \frac{s}{1-2\delta}. \]
Here $0<\varepsilon <\varepsilon_0$ is arbitrary. Thus
\[  \mdim_{\mathrm{H}} (\mathcal{X},T,d) \leq \frac{s}{1-2\delta}. \]
Letting $\delta\to 0$ and $s\to \mdim_{\mathrm{H},L^1}(\mathcal{X},T,d)$, we get the statement.
\end{proof}

\begin{remark}
The same argument also proves that the lower mean Hausdorff dimension 
$\underline{\mdim}_{\mathrm{H}}(\mathcal{X},T,d)$ (see Remark \ref{remark: lower mean Hausdorff dimension}) coincides with
\[   \lim_{\varepsilon \to 0} \left(\liminf_{N\to \infty} \frac{\dim_{\mathrm{H}}(\mathcal{X},\bar{d}_N,\varepsilon)}{N}\right)  \]
if $(\mathcal{X},d)$ has the tame growth of covering numbers.
\end{remark}

\begin{remark}  \label{remark: metric mean dimension and tame growth of covering numbers}
In \cite[Lemmas 25 and 26]{Lindenstrauss--Tsukamoto rate distortion}
we showed a similar result for metric mean dimension.
In \cite[Lemma 25]{Lindenstrauss--Tsukamoto rate distortion} we proved the following statement:
Let $(\mathcal{X},T)$ be a dynamical system with a metric $d$.
For any integer $N\geq 1$ and real numbers $\varepsilon >0$ and $L>1$, we have
\[  \log \#(\mathcal{X}, \bar{d}_N,\varepsilon) \geq \log \#(\mathcal{X},d_N,2L\varepsilon)-N-\frac{N}{L}\log \#(\mathcal{X},d,\varepsilon).\]
This will be used in \S \ref{section: example: algebraic actions}.
\end{remark}

\subsection{Proof of Theorem \ref{theorem: existence of nice measures}}
\label{subsection: proof of existence of nice measures}

Theorem \ref{theorem: existence of nice measures} follows from 
Lemma \ref{lemma: mean Hausdorff dimension and L^1 mean Hausdorff dimension}
and the next theorem.

\begin{theorem} \label{theorem: L^1 version of existence of nice measures} 
For any dynamical system $(\mathcal{X},T)$ with a metric $d$
\[ \mdim_{\mathrm{H}, L^1} (\mathcal{X},T,d) \leq \sup_{\mu\in \mathscr{M}^T(\mathcal{X})} 
    \underline{\rdim}(\mathcal{X},T,d,\mu). \]
\end{theorem}

\noindent
Notice that here we do not assume the condition of the tame growth of covering numbers.    

\medskip

We use the following elementary lemma.

\begin{lemma}  \label{lemma: optimal transport}
Let $A$ be a finite set and $\{\mu_n\}$ a sequence of probability measures on $A$.
Suppose $\mu_n$ converges to some $\mu$ in the weak$^*$ topology\footnote{Since we assume that $A$ is a finite set,
this just means that $\mu_n(x)\to \mu(x)$ at each $x\in A$}.
Then there exists a sequence of probability measures $\pi_n$ on $A\times A$ such that 
\begin{itemize}
   \item $\pi_n$ is a coupling between $\mu_n$ and $\mu$, namely its first and second marginals are equal to $\mu_n$ and $\mu$ 
           respectively.
   \item $\pi_n$ converges to $(\mathrm{id}\times \mathrm{id})_*\mu$ in the weak$^*$ topology, namely 
   \[  \pi_n(a,b) \to \begin{cases}  \mu(a)  \quad & (\text{if } a=b) \\
                                            0          \quad & (\text{if } a\neq b)
                         \end{cases}. \]                    
\end{itemize}
\end{lemma}

\begin{proof}
This follows from a much more general fact on optimal transport 
that the Wasserstein distance metrizes the weak$^*$ topology
\cite[Theorem 6.9]{Villani}.
See \cite[Appendix]{Lindenstrauss--Tsukamoto rate distortion}
for an elementary self-contained proof.
\end{proof}

As in the proof of Lemma \ref{lemma: L^1 metric and rate distortion dimension},
we extend the definition of $\bar{d}_n$.
For $x= (x_0, \dots, x_{n-1})$ and $y=(y_0,\dots, y_{n-1})$ in $\mathcal{X}^n$ we set 
\[ \bar{d}_n(x,y) = \frac{1}{n} \sum_{k=0}^{n-1} d(x_k, y_k). \]

\begin{proof}[Proof of Theorem \ref{theorem: L^1 version of existence of nice measures}]
We can assume $\mdim_{\mathrm{H},L^1}(\mathcal{X},T,d) >0$.
Let $0<s<\mdim_{\mathrm{H},L^1}(\mathcal{X},T,d)$.
We will prove that there exists $\mu\in \mathscr{M}^T(\mathcal{X})$ satisfying 
$\underline{\rdim}(\mathcal{X},T,d,\mu) \geq s$.

Take $0<c<1$ with $c\cdot \mdim_{\mathrm{H},L^1}(\mathcal{X},T,d) > s$
and let $\delta_0 = \delta_0(c)\in (0,1)$ be the constant given by 
Corollary \ref{cor: quantitative Frostman's lemma}.
There exist $0<\delta<\delta_0$ and $n_1<n_2<n_3<\dots$ satisfying 
$c\cdot \dim_{\mathrm{H}}(\mathcal{X},\bar{d}_{n_k},\delta) > s n_k$.
By Corollary \ref{cor: quantitative Frostman's lemma} we can find Borel probability measures $\nu_k$ on $\mathcal{X}$
satisfying 
\begin{equation}  \label{eq: starting point: measure given by geometric measure theory}
  \nu_k(E) \leq \left(\diam(E,\bar{d}_{n_k})\right)^{s n_k},  \quad 
    \forall E \subset \mathcal{X} \text{ with } \diam (E,\bar{d}_{n_k}) < \frac{\delta}{6}.
\end{equation}
Set 
\[  \mu_k = \frac{1}{n_k} \sum_{n=0}^{n_k-1} T^n_*\nu_k. \]
By choosing a subsequence (also denoted by $\mu_k$), we can assume that 
$\mu_k$ converges to some $T$-invariant probability measure $\mu$ on $\mathcal{X}$
in the weak$^*$ topology.
We will show $\underline{\rdim}(\mathcal{X},T,d,\mu) \geq s$.

Let $\varepsilon$ be any positive number with $2\varepsilon \log(1/\varepsilon) \leq \delta/10$.
We would like to prove an estimate such as
\[  R(d,\mu,\varepsilon) \geq s \log(1/\varepsilon) + \text{small error terms}. \]
For this purpose, let $X$ and $Y=(Y_0,\dots, Y_{m-1})$ be coupled random variables such that 
$X, Y_0, \dots, Y_{m-1}$ take values in $\mathcal{X}$ and satisfy 
\[   \mathrm{Law} X = \mu, \quad  \mathbb{E} \left(\frac{1}{m} \sum_{j=0}^{m-1} d(T^j X, Y_j)\right) < \varepsilon. \]
We need to show $I(X;Y)\geq sm \log(1/\varepsilon) + \text{small error terms}$.
As we remarked in Remark \ref{remark: Y takes only finitely many values in the definition of R}, 
we can assume that $Y$ takes only finitely many values.
Let $\mathcal{Y}\subset \mathcal{X}^m$ be the set of possible values of $Y$.

\vspace{0.3cm}
\noindent
\textbf{Idea of the proof.}
Here we roughly explain the idea of the proof, assuming $m=1$. That is to say, we give a lower bound on $I(X;Y)$ if $\mathrm{Law} X = \mu$ and $Y$ is any $\mathcal{X}$-valued random variable coupled to it with $ \mathbb{E} d(X, Y)<\varepsilon$.

Since $\mu_k\to \mu$, we can find random variables $X(k)$ coupled to $X$ such that
$X(k)$ take values in $\mathcal{X}$ with $\mathrm{Law} X(k) = \mu_k$ and 
\begin{equation} \label{eq: condition on X(k) in toy model}
  \mathbb{E} d(X, X(k)) \to 0, \quad I(X(k); Y) \to I(X;Y).
\end{equation}
(Here we have ignored a technical problem. The above convergence $I(X(k);Y) \to I(X;Y)$ essentially follows from Lemma 
\ref{lemma: law convergence implies the convergence of mutual information}.
But Lemma \ref{lemma: law convergence implies the convergence of mutual information} is valid 
only if the underlying space is a finite set.
We will address this issue by introducing an appropriate finite partition $\mathcal{P}$ below.
At the moment we pretend that $X$ and $X(k)$ take values in a finite set.)

Using this coupling between $X(k)$ and $Y$, we will construct a new coupling between $Z(k) = (X'(k), TX'(k), \dots, T^{n_k-1}X'(k))$ with $X'(k)$ a random variable taking values in $\mathcal{X}$ and obeying $\nu_k$, and a new random variable $W(k)= (W(k)_0, \dots, W(k)_{n_k-1})$ so that 
\[ \mathbb{E} \bar{d}_{n_k}(Z(k), W(k))<\varepsilon\] 
and $I(Z(k); W(k))$ is bounded from above using $I(X;Y)$. As we can bound $I(Z(k); W(k))$ from bellow using Lemma~\ref{lemma: from geometric measure theory to rate distortion theory} this will give us the desired lower bound on $I(X;Y)$.

We define the random variable $W(k)= (W(k)_0, \dots, W(k)_{n_k-1})$, coupled to $Z(k)$ and taking values in $\mathcal{X}^{n_k}$, as follows. Let $\rho_k(y|x) = \mathbb{P}(Y=y|X(k)=x)$. Then the law of $W(k)= (W(k)_0, \dots, W(k)_{n_k-1})$ is determined by requiring that 
$W(k)_0, \dots, W(k)_{n_k-1}$ are conditionally independent given $Z(k)$ and 
\[  \mathbb{P}(W(k)_n = y| T^n X'(k) = x) = \rho_k(y|x). \]

By (\ref{eq: condition on X(k) in toy model}) we have that
\[  \mathbb{E} \bar{d}_{n_k}(Z(k), W(k)) = \mathbb{E} d(X(k), Y) \to \mathbb{E} d(X,Y) < \varepsilon, \]
so that indeed $\mathbb{E} \bar{d}_{n_k}(Z(k), W(k)) < \varepsilon$ for large $k$.

Since $W(k)_0, \dots, W(k)_{n_k-1}$ are conditionally independent given $Z(k)$, we use
the subadditivity of mutual information (Lemma \ref{lemma: subadditivity of mutual information})
and get
\[   I(Z(k);W(k)) \leq \sum_{n=0}^{n_k-1} I(Z(k); W(k)_n) = \sum_{n=0}^{n_k-1} I(T^n X'(k); W(k)_n). \]
By definition of $W(k)_n$, we have $I(T^n X'(k); W(k)_n) = I(T^n_*\nu_k, \rho_k)$ in the notation of Lemma~\ref{lemma: concavity/convexity of mutual information}.
Recall $\mu_k = (1/n_k) \sum_{n=0}^{n_k-1} T^n_*\nu_k$.
By using the concavity property of mutual information given in Lemma \ref{lemma: concavity/convexity of mutual information} it follows that
\[  \frac{1}{n_k}\sum_{n=0}^{n_k-1} I(T^n_* \nu_k, \rho_k) \leq I(\mu_k, \rho_k) = I(X(k); Y). \]
 Thus we get 
\begin{equation}\label{equation: mutual information of Z,W}
  I(Z(k);W(k)) \leq n_k I(X(k); Y)  
\end{equation} 
 Since $\mathrm{Law} X'(k) = \nu_k$, it follows from (\ref{eq: starting point: measure given by geometric measure theory}) and 
Lemma \ref{lemma: from geometric measure theory to rate distortion theory} that 
\[  I(Z(k);W(k)) \geq s n_k \log(1/\varepsilon) + \text{small error terms}, \]
hence by \eqref{equation: mutual information of Z,W} we see that
\[ I(X(k);Y) \geq s  \log(1/\varepsilon) + \text{small error terms}. \]
As $I(X(k);Y) \to I(X;Y)$ by (\ref{eq: condition on X(k) in toy model}) 
\[  I(X;Y) \geq s \log (1/\varepsilon) + \text{small error terms}. \]
This is what we want to prove.
$\blacksquare$
\vspace{0.3cm}

Now we return to the proof.
Recall our situation: The random variables $X$ and $Y=(Y_0,\dots, Y_{m-1})$ 
take values in $\mathcal{X}$ and a finite set $\mathcal{Y} \subset \mathcal{X}^m$ respectively. 
They satisfy $\mathrm{Law} X= \mu$ and 
\[  \frac{1}{m}\mathbb{E} \left(\sum_{j=0}^{m-1} d(T^j X, Y_j)\right) < \varepsilon.\]
We want to estimate $I(X;Y)$ from below.

\medskip

We choose a positive number $\tau$ satisfying 
\begin{equation} \label{eq: choice of tau in the proof of the existence of nice measures}
   \tau \leq \min\left(\frac{\varepsilon}{3}, \frac{\delta}{20}\right), \quad 
   \frac{\tau}{2} + \mathbb{E}\left(\frac{1}{m}\sum_{j=0}^{m-1} d(T^j X, Y_j)\right) < \varepsilon.
\end{equation}
We take a finite partition $\mathcal{P}  =\{P_1,\dots,P_L\}$ of $\mathcal{X}$ such that for all $1\leq l\leq L$
\begin{equation*}
     \diam(P_l,d) < \frac{\tau}{2}, \quad \mu(\partial P_l) = 0. 
\end{equation*}     
Pick a point $p_l$ from each $P_l$ and set $A=\{p_1,\dots,p_L\}$.
We define a map $\mathcal{P}:\mathcal{X}\to A$ by 
$\mathcal{P}(P_l) = \{p_l\}$.
For $n\geq 1$ we define a map $\mathcal{P}^n: \mathcal{X} \to A^n$ by 
$\mathcal{P}^n(x) = \left(\mathcal{P}(x), \mathcal{P}(Tx), \dots, \mathcal{P}(T^{n-1}x)\right)$.

\begin{claim}  \label{claim: property of pushforward of nu_k}
$\mathcal{P}^{n_k}_* \nu_k(E) \leq \left(\tau + \diam \left(E, \bar{d}_{n_k}\right)\right)^{s n_k}$
for all $E\subset A^{n_k}$ with $\diam \left(E, \bar{d}_{n_k}\right) < \delta/10$.
\end{claim}

\begin{proof}
We have $\mathcal{P}^{n_k}_* \nu_k(E) = \nu_k\left(\left(\mathcal{P}^{n_k}\right)^{-1}E\right)$.
Since $\tau \leq \delta/20$
\[  \diam \left(\left(\mathcal{P}^{n_k}\right)^{-1}E, \bar{d}_{n_k}\right) < \tau + \diam\left(E,\bar{d}_{n_k}\right)
     < \frac{\delta}{6}. \]
Then by (\ref{eq: starting point: measure given by geometric measure theory})
 \begin{equation*}
   \begin{split}
       \nu_k\left(\left(\mathcal{P}^{n_k}\right)^{-1}E\right) &\leq 
       \left( \diam \left(\left(\mathcal{P}^{n_k}\right)^{-1}E, \bar{d}_{n_k}\right)\right)^{s n_k} \\
      & <  \left(\tau + \diam\left(E,\bar{d}_{n_k}\right)\right)^{sn_k}. 
   \end{split} 
 \end{equation*} 
\end{proof}

It follows from $\mu_k\to \mu$ and $\mu(\partial P_l)=0$ that 
$\mathcal{P}^{m}_* \mu_k \to \mathcal{P}^m_*\mu$ as $k\to \infty$.
Then by Lemma \ref{lemma: optimal transport} there exists a sequence of couplings $\pi_k$ between $\mathcal{P}^{m}_* \mu_k$
and $\mathcal{P}^m_*\mu$ converging to $(\mathrm{id}\times \mathrm{id})_* \mathcal{P}^m_*\mu$.
We take a random variable $X(k)$ coupled to $\mathcal{P}^m(X)$ such that 
$X(k)$ takes values in $A^m$ with $\mathrm{Law}\left(X(k),\mathcal{P}^m(X)\right) = \pi_k$. 
(In particular $X(k)$ obeys $\mathcal{P}^m_*\mu_k$.)
This satisfies 
\begin{equation}  \label{eq: choice of X(k)}
   \mathbb{E} \,\bar{d}_m\left(X(k), \mathcal{P}^m(X)\right) \to 0, \quad 
   I\left(X(k); Y\right) \to I\left(\mathcal{P}^m(X); Y\right).
\end{equation}
The latter condition\footnote{The coupling between $X(k)$ and $Y$ is given by the probability mass function 
\[  \sum_{x'\in A^m} \pi_k(x,x') \mathbb{P}(Y=y|\mathcal{P}^m(X)=x'), \]
which converges to $\mathbb{P}(\mathcal{P}^m(X)=x, Y=y)$.}
 follows from Lemma \ref{lemma: law convergence implies the convergence of mutual information}.
Since ${\diam (\mathcal{P}_l, d) < \tau/2}$,
\begin{equation}  \label{eq: distortion between X(k) and Y}
  \begin{split}
    \mathbb{E} \bar{d}_m\left(X(k), Y\right) & <  \mathbb{E} \bar{d}_m\left(X(k), \mathcal{P}^m(X)\right) + 
       \frac{\tau}{2} + \mathbb{E}\left(\frac{1}{m}\sum_{j=0}^{m-1} d(T^j X, Y_j)\right) \\
     & \to  \frac{\tau}{2} + \mathbb{E}\left(\frac{1}{m}\sum_{j=0}^{m-1} d(T^j X, Y_j)\right)  < \varepsilon \quad
     \text{by (\ref{eq: choice of tau in the proof of the existence of nice measures})}.
   \end{split}
\end{equation}

For $x = (x_0, \dots, x_{n-1})\in \mathcal{X}^n$ and $0\leq a\leq b<n$ we denote 
$x_a^b = (x_a, x_{a+1}, \dots, x_b)$.
For $x,y\in \mathcal{X}^m$ with $\mathcal{P}^m_*\mu_k(x)> 0$ we consider a conditional probability mass function 
\[ \rho_k(y|x) = \mathbb{P}(Y=y|X(k) = x). \]
Fix a point $a\in \mathcal{X}$.
Let $n_k = mq + r$ with $m\leq r < 2m$.
For $x, y\in \mathcal{X}^{n_k}$ and $0\leq j < m$ we define a conditional probability mass function 
\begin{equation*}
   \sigma_{k, j}(y|x) = \prod_{i=0}^{q-1} \rho_k\left( y_{j+im}^{j+im+m-1}| x_{j+im}^{j+im+m-1}\right)  \cdot 
                            \prod_{n\in [0,j)\cup [mq+j, n_k)} \delta_a(y_n).
\end{equation*}
Here $\delta_a(\cdot)$ is the delta probability measure at $a$ on $\mathcal{X}$.
We set 
\[  \sigma_k(y|x) = \frac{\sigma_{k,0}(y|x) + \sigma_{k,1}(y|x) + \dots+ \sigma_{k,m-1}(y|x)}{m}. \]
This is defined for $x\in \mathcal{X}^{n_k}$ with $\mathcal{P}^{n_k}_*\nu_k(x) > 0$.

Let $X'(k)$ be a random variable taking values in $\mathcal{X}$ and obeying $\nu_k$.
We set $Z(k) = \mathcal{P}^{n_k}\left(X'(k)\right)$.
We take a random variable $W(k)$ coupled to $Z(k)$ and taking values in $\mathcal{X}^{n_k}$ with 
\[  \mathbb{P}\left(W(k) = y| Z(k) = x\right) = \sigma_k (y|x). \]
For $0\leq j < m$ we also take a random variable $W(k,j)$ coupled to $Z(k)$ and taking values in $\mathcal{X}^{n_k}$ with 
\[  \mathbb{P}\left(W(k,j) = y| Z(k) = x\right) = \sigma_{k,j} (y|x). \]

\begin{claim}  \label{claim: distortion between Z(k) and W(k)}
 $\mathbb{E}\, \bar{d}_{n_k}\left(Z(k), W(k)\right) < \varepsilon$ for large $k$.
\end{claim}

\begin{proof}
\[  \mathbb{E} \bar{d}_{n_k}\left(Z(k),W(k)\right) 
    = \frac{1}{m}\sum_{j=0}^{m-1} \mathbb{E} \bar{d}_{n_k}\left(Z(k), W(k,j)\right). \]
$\bar{d}_{n_k}\left(Z(k), W(k,j)\right)$ is bounded by 
\[ \frac{r\cdot \diam(\mathcal{X},d)}{n_k} + \frac{m}{n_k} \sum_{i=0}^{q-1}
    \bar{d}_m\left(\mathcal{P}^m(T^{j+im}X'(k)), W(k,j)_{j+im}^{j+im+m-1}\right). \]
$\mathbb{E} \bar{d}_m\left(\mathcal{P}^m(T^{j+im}X'(k)), W(k,j)_{j+im}^{j+im+m-1}\right)$
is equal to 
\begin{equation*}
 \sum_{x,y\in \mathcal{X}^m} \bar{d}_m(x,y) \rho_k(y|x) \mathcal{P}^m_*T^{j+im}_*\nu_k(x). 
\end{equation*}   
Therefore $\mathbb{E} \bar{d}_{n_k}\left(Z(k),W(k)\right)$ is bounded by 
\begin{equation*}
   \begin{split}
    & \frac{r\cdot \diam(\mathcal{X},d)}{n_k} 
     + \sum_{x,y\in \mathcal{X}^m} \bar{d}_m(x,y) \rho_k(y|x)
        \left(\frac{1}{n_k}\sum_{\substack{0\leq i <q \\ 0\leq j < m}} \mathcal{P}^m_*T^{j+im}_*\nu_k(x) \right) \\
        \leq & 
         \frac{r\cdot \diam(\mathcal{X},d)}{n_k} 
     + \sum_{x,y\in \mathcal{X}^m} \bar{d}_m(x,y) \rho_k(y|x)
        \left(\frac{1}{n_k}\sum_{n=0}^{n_k-1} \mathcal{P}^m_*T^{n}_*\nu_k(x) \right) \\
    = & 
         \frac{r\cdot \diam(\mathcal{X},d)}{n_k} 
     + \sum_{x,y\in \mathcal{X}^m} \bar{d}_m(x,y) \rho_k(y|x) \mathcal{P}^m_*\mu_k(x) \\
    = & \frac{r\cdot \diam(\mathcal{X},d)}{n_k} + \mathbb{E} \bar{d}_m\left(X(k), Y\right).
   \end{split}
\end{equation*}   
From $r\leq 2m$ and (\ref{eq: distortion between X(k) and Y}), this is less than $\varepsilon$ for large $k$.
\end{proof}

\begin{claim}  \label{claim: mutual information between Z(k) and W(k)}
\[   \frac{1}{n_k} I\left(Z(k); W(k)\right) \leq \frac{1}{m} I\left(X(k); Y\right). \]
\end{claim}

\begin{proof}
By the convexity of mutual information (Lemma \ref{lemma: concavity/convexity of mutual information})
\[  I\left(Z(k);W(k)\right) \leq \frac{1}{m}\sum_{j=0}^{m-1} I\left(Z(k); W(k,j)\right). \]
By the subadditivity of mutual information under conditional independence (Lemma \ref{lemma: subadditivity of mutual information})
\[ I\left(Z(k);W(k,j)\right) \leq \sum_{i=0}^{q-1} I\left(Z(k); W(k,j)_{j+im}^{j+im+m-1}\right). \]
The term $I\left(Z(k); W(k,j)_{j+im}^{j+im+m-1}\right)$ is equal to 
\[  I\left(\mathcal{P}^m(T^{j+im} X'(k)); W(k,j)_{j+im}^{j+im+m-1}\right) = I\left(\mathcal{P}^m_*T^{j+im}_*\nu_k, \rho_k\right). \]
Hence 
\begin{equation*}
   \begin{split}
   \frac{m}{n_k} I\left(Z(k);W(k)\right)  \leq &
     \frac{1}{n_k} \sum_{\substack{0\leq j < m\\ 0\leq i<q}} I\left(\mathcal{P}^m_*(T^{j+im}_*\nu_k), \rho_k\right) \\
     \leq &  \frac{1}{n_k} \sum_{n=0}^{n_k-1} I\left(\mathcal{P}^m_*T^n_*\nu_k, \rho_k\right) \\
     \leq &  I\left(\frac{1}{n_k}\sum_{n=0}^{n_k-1}\mathcal{P}^m_*T^n_*\nu_k, \rho_k\right) \\
     &  \text{ by the concavity in Lemma \ref{lemma: concavity/convexity of mutual information}} \\
      =  & I\left(\mathcal{P}^m_*\mu_k, \rho_k\right) \\
      = &  I\left(X(k); Y\right).
   \end{split}
\end{equation*}   
\end{proof}

Recall $2\varepsilon \log (1/\varepsilon) \leq \delta/10$, $\tau \leq \min(\varepsilon/3,\delta/20)$ and 
$\mathrm{Law} Z(k) = \mathcal{P}^{n_k}_*\nu_k$.
We apply Lemma \ref{lemma: from geometric measure theory to rate distortion theory} with Claims 
\ref{claim: property of pushforward of nu_k} and
\ref{claim: distortion between Z(k) and W(k)} 
to $\left(Z(k), W(k)\right)$:
\[  I\left(Z(k); W(k)\right) \geq s n_k \log(1/\varepsilon) - C(sn_k+1) \quad 
     \text{for large $k$}. \]
Here $C$ is a universal positive constant.
From Claim \ref{claim: mutual information between Z(k) and W(k)}
\[  \frac{1}{m} I\left(X(k); Y\right) \geq s \log(1/\varepsilon) - C\left(s + \frac{1}{n_k}\right). \]
We know $I\left(X(k); Y\right) \to I\left(\mathcal{P}^m(X); Y\right)$ as $k\to \infty$ in (\ref{eq: choice of X(k)}).
Hence 
\[  \frac{1}{m} I\left(\mathcal{P}^m(X);Y\right) \geq s \log(1/\varepsilon) - Cs. \]
By the data-processing inequality (Lemma \ref{lemma: data-processing inequality}) 
\[ \frac{1}{m} I(X;Y) \geq  \frac{1}{m} I\left(\mathcal{P}^m(X);Y\right) \geq s \log(1/\varepsilon) - Cs. \]
Therefore for any positive number $\varepsilon$ with $2\varepsilon \log(1/\varepsilon) \leq \delta/10$
\[  R(d,\mu,\varepsilon) \geq s \log(1/\varepsilon) - C s. \] 
Thus
\[  \underline{\rdim}(\mathcal{X},T,d,\mu) = \liminf_{\varepsilon \to 0} \frac{R(d,\mu,\varepsilon)}{\log(1/\varepsilon)}
     \geq s. \]
\end{proof}

\section{Proof of Theorem \ref{theorem: existence of nice metrics}}
\label{section: proof of existence of nice metrics}

The purpose of this section is to prove Theorem \ref{theorem: existence of nice metrics}.
The proof does not involve rate distortion theory, and in particular 
it is independent of \S \ref{section: proof of existence of nice measures}.

\subsection{Background: Pontrjagin--Schnirelmann's theorem}  \label{subsection: background: PS theorem}

Theorems \ref{theorem: lower metric mean dimension coincides with mean dimension} and \ref{theorem: existence of nice metrics} 
look quite similar.
Theorem \ref{theorem: lower metric mean dimension coincides with mean dimension} claims the existence of $d$
satisfying $\underline{\mdim}_{\mathrm{M}}(\mathcal{X},T,d) = \mdim(\mathcal{X},T)$ wheres 
Theorem \ref{theorem: existence of nice metrics} claims the existence of $d$ satisfying 
$\overline{\mdim}_{\mathrm{M}}(\mathcal{X},T,d) = \mdim(\mathcal{X},T)$.
But indeed the upper case (Theorem \ref{theorem: existence of nice metrics}) is substantially subtler.
The difficulty is already visible in a classical, non-dynamical setting.
Let $\mathcal{X}$ be a compact metrizable space.
Pontrjagin--Schnirelmann \cite{Pontrjagin--Schnirelmann} proved\footnote{Indeed we \textit{cannot} find this statement in 
\cite{Pontrjagin--Schnirelmann}. The main theorem of their paper states that
$\dim \mathcal{X}$ is equal to the infimum of $\underline{\dim}_{\mathrm{M}}(\mathcal{X},d)$ over $d\in \mathscr{D}(\mathcal{X})$.
But their argument actually proves Theorem \ref{theorem: PS theorem}.}

\begin{theorem} \label{theorem: PS theorem} 
There exists a metric $d$ on $\mathcal{X}$ satisfying 
$\overline{\dim}_{\mathrm{M}}(\mathcal{X},d) = \dim \mathcal{X}$.
\end{theorem}

Compare this statement with 

\begin{theorem} \label{theorem: lower PS theorem}
There exists a metric $d$ on $\mathcal{X}$ satisfying 
$\underline{\dim}_{\mathrm{M}}(\mathcal{X},d) = \dim \mathcal{X}$. 
\end{theorem}

They look similar. But their natures are different.
A (now) standard approach to Theorem \ref{theorem: lower PS theorem} is to use the Baire category theorem as follows.
Fix an arbitrary metric $d$ on $\mathcal{X}$ and consider an infinite dimensional Banach space $(V,\norm{\cdot})$.
We denote by $C(\mathcal{X},V)$ the space of continuous maps from $\mathcal{X}$ to $V$ endowed with the norm topology.
For each $n\geq 1$ we consider $A_n\subset C(\mathcal{X},V)$ consisting of $f:\mathcal{X}\to V$ such that $f$ is a $(1/n)$-embedding with 
respect to $d$ and satisfies 
\[  \exists \varepsilon < 1/n:  \frac{\log \#(f(\mathcal{X}),\norm{\cdot},\varepsilon)}{\log (1/\varepsilon)} < \dim \mathcal{X} + \frac{1}{n}. \]
It is not hard to show that $A_n$ are open and dense\footnote{``Open'' is easy.
To show ``dense'', take arbitrary $f\in C(\mathcal{X},V)$ and $\delta>0$.
Choose $0<\varepsilon<1/n$ such that $d(x,y)<\varepsilon$ implies $\norm{f(x)-f(y)}<\delta$.
There exists an $\varepsilon$-embedding $\pi:\mathcal{X}\to P$
in a simplicial complex $P$ of dimension $\leq \dim \mathcal{X}$.  
From Lemma \ref{lemma: preparations on linear maps} (2) and (3) in \S \ref{subsection: preparation on combinatorial topology}
we can find a linear embedding $g:P\to V$ with $\norm{g(\pi(x))-f(x)} < \delta$.
From Lemma \ref{lemma: preparations on linear maps} (1), $\log \#(g(P), \norm{\cdot},\varepsilon')/\log (1/\varepsilon')$ is less
than $\dim \mathcal{X} + 1/n$ for sufficiently small $\varepsilon'$.
This shows $g\circ \pi\in A_n$. So $A_n$ is dense.
Therefore the main point of the proof of Theorem \ref{theorem: lower PS theorem} is a ``polyhedral approximation''.
The basic idea of the proof of Theorem \ref{theorem: PS theorem} is also a polyhedral approximation, but in a \textit{much more accurate way}.
See \S \ref{subsection: warmup: the proof of PS theorem}.}.
Then $\bigcap_{n=1}^\infty A_n$ is a residual (i.e. dense and $G_\delta$) subset of $C(\mathcal{X},V)$ (in particular, non-empty).
On the other hand this is equal to 
\[ \left\{f\in C(\mathcal{X},V)\middle|\, \text{$f$ is an embedding and }
           \underline{\dim}_{\mathrm{M}}(f(\mathcal{X}), \norm{\cdot}) = \dim \mathcal{X} \right\}. \]
Pick $f$ in this set. Then the metric $\norm{f(x)-f(y)}$ $(x,y\in \mathcal{X})$ has the lower Minkowski dimension 
equal to $\dim \mathcal{X}$.
This proves
Theorem \ref{theorem: lower PS theorem}.

Let's try a similar approach to Theorem \ref{theorem: PS theorem}. 
It is natural to consider
\begin{equation} \label{eq: nice set for PS theorem}
    \left\{f\in C(\mathcal{X},V)\middle|\, \text{$f$ is an embedding and }
           \overline{\dim}_{\mathrm{M}}(f(\mathcal{X}), \norm{\cdot}) = \dim \mathcal{X} \right\}. 
\end{equation}           
One might hope that this is also a residual subset of  $C(\mathcal{X},V)$.
But this does not hold true.
We can prove that if $\mathcal{X}$ is an infinite set then 
\[ \left\{f\in C(\mathcal{X},V)\middle|\, \overline{\dim}_{\mathrm{M}}(f(\mathcal{X}), \norm{\cdot}) = \infty \right\} \]
is a residual subset of $C(\mathcal{X},V)$.
This implies that the set (\ref{eq: nice set for PS theorem}) is never residual
in a nontrivial situation (namely if $\mathcal{X}$ is infinite and $\dim\mathcal{X} < \infty$).
It is a very thin set.
So it is more delicate to find an element in (\ref{eq: nice set for PS theorem}).
We note that the set (\ref{eq: nice set for PS theorem}) is however dense; see \S \ref{subsection: warmup: the proof of PS theorem}.

\subsection{Preparations on combinatorial topology}  \label{subsection: preparation on combinatorial topology}

As we promised in \S \ref{subsection: topological and metric mean dimensions},
all simplicial complexes are assumed to have only finitely many vertices.

Let $P$ be a simplicial complex. We denote the set of vertices of $P$ by $\ver(P)$.
For $v\in \ver(P)$ we define the \textbf{open star} $O_P(v)$ as the union of
the open simplices of $P$ one of whose vertex is $v$.
(We declare that $\{v\}$ itself is an open simplex.)
The open star $O_P(v)$ is an open neighborhood of $v$ and 
$\{O_P(v)\}_{v\in \ver(P)}$ is an open cover of $P$.
When vertices $v_0,\dots, v_n$ span a simplex $\Delta$ in $P$, we set 
$O_P(\Delta) = \bigcup_{i=0}^n O_P(v_i)$.

Let $P$ and $Q$ be simplicial complexes.
A map $f:P\to Q$ is said to be \textbf{simplicial} if it satisfies the following two conditions:
\begin{itemize}
   \item For every simplex $\Delta\subset P$, $f(\Delta)$ is a simplex in $Q$.
           (In particular $f(v)\in \ver(Q)$ for every $v\in \ver(P)$.)
   \item If $v_0,\dots, v_n \in \ver(P)$ span a simplex in $P$ then 
           \[  f\left(\sum_{i=0}^n \lambda_i v_i\right) = \sum_{i=0}^n \lambda_i f(v_i), \]
           where $0\leq \lambda_i\leq 1$ and $\sum_{i=0}^n \lambda_i = 1$.
\end{itemize} 

Let $V$ be a vector space over real numbers.
A map $f:P\to V$ is said to be \textbf{linear} if for any $v_0,\dots,v_n \in \ver(P)$ spanning a simplex in $P$ 
\[  f\left(\sum_{i=0}^n \lambda_i v_i\right) = \sum_{i=0}^n \lambda_i f(v_i), \]
  where $0\leq \lambda_i\leq 1$ and $\sum_{i=0}^n \lambda_i = 1$.
We denote the space of linear maps $f:P\to V$ by $\mathrm{Hom}(P,V)$.

\begin{lemma}  \label{lemma: preparations on linear maps}
Let $(V,\norm{\cdot})$ be a Banach space and $P$ a simplicial complex.
  \begin{enumerate}
     \item If $f:P\to V$ is a linear map with $\diam f(P) \leq 2$, then for any $0<\varepsilon \leq 1$
     \[  \#\left(f(P),\norm{\cdot},\varepsilon\right) \leq C(P) \cdot (1/\varepsilon)^{\dim P}, \]
      where $C(P)$ is a positive constant depending only on $\dim P$ and the number of simplices of $P$.
     \item Suppose $V$ is infinite dimensional. 
             Then the set
             \begin{equation} \label{eq: the set of linear embeddings}
                  \{f\in \mathrm{Hom}(P,V)|\, \text{$f$ is injective} \} 
             \end{equation}     
             is dense in $\mathrm{Hom}(P,V)$. (Indeed it is also open; but we do not need this.)
             Here $\mathrm{Hom}(P,V) \cong V^{\ver(P)}$ is endowed with the product topology.
     \item Let $(\mathcal{X},d)$ be a compact metric space and $\varepsilon, \delta >0$.
             Let $\pi:\mathcal{X}\to P$ be a continuous map satisfying $\diam\, \pi^{-1}(O_P(v)) < \varepsilon$ for all $v\in \ver(P)$. 
             Let $f:\mathcal{X}\to V$ be a continuous map 
             such that
             \[  d(x,y) < \varepsilon \Longrightarrow  \norm{f(x)-f(y)} < \delta. \]
             Then there exists a linear map $g:P\to V$ satisfying 
             \[  \norm{f(x)-g(\pi(x))} < \delta  \quad (x\in \mathcal{X}). \]   
             Moreover if $f(\mathcal{X})\subset B_1^\circ (V)$ (the open unit ball) then it can be chosen so that 
             $g(P) \subset B_1^\circ (V)$.
   \end{enumerate}
\end{lemma}

\begin{proof}
(1) We can assume that $P$ is a simplex (we denote its vertices by $v_0, \dots, v_n$)
and that $f(v_0)=0$ and $\norm{f(v_i)}\leq 2$.
Set 
\[  \Delta = \{(\lambda_1,\dots,\lambda_n)\in [0,1]^n|\, \lambda_1+\dots+\lambda_n\leq 1\}. \]
Then $f(P)$ is covered by the $(\varepsilon/3)$-open balls 
\[  B^\circ_{\varepsilon/3} \left(\lambda_1 f(v_1)+\dots+ \lambda_n f(v_n)\right) \]
where 
\[   (\lambda_1,\dots,\lambda_n)\in \Delta \cap \left(\frac{\varepsilon}{6n}\mathbb{Z}\right)^n.  \]

(2)   Let $\ver(P) = \{v_0,\dots,v_n\}$.
The set (\ref{eq: the set of linear embeddings}) contains 
\[   \left\{f\in \mathrm{Hom}(P,V)\middle|\, \text{$f(v_0), \dots, f(v_n)$ are affinely independent}\right\}, \]
which is dense because $V$ is infinite dimensional.

(3) Let $v\in \ver(P)$.
Pick $x_v\in \pi^{-1}(O_P(v))$ and set $g(v) = f(x_v)$.
If $\pi^{-1}(O_P(v))=\emptyset$ then $g(v)$ may be an arbitrary point in $B_1^\circ(V)$.
We extend $g$ to a linear map $g:P\to V$.
Let $x\in \mathcal{X}$ and $\Delta^\circ$ be the open simplex of $P$ containing $\pi(x)$.      
Let $v_0,\dots,v_n$ be the vertices of $\Delta^\circ$.
Then $\pi(x) = \sum_{i=0}^n \lambda_i v_i$ with $0<\lambda_i\leq 1$ and $\sum_{i=0}^n \lambda_i =1$.
$g(\pi(x)) = \sum_{i=0}^n \lambda_i f(x_{v_i})$.
   
Since $\pi(x) \in O_P(v_i)$, $x\in \pi^{-1}(O_P(v_i))$ and hence $d(x,x_{v_i}) < \varepsilon$.
Then $\norm{f(x)-f(x_{v_i})} < \delta$. It follows that 
\[  \norm{f(x)-g(\pi(x))} \leq \sum_{i=0}^n \lambda_i \norm{f(x)- f(x_{v_i})} < \delta. \]      
If $f(\mathcal{X})\subset B_1^\circ(V)$ then $g(v)\in B_1^\circ (V)$ for all $v\in \ver(P)$ and hence 
$g(P)\subset B_1^\circ(V)$.     
\end{proof}

Let $f:\mathcal{X}\to P$ be a continuous map from a topological space $\mathcal{X}$ to a simplicial complex $P$.
It is said to be \textbf{essential} if for any $v_0,\dots, v_n\in \ver(P)$ spanning a simplex in $P$ 
\[  f^{-1}\left(O_P(v_0)\cap \dots \cap O_P(v_n)\right) \neq \emptyset. \]

\begin{lemma} \label{lemma: essential map}
Let $f:\mathcal{X}\to P$ be a continuous map from a topological space $\mathcal{X}$ to a simplicial complex $P$.
There exists a subcomplex $P'\subset P$ such that $f(\mathcal{X})\subset P'$ and $f:\mathcal{X}\to P'$ is 
essential.
\end{lemma}

\begin{proof}
It is easy to check that $f:\mathcal{X}\to P$ is essential if and only if $f(\mathcal{X})\not \subset P'$ for any proper subcomplex 
$P'\subset P$. 
Then the statement is trivial; just take the minimal subcomplex $P'\subset P$ containing $f(\mathcal{X})$.
\end{proof}

The next lemma is one of the central ingredients of the proof Theorem \ref{theorem: existence of nice metrics}.
It has a spirit similar to Lemma \ref{lemma: preparations on linear maps} (3), though its
statement is rather technical (see Corollary \ref{corollary: approximately commuting simplicial maps baby version}
for a simplified version).
Its rough idea is as follows: Let $\mathcal{X}$ be a topological space and $P, Q$ simplicial complexes.
Let $\pi:\mathcal{X} \to P$ and $q:\mathcal{X}\to Q$ be continuous maps.
We would like to formulate a condition which guarantees the existence of a simplicial map $h:P\to Q$
such that 
$h\circ \pi$ is approximately equal to $q$.

For two open covers $\mathcal{U}$ and $\mathcal{V}$ of $\mathcal{X}$, we denote by $\mathcal{U}\prec \mathcal{V}$
if $\mathcal{V}$ is a refinement of $\mathcal{U}$, namely for every $V\in \mathcal{V}$ there exists $U\in \mathcal{U}$
satisfying $V\subset U$.

\begin{lemma} \label{lemma: approximately commuting simplicial maps}
Let $\mathcal{X}$ be a topological space and $P, Q$ simplicial complexes.
Let $\pi:\mathcal{X}\to P$ and $q_i: \mathcal{X}\to Q$, $1\leq i\leq N$, be continuous maps.
We suppose that $\pi:\mathcal{X}\to P$ is essential and satisfies
   \[   \left\{q_i^{-1}\left(O_Q(w)\right)\right\}_{w\in \ver(Q)} \prec \left\{\pi^{-1}\left(O_P(v)\right)\right\}_{v\in \ver(P)} \]
 for every $1\leq i\leq N$. (Here both sides are open covers of $\mathcal{X}$.)
Then there exist simplicial maps $h_i: P\to Q$, $1\leq i\leq N$, such that 
\begin{enumerate}
   \item For every $1\leq i\leq N$ and $x\in \mathcal{X}$ the two points $q_i(x)$ and $h_i\circ \pi(x)$ belong to the same simplex of $Q$.
   \item Let $1\leq i\leq N$. Let $\Delta\subset P$ be a simplex and $Q'\subset Q$ a subcomplex.
   If $\pi^{-1}(O_P(\Delta))\subset q_i^{-1}\left(Q'\right)$ 
   then $h_i(\Delta)\subset Q'$.
   \item For $1\leq i <j\leq N$ and a simplex $\Delta\subset P$, if $q_i=q_j$ on $\pi^{-1}(O_P(\Delta))$ then 
           $h_i=h_j$ on $\Delta$.
\end{enumerate}
\end{lemma}

\begin{proof}
Let $v\in \ver(P)$.
We can choose $h_i(v)\in \ver(Q)$, $1\leq i\leq N$, satisfying 
\begin{itemize}
   \item $\pi^{-1}(O_P(v)) \subset q_i^{-1}\left(O_Q(h_i(v))\right)$.
   \item If $q_i=q_j$ on $\pi^{-1}(O_P(v))$ then $h_i(v)=h_j(v)$.
\end{itemize}

Suppose $v_0,\dots,v_n\in \ver(P)$ span a simplex in $P$.
Since $\pi$ is essential, 
\begin{equation*}
   \begin{split}
       \emptyset\neq &\pi^{-1}\left(O_P(v_0)\cap \dots \cap O_P(v_n)\right)  \\
       \subset &
       q_i^{-1}\left(O_Q(h_i(v_0)) \cap \dots \cap O_Q(h_i(v_n))\right). 
   \end{split}
\end{equation*}       
In particular $O_Q(h_i(v_0)) \cap \dots \cap O_Q(h_i(v_n)) \neq \emptyset$. 
So $h_i(v_0),\dots, h_i(v_n)$ span a simplex in $Q$.
This implies that we can extend $h_i$ to a simplicial map from $P$ to $Q$.
The condition (3) immediately follows from the choices of $h_i(v)$.

For the proof of (1),
take $x\in \mathcal{X}$ and let $v_0,\dots,v_n \in \ver(P)$ be the vertices of the open simplex of $P$
containing $\pi(x)$.
Then $h_i(\pi(x))$ belongs to the simplex spanned by $h_i(v_0),\dots, h_i(v_n)$.
On the other hand 
\begin{equation*}
   \begin{split}
     x\in & \pi^{-1}\left(O_P(v_0)\cap \dots \cap O_P(v_n)\right) \\
     \subset & q_i^{-1}\left(O_Q(h_i(v_0))\cap \dots\cap O_Q(h_i(v_n))\right). 
   \end{split}
\end{equation*}     
Hence $q_i(x)\in O_Q(h_i(v_0))\cap \dots\cap O_Q(h_i(v_n))$ and there exists a simplex $\Delta\subset Q$
containing $q_i(x)$ and $h_i(v_0),\dots, h_i(v_n)$. 
Then $\Delta$ contains both $q_i(x)$ and 
$h_i(\pi(x))$.

For the proof of (2), take a subcomplex $Q'\subset Q$. Then 
\begin{claim}
  \item Let $v\in \ver(P)$ and $1\leq i\leq N$. If $\pi^{-1}\left(O_P(v)\right)\subset q_i^{-1}(Q')$ then 
          $h_i(v)\in Q'$ and $\pi^{-1}\left(O_P(v)\right) \subset q_i^{-1}\left(O_{Q'}(h_i(v))\right)$.
\end{claim}
\begin{proof}
\[  q_i\left(\pi^{-1}\left(O_P(v)\right)\right) \subset Q' \cap O_Q(h_i(v)). \]
If $h_i(v)\not \in Q'$ then the right-hand side is empty. 
So $h_i(v) \in Q'$ and hence $Q' \cap O_Q(h_i(v)) = O_{Q'}(h_i(v))$.
\end{proof}

Let $\Delta\subset P$ be a simplex with vertices $v_0,\dots,v_n$.
If $\pi^{-1}\left(O_P(\Delta)\right)\subset q_i^{-1}(Q')$ then $h_i(v_0), \dots, h_i(v_n)\in Q'$ and 
(since $\pi$ is essential)
\begin{equation*}
   \begin{split}
   \emptyset \neq &\pi^{-1}\left(O_P(v_0)\cap \dots \cap O_P(v_n)\right)  \\
   \subset & 
    q_i^{-1}\left(O_{Q'}(h_i(v_0)) \cap \dots \cap O_{Q'}(h_i(v_n))\right). 
   \end{split}
\end{equation*}   
In particular $O_{Q'}(h_i(v_0)) \cap \dots \cap O_{Q'}(h_i(v_n)) \neq \emptyset$.
So $h_i(v_0),\dots, h_i(v_n)$ span a simplex in $Q'$.
Then $h_i(\Delta)\subset Q'$.
\end{proof}

Letting $N=1$ in Lemma \ref{lemma: approximately commuting simplicial maps}, 
we get the following corollary. This is used in \S \ref{subsection: warmup: the proof of PS theorem}.

\begin{corollary}  \label{corollary: approximately commuting simplicial maps baby version}
Let $\mathcal{X}$ be a topological space and $P, Q$ simplicial complexes.
Let $\pi:\mathcal{X}\to P$ and $q:\mathcal{X}\to Q$ be continuous maps.
If $\pi$ is essential and 
\[  \left\{q^{-1}\left(O_Q(w)\right)\right\}_{w\in \ver(Q)} \prec \left\{\pi^{-1}\left(O_P(v)\right)\right\}_{v\in \ver(P)}, \]
then there exists a simplicial map $h:P\to Q$ such that for every $x\in \mathcal{X}$
the two points $q(x)$ and $h(\pi(x))$ belong to the same simplex in $Q$.
\end{corollary}

We need to introduce a notation for Lebesgue number.
Let $(\mathcal{X},d)$ be a compact metric space and $\mathcal{U}$ its open cover.
We denote the \textbf{Lebesgue number} of $\mathcal{U}$
 by $LN(\mathcal{X},d,\mathcal{U})$, namely it is 
the supremum of $\varepsilon>0$ such that if a subset $A\subset \mathcal{X}$ satisfies $\diam A < \varepsilon$ then
there exists $U\in \mathcal{U}$ containing $A$.

\subsection{Warmup: the proof of Pontrjagin--Schnirelmann's theorem} \label{subsection: warmup: the proof of PS theorem}

Here we prove Pontrjagin--Schnirelmann's theorem (Theorem \ref{theorem: PS theorem}) 
by using the preparations of \S \ref{subsection: preparation on combinatorial topology}.
This is a toy model of the proof of Theorem \ref{theorem: existence of nice metrics}.
(This subsection is logically independent of the proof of Theorem \ref{theorem: existence of nice metrics}.)
Our proof of Theorem \ref{theorem: PS theorem} roughly follows the line of ideas of \cite{Pontrjagin--Schnirelmann}.
Our purpose here is to help readers to get acquainted with how to use lemmas in the previous subsection.
Theorem \ref{theorem: PS theorem} follows from

\begin{theorem} \label{theorem: variant of PS theorem}
Let $(V,\norm{\cdot})$ be an infinite dimensional Banach space and $\mathcal{X}$ a compact metrizable space.
For a dense subset of $f$ in $C(\mathcal{X},V)$ (the space of continuous maps from $\mathcal{X}$ to $V$ endowed with the norm topology),
$f$ is a topological embedding and satisfies 
\[  \overline{\dim}_{\mathrm{M}}(f(\mathcal{X}),\norm{\cdot}) = \dim \mathcal{X}. \]
\end{theorem}

\begin{proof}
We can assume that $D:=\dim \mathcal{X}<\infty$.
Fix a metric $d$ on $\mathcal{X}$ and take arbitrary $f\in C(\mathcal{X},V)$ and $\eta>0$.
We want to construct a topological embedding $f':\mathcal{X}\to V$ satisfying 
$\norm{f(x)-f'(x)} < \eta$ and $\overline{\dim}_{\mathrm{M}}(f'(\mathcal{X}), \norm{\cdot}) = D$.
(It is enough to prove $\overline{\dim}_{\mathrm{M}}(f(\mathcal{X}),\norm{\cdot}) \leq D$ because 
Minkowski dimension always dominates topological dimension.)
We may assume that $f(\mathcal{X})$ is contained in the open unit ball $B^\circ_1(V)$.
We will inductively construct the following data $(n\geq 1)$:

\begin{data}   \label{data: PS theorem}
  \begin{enumerate}
      \item Positive numbers $\varepsilon_n$ and $\delta_n$ with $\varepsilon_{n+1}<\varepsilon_n/2$,
              $\delta_{n+1}<\delta_n/2$ and $\delta_1<\eta/2$.
      \item A $(1/n)$-embedding $\pi_n: (\mathcal{X},d)\to P_n$ such that $P_n$ is a simplicial complex of dimension $\leq D$.
      \item A linear embedding $g_n: P_n\to B_1^\circ(V)$.
  \end{enumerate}
\end{data}

We assume the following conditions:

\begin{condition} \label{condition: PS theorem}
  \begin{enumerate}
     \item 
     \[   \#\left(g_n(P_n),\norm{\cdot},\varepsilon\right) < \begin{cases}
             \left(\frac{2}{\varepsilon}\right)^{D+\frac{1}{n-1}}  
              \quad (\varepsilon <\varepsilon_{n-1})  \\
             \left(\frac{1}{\varepsilon}\right)^{D+\frac{1}{n}}  \quad (\varepsilon < \varepsilon_n).
             \end{cases}    \]
      Here the former condition is empty for $n=1$.
      \item Set $f_n = g_n\circ \pi_n : \mathcal{X}\to V$. 
             If a continuous map $f':\mathcal{X}\to V$ satisfies $\norm{f'(x)-f_n(x)} < \delta_n$ then $f'$ is a $(1/n)$-embedding 
             with respect to $d$.
      \item 
      \[   \norm{f_1(x)-f(x)} < \frac{\eta}{2}, \]
      \[   \norm{f_n(x)-f_{n+1}(x)} < \min\left(\frac{\varepsilon_n}{8}, \frac{\delta_n}{2}\right). \]       
  \end{enumerate}
\end{condition}

Suppose we have constructed the above data.
Then we can define $f'\in C(\mathcal{X},V)$ by $f'(x) = \lim_{n\to \infty} f_n(x)$.
It follows from Condition \ref{condition: PS theorem} (3) that $\norm{f'(x)-f_n(x)} < \delta_n$ for all $n\geq 1$. 
So by Condition \ref{condition: PS theorem} (2) $f'$ is a $(1/n)$-embedding 
for all $n\geq 1$, which implies that $f'$ is a topological embedding.
It also satisfies $\norm{f'(x)-f(x)} < \eta$.

We want to prove $\overline{\dim}_{\mathrm{M}} \left(f'(\mathcal{X}), \norm{\cdot}\right)\leq D$.
Let $0<\varepsilon <\varepsilon_1$. 
Take $n> 1$ with $\varepsilon_n \leq \varepsilon < \varepsilon_{n-1}$.
It follows from Condition \ref{condition: PS theorem} (3) that 
$\norm{f'(x)-f_n(x)} < \varepsilon_n/4$. Hence 
\begin{equation*}
    \begin{split}
      \#\left(f'(\mathcal{X}), \norm{\cdot},\varepsilon\right) 
      & \leq \#\left(f_n(\mathcal{X}), \norm{\cdot},\varepsilon -\frac{\varepsilon_n}{2}\right)   \\
      & \leq \#\left(g_n(P_n),\norm{\cdot},\varepsilon-\frac{\varepsilon_n}{2}\right)  \quad 
         (\text{by $f_n(\mathcal{X}) \subset g_n(P_n)$}) \\
      & < \left(\frac{2}{\varepsilon-\frac{\varepsilon_n}{2}}\right)^{D+\frac{1}{n-1}} \quad 
         (\text{by Condition \ref{condition: PS theorem} (1)}) \\
      & \leq \left(\frac{4}{\varepsilon}\right)^{D+ \frac{1}{n-1}} \quad 
         (\text{by $\varepsilon \geq \varepsilon_n$}).   
    \end{split}
\end{equation*}
Since $n\to \infty$ as $\varepsilon \to 0$, this shows
$\overline{\dim}_{\mathrm{M}}\left(f'(\mathcal{X}),\norm{\cdot}\right) \leq D$.

Now we start to construct the data.
We choose $0<\tau_1<1$ so that 
\[  d(x,y) < \tau_1 \Longrightarrow \norm{f(x)-f(y)} < \eta/2.  \]
Let $\pi_1:(\mathcal{X},d) \to P_1$ be a $\tau_1$-embedding in a simplicial complex $P_1$ with $\dim P_1\leq D$.
(Since $D = \dim \mathcal{X} = \lim_{\varepsilon \to 0} \widim_\varepsilon (\mathcal{X},d)$, we can find such a map.)
By subdividing $P_1$ sufficiently fine, we can assume $\diam \, \pi_1^{-1}(O_{P_1}(v)) < \tau_1$ for all 
$v\in \ver(P_1)$.
Then by Lemma \ref{lemma: preparations on linear maps} (3)
we can find a linear map $\tilde{g}_1:P_1\to B_1^\circ(V)$ satisfying
$\norm{\tilde{g}_1(\pi_1(x))-f(x)} < \eta/2$.
Since linear embeddings are dense in $\mathrm{Hom}(P_1,V)$ (Lemma \ref{lemma: preparations on linear maps} (2)),
we can also find a linear embedding $g_1:P_1\to B_1^\circ(V)$ satisfying 
$\norm{g_1(\pi_1(x))-f(x)} < \eta/2$.
By Lemma \ref{lemma: preparations on linear maps} (1), we can find $\varepsilon_1>0$ satisfying 
Condition \ref{condition: PS theorem} (1).
The map $f_1 = g_1\circ \pi_1$ is a $1$-embedding and ``$1$-embedding'' is an open condition.
So there exists $\delta_1>0$ such that Condition \ref{condition: PS theorem} (2) holds true.
This finishes the construction for $n=1$.

Suppose that we have already done the construction for the $n$-th step.
We try to construct the data for the $(n+1)$-th step.
We subdivide $P_n$ sufficiently fine so that every simplex $\Delta\subset P_n$ satisfies 
$\diam \left(g_n(\Delta), \norm{\cdot}\right) < \min(\varepsilon_n/8,\delta_n/2)$.
We take $0<\tau_{n+1}<1/(n+1)$ with 
\[  \tau_{n+1} < LN\left(\mathcal{X}, d, \{\pi_n^{-1}(O_{P_n}(v))\}_{v\in \ver(P_n)}\right). \]
Take a $\tau_{n+1}$-embedding $\pi_{n+1}:(\mathcal{X},d)\to P_{n+1}$ with a simplicial complex $P_{n+1}$
of dimension $\leq D$.
By subdividing $P_{n+1}$ sufficiently fine, we can assume 
$\diam \left(\pi_{n+1}^{-1}(O_{P_{n+1}}(v))\right) < \tau_{n+1}$ for all $v\in \ver(P_{n+1})$.
Moreover, by replacing $P_{n+1}$ with a subcomplex (if necessarily), we can assume that 
$\pi_{n+1}:\mathcal{X}\to P_{n+1}$ is essential (Lemma \ref{lemma: essential map}).
The open cover $\left\{\pi_{n+1}^{-1}\left(O_{P_{n+1}}(v)\right)\right\}_{v\in \ver(P_{n+1})}$ of $\mathcal{X}$ becomes a 
refinement of $\left\{\pi_n^{-1}\left(O_{P_n}(v)\right)\right\}_{v\in \ver(P_n)}$ because of the 
Lebesgue number condition above.
Then by applying Corollary \ref{corollary: approximately commuting simplicial maps baby version}
to $\pi_{n+1}: \mathcal{X}\to P_{n+1}$ and $\pi_n:\mathcal{X}\to P_n$ 
(with $P=P_{n+1}$ and $Q=P_n$), 
we can find a simplicial map 
$h:P_{n+1}\to P_n$ such that for every $x\in \mathcal{X}$ the two points 
$\pi_n(x)$ and $h(\pi_{n+1}(x))$ belong to the same simplex of $P_n$.
Set $\tilde{g}_{n+1} = g_n\circ h: P_{n+1} \to B^\circ_1(V)$.
This satisfies (recall $f_n = g_n\circ \pi_n$)
\[  \norm{\tilde{g}_{n+1}(\pi_{n+1}(x))-f_n(x)} < \min\left(\frac{\varepsilon_n}{8}, \frac{\delta_n}{2}\right). \]
Since $\tilde{g}_{n+1}(P_{n+1}) \subset g_n(P_n)$, the induction hypothesis implies
\begin{equation}  \label{eq: covering number of tilde{g}_{n+1} in PS theorem}
   \#\left(\tilde{g}_{n+1}(P_{n+1}), \norm{\cdot},\varepsilon\right) < \left(\frac{1}{\varepsilon}\right)^{D+\frac{1}{n}}  \quad 
   (\varepsilon < \varepsilon_n).  
\end{equation}
By Lemma \ref{lemma: preparations on linear maps} (1), there exists 
$0<\varepsilon_{n+1}<\varepsilon_n/2$ such that for all linear maps $g:P_{n+1}\to V$ with $g(P_{n+1}) \subset B^\circ_1(V)$
\[  \#\left(g(P_{n+1}),\norm{\cdot},\varepsilon\right) < \left(\frac{1}{\varepsilon}\right)^{D+\frac{1}{n+1}} \quad 
     (\varepsilon < \varepsilon_{n+1}). \]
We slightly perturb $\tilde{g}_{n+1}$ by Lemma \ref{lemma: preparations on linear maps} (2):
There exists a linear embedding $g_{n+1}:P_{n+1}\to B_1^\circ(V)$ such that 
\[   \norm{g_{n+1}(\pi_{n+1}(x))-f_n(x)} < \min\left(\frac{\varepsilon_n}{8}, \frac{\delta_n}{2}\right), \]
\begin{equation}  \label{eq: difference between g_n and g_{n+1} in PS theorem}
   \norm{g_{n+1}(u)- \tilde{g}_{n+1}(u)} < \frac{\varepsilon_{n+1}}{4} \quad (u\in P_{n+1}).
\end{equation}
By the choice of $\varepsilon_{n+1}$ we have 
\[  \#\left(g_{n+1}(P_{n+1}),\norm{\cdot},\varepsilon\right) < \left(\frac{1}{\varepsilon}\right)^{D+\frac{1}{n+1}} \quad 
     (\varepsilon < \varepsilon_{n+1}). \]
For $\varepsilon_{n+1}\leq \varepsilon < \varepsilon_n$
\begin{equation*}
   \begin{split}
      \#\left(g_{n+1}(P_{n+1}),\norm{\cdot},\varepsilon\right) & \leq 
      \#\left(\tilde{g}_{n+1}(P_{n+1}), \norm{\cdot}, \varepsilon- \frac{\varepsilon_{n+1}}{2}\right) \quad 
      (\text{by (\ref{eq: difference between g_n and g_{n+1} in PS theorem}}))     \\
      & \leq      \#\left(\tilde{g}_{n+1}(P_{n+1}), \norm{\cdot}, \frac{\varepsilon}{2}\right) \quad 
      \left(\text{by }\varepsilon-\frac{\varepsilon_{n+1}}{2} \geq \frac{\varepsilon}{2}\right) \\
      & < \left(\frac{2}{\varepsilon}\right)^{D+\frac{1}{n}} \quad 
      (\text{by (\ref{eq: covering number of tilde{g}_{n+1} in PS theorem})}).
   \end{split}
\end{equation*}   
$f_{n+1} = g_{n+1}\circ \pi_{n+1}$ is a $1/(n+1)$-embedding.
So we can find $0<\delta_{n+1}<\delta_n/2$ satisfying Condition \ref{condition: PS theorem} (2).
This has completed the construction for the $(n+1)$-th step.
\end{proof}

\subsection{Dynamical tiling construction}  \label{subsection: dynamical tiling construction}

Here we review a construction introduced in \cite{Gutman--Lindenstrauss--Tsukamoto}.
Let $(\mathcal{X},T)$ be a dynamical system and $\varphi:\mathcal{X}\to [0,1]$ a continuous function.
For $x\in \mathcal{X}$ we consider 
\begin{equation} \label{eq: marker points from varphi}
    \left\{\left(a, \frac{1}{\varphi(T^a x)}\right)\middle|\, a\in \mathbb{Z} \text{ with $\varphi(T^a x) > 0$} \right\}
     \subset \mathbb{R}^2. 
\end{equation}     
 We assume that this is nonempty for every $x\in \mathcal{X}$.
 (Namely, for every $x\in \mathcal{X}$, there exists $a\in \mathbb{Z}$ with $\varphi(T^a x)>0$.)
Let $\mathbb{R}^2 = \bigcup_{a \in \mathbb{Z}} V_\varphi(x,a)$ be the associated \textbf{Voronoi diagram}, namely 
$V_\varphi(x,a)$ is the set of $u\in \mathbb{R}^2$ satisfying 
\[  \left|u-\left(a, \frac{1}{\varphi(T^a x)}\right)\right| \leq  
      \left|u-\left(b, \frac{1}{\varphi(T^b x)}\right)\right|    \]
for any $b \in \mathbb{Z}$ with $\varphi(T^b x) > 0$.      
This is a convex subset of the plane.
We set 
\[  I_\varphi(x,a) = V_\varphi(x,a) \cap (\mathbb{R}\times \{0\}). \]
See Figure \ref{fig: tiling}.
If $\varphi(T^a x)=0$ then $V_\varphi(x,a) = I_\varphi(x,a)= \emptyset$.

\begin{figure}[h]
    \centering
    \includegraphics[width=5.0in]{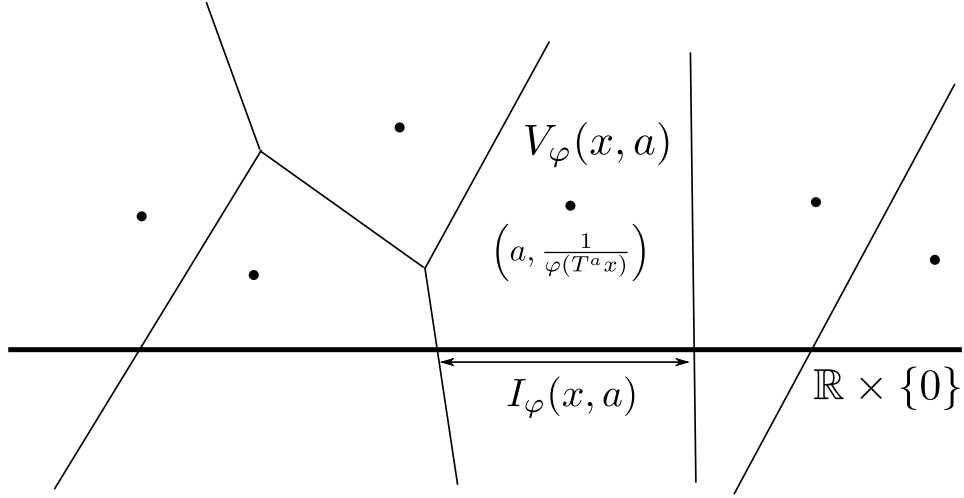}
    \caption{Dynamical tiling construction} 
    \label{fig: tiling}
\end{figure}

We naturally identify $\mathbb{R}\times \{0\}$ with $\mathbb{R}$.
Then this construction gives a decomposition of $\mathbb{R}$:
\[  \mathbb{R} = \bigcup_{a \in \mathbb{Z}} I_\varphi(x,a). \]
$I_\varphi(x,a)$ are closed intervals.
We set 
\[  \partial_\varphi(x) = \bigcup_{a \in \mathbb{Z}} \partial I_\varphi(x,a) \subset \mathbb{R}, \]
where $\partial I_\varphi(x,a)$ is the boundary of the interval $I_\varphi(x,a)$
(e.g. $\partial [0,1] = \{0,1\}$).
This construction is equivariant, namely 
\[  I_\varphi(T^n x, a) = -n + I_\varphi(x, a+n), \quad \partial_\varphi(T^n x) = -n + \partial_\varphi(x). \]

\begin{lemma} \label{lemma: dynamical tiling}
Suppose $(\mathcal{X},T)$ has the marker property.
Then for any $\varepsilon>0$ we can find a continuous function $\varphi:\mathcal{X}\to [0,1]$ such that 
(\ref{eq: marker points from varphi}) is nonempty for every $x\in \mathcal{X}$ and that it satisfies
the following conditions.
\begin{enumerate}
   \item  There exists $M>0$ such that $I_\varphi(x,a) \subset (a-M,a+M)$ for all $x\in \mathcal{X}$ and $a\in \mathbb{Z}$
            (in particular, all $I_\varphi(x,a)$ are finite length intervals).  
   \item  
          \[ \lim_{R\to \infty} \frac{\sup_{x\in \mathcal{X}} |\partial_\varphi(x)\cap [0,R]|}{R} < \varepsilon. \]
          Here $|\partial_{\varphi}(x)\cap [0,R]|$ is the cardinality of 
          $\partial_{\varphi}(x)\cap [0,R]$.
         Notice that the above (1) implies that $\partial_\varphi (x)$ is a discrete set in the real line.
    \item The intervals $I_\varphi(x,a)$ continuously depend on $x\in \mathcal{X}$:
             i.e.~if $x_k\to x$ in $\mathcal{X}$ and $I_\varphi(x,a)$ has positive length then 
            $I_\varphi(x_k,a)$ converges to $I_\varphi(x,a)$ in the Hausdorff topology, 
            and if $I_\varphi(x,a)=\emptyset$ then for all $k$ large enough $I_\varphi(x_k,a)$ is also empty.
\end{enumerate}
\end{lemma}

\begin{proof}
Take $N> 1/\varepsilon$.
From the marker property, there exists an open set $U\subset \mathcal{X}$ such that 
$U\cap T^{-n} U = \emptyset$ for $1\leq n\leq N$ and 
$\mathcal{X} = \bigcup_{n\in \mathbb{Z}} T^{-n} U$.
We can find $M> N$ and a compact subset $K\subset U$ with 
$\mathcal{X} = \bigcup_{n=0}^{M-1} T^{-n} K$.
Let $\varphi:\mathcal{X}\to [0,1]$ be a continuous function such that $\varphi = 1$ on $K$ and 
$\supp\, \varphi \subset U$.

Let $x\in \mathcal{X}$ and consider
\[  \Lambda_x = \{a\in \mathbb{Z}|\, \varphi(T^a x) > 0\}, \quad
    \Lambda'_x = \{a\in \mathbb{Z}|\, \varphi(T^a x) =1\}. \]
Any gap of $\Lambda_x$ (i.e. the difference between two successive numbers in $\Lambda_x$) is larger than $N$, and 
any gap of $\Lambda'_x$ is smaller than or equal to $M$. 
For $a\in \Lambda_x'$ the interval $I_\varphi(x,a)$ contains $a$ as an interior point.

Let $a\in \Lambda_x$. There exist $s,t\in \Lambda_x'$ such that $s<a<t$ and $a-s, t-a \leq M$.
Then $I_\varphi(x,a) \subset (s,t) \subset (a-M,a+M)$.
The continuity of $I_\varphi(x,a)$ is an immediate consequence of the definition.
The condition (2) follows from
\[ \lim_{R\to \infty} \frac{\sup_{x\in \mathcal{X}} |\partial_\varphi(x)\cap [0,R]|}{R} 
    \leq \lim_{R\to \infty} \frac{\sup_{x\in \mathcal{X}} |\Lambda_x\cap [0,R]|}{R} 
     \leq \frac{1}{N}< \varepsilon. \]
\end{proof}

\subsection{Proof of Theorem \ref{theorem: existence of nice metrics}}  \label{subsection: proof of existence of nice metrics}

Let $(V,\norm{\cdot})$ be an infinite dimensional Banach space and $(\mathcal{X},T)$ a dynamical system.
As in \S \ref{subsection: warmup: the proof of PS theorem} we denote by $C(\mathcal{X},V)$ 
the space of continuous maps from $\mathcal{X}$ to $V$ endowed with the norm topology.
Theorem \ref{theorem: existence of nice metrics} follows from

\begin{theorem} \label{theorem: dynamical PS theorem}
Suppose $(\mathcal{X},T)$ has the marker property.
For a dense subset of $f\in C(\mathcal{X},V)$, $f$ is a topological embedding and satisfies
\[  \overline{\mdim}_{\mathrm{M}}(\mathcal{X},T, f^*\norm{\cdot}) = \mdim(\mathcal{X},T). \]
Here $f^*\norm{\cdot}$ is the metric $\norm{f(x)-f(y)}$ $(x,y\in \mathcal{X})$.
\end{theorem}

\begin{proof}
We can assume $D:=\mdim(\mathcal{X},T) < \infty$.
Fix a metric $d$ on $\mathcal{X}$ and take an arbitrary $f\in C(\mathcal{X},V)$ and $\eta>0$.
We want to construct a topological embedding $f':\mathcal{X}\to V$ satisfying 
$\norm{f(x)-f'(x)} < \eta$ for all $x\in \mathcal{X}$ and $\overline{\mdim}_{\mathrm{M}}\left(\mathcal{X},T,(f')^*\norm{\cdot}\right) = D$.
(It is enough to show $\overline{\mdim}_{\mathrm{M}}\left(\mathcal{X},T,(f')^*\norm{\cdot}\right) \leq D$ 
because the reverse inequality is always true by Theorem \ref{theorem: metric mean dimension dominates mean dimension}). 
We may assume that $f(\mathcal{X})$ is contained in the open unit ball $B_1^\circ(V)$.

We prepare some notations. For a natural number $N$ we set 
$[N] = \{0,1,2,\dots,N-1\}$.
We define a norm $\norm{\cdot}_N$ on $V^N$ (the $N$-th power of $V$) by 
\[ \norm{(x_0,x_1,\dots,x_{N-1})}_N = \max\left(\norm{x_0},\norm{x_1},\dots,\norm{x_{N-1}}\right). \]
For a simplicial complex $P$ we define the number $A(P)$ as the minimum of $A\geq 1$
such that for any linear map $g:P\to B_1(V)$
we have $\#\left(g(P),\norm{\cdot},\varepsilon\right) \leq (1/\varepsilon)^A$ for all $0<\varepsilon \leq 1/2$.
(Such a number always exists by Lemma \ref{lemma: preparations on linear maps} (1).)
For simplicial complexes $P$ and $Q$, we denote their join by $P*Q$, namely it is the quotient of $[0,1]\times P\times Q$ by 
the equivalence relation 
\[  (0, p,q ) \sim (0,p,q'), \quad (1, p,q) \sim (1, p', q), \quad (p,p'\in P, q,q'\in Q). \]
We denote the equivalence class of $(t, p, q)$ by $(1-t)p\oplus tq$.
We identify $P$ and $Q$ with $\{(0,p,*)|\, p\in P\}$ and $\{(1,*,q)|\, q\in Q\}$ in $P*Q$ respectively.
If $g:P\to B_1(V)$ and $g':Q\to B_1(V)$ we define the map 
$g * g' : P*Q \to B_1(V)$ by sending $(1-t)p\oplus tq$ to $(1-t)g(p)+tg'(q)$.

\medskip
For a continuous map $f':\mathcal{X}\to V$ and $I\subset \mathbb{R}$ we define 
$\Phi_{f',I}:\mathcal{X}\to V^{I\cap \mathbb{Z}}$ by 
\[ \Phi_{f',I}(x) = \left(f'(T^a x)\right)_{a\in I\cap \mathbb{Z}}. \]
For a natural number $R$ we set $\Phi_{f', R} := \Phi_{f', [0,R)}: \mathcal{X} \to V^R$.
We fix a continuous function $\alpha:\mathbb{R}\to [0,1]$ such that $\alpha(t)=1$ for $t\leq 1/2$ and $\alpha(t)=0$ 
for $t\geq 3/4$.

We will inductively construct the following data for $n\geq 1$.

\begin{data} \label{data: dynamical PS theorem}
   \begin{enumerate}
     \item $1/2>\varepsilon_1>\varepsilon_2>\dots>0$ with $\varepsilon_{n+1}<\varepsilon_n/2$ and 
             $\eta/2>\delta_1>\delta_2>\dots>0$ with $\delta_{n+1}<\delta_n/2$.
     \item A natural number $N_n$.     
     \item A continuous function $\varphi_n: \mathcal{X}\to [0,1]$ such that for every $x\in \mathcal{X}$ there exists $a\in \mathbb{Z}$
             with $\varphi_n(T^a x)>0$. We apply the dynamical tiling construction of \S \ref{subsection: dynamical tiling construction}
             to this function and get the decomposition 
             $\mathbb{R}=\bigcup_{a\in \mathbb{Z}} I_{\varphi_n}(x,a)$ for each $x\in \mathcal{X}$.
     \item A $(1/n)$-embedding $\pi_n: (\mathcal{X},d_{N_n}) \to P_n$ with a simplicial complex $P_n$ of dimension 
             less than $(D+\frac{1}{n}) N_n$.
     \item A $(1/n)$-embedding $\pi'_n: (\mathcal{X},d) \to Q_n$ with a simplicial complex $Q_n$.        
     \item For each $\lambda\in [N_n]$, a linear embedding $g_{n,\lambda}:P_n\to B_1^\circ (V)$.
     \item A linear embedding $g'_n: Q_n\to B_1^\circ(V)$.
   \end{enumerate}
\end{data}

We assume the following conditions.

\begin{condition} \label{condition: dynamical PS theorem}
  \begin{enumerate}
    \item For each $\lambda\in [N_n]$, the join $g_{n,\lambda}* g'_n: P_n*Q_n\to B_1^\circ (V)$ is a linear embedding.
            For $\lambda_1\neq \lambda_2$
            \[  g_{n,\lambda_1}*g'_n(P_n*Q_n) \cap g_{n,\lambda_2}*g'_n(P_n*Q_n) = g'_n(Q_n). \]
   \item Set $g_n = (g_{n,0}, g_{n,1},\dots, g_{n,N_n-1}): P_n\to V^{N_n}$. Then 
         \[  \#\left(g_n(P_n), \norm{\cdot}_{N_n}, \varepsilon\right) <  
              \begin{cases}
                    4^{N_n} \left(\frac{2}{\varepsilon}\right)^{\left(D+\frac{2}{n-1}\right)N_n}   \quad
                    & (0< \varepsilon < \varepsilon_{n-1}) \\
                   \left(\frac{1}{\varepsilon}\right)^{\left(D+\frac{1}{n}\right)N_n} \quad & (0< \varepsilon < \varepsilon_n)    
              \end{cases}.  \]
           Here the former condition is empty for $n=1$.   
   \item   There exists $M_n>0$ such that $I_{\varphi_n}(x,a)\subset (a-M_n,a+M_n)$ for all $x\in \mathcal{X}$ and 
             $a\in \mathbb{Z}$.
           The sets $\partial_{\varphi_n}(x)$ are discrete in $\mathbb{R}$ and satisfy
           \[  \lim_{R\to \infty} \frac{\sup_{x\in \mathcal{X}} |\partial_{\varphi_n}(x)\cap [0,R]|}{R} <
                \frac{1}{2n N_n^2 \cdot A(P_n*Q_n)}. \]
   \item We define a continuous map $f_n:\mathcal{X}\to B_1^\circ(V)$ as follows: Let $x\in \mathcal{X}$ and take $a\in \mathbb{Z}$ with 
           $0\in I_{\varphi_n}(x,a)$.
           Take $b\in \mathbb{Z}$ such that $b\equiv a \, (\mathrm{mod} N_n)$ and $0\in b+[N_n]$. Set 
          \begin{equation*}
               \begin{split}
                      f_n(x) = &\left\{1-\alpha\left(\dist(0, \partial_{\varphi_n}(x))\right)\right\} g_{n,-b}\left(\pi_n(T^b x)\right)    \\
                                 & + \alpha\left(\dist(0,\partial_{\varphi_n}(x))\right) g'_n\left(\pi'_n(x)\right), 
               \end{split}
          \end{equation*}     
          where $\dist(0,\partial_{\varphi_n}(x)) := \min_{t\in \partial_{\varphi_n}(x)} |t|$.
          Then we assume that if a continuous map $f':\mathcal{X}\to V$ satisfies $\norm{f'(x)-f_n(x)}< \delta_n$ for all $x\in \mathcal{X}$
          then it is a $(1/n)$-embedding with respect to $d$.
  \item  For all $x\in \mathcal{X}$
       \begin{equation*}
           \begin{split}
             \norm{f(x)-f_1(x)} &< \frac{\eta}{2}, \\
             \norm{f_n(x)-f_{n+1}(x)} &< \min\left(\frac{\varepsilon_n}{8}, \frac{\delta_n}{2}\right). 
            \end{split} 
       \end{equation*}            
  \end{enumerate}
\end{condition}

Suppose we have constructed the above data.
Then we can define $f'\in C(\mathcal{X},V)$ by $f'(x) = \lim_{n\to \infty} f_n(x)$.
It satisfies $\norm{f'(x)-f(x)} < \eta$ and 
 $\norm{f'(x)-f_n(x)} < \min(\varepsilon_n/4,\delta_n)$ for all $n\geq 1$.
From Condition \ref{condition: dynamical PS theorem} (4), $f'$ is a $(1/n)$-embedding for all $n\geq 1$.
So it is a topological embedding.
Set $d'(x,y) = \norm{f'(x)-f'(y)}$.
We want to show $\overline{\mdim}_{\mathrm{M}}\left(\mathcal{X},T,d'\right) \leq D$.
Notice that $\overline{\mdim}_{\mathrm{M}}\left(\mathcal{X},T,d'\right)$ is equal to
\begin{equation} \label{eq: metric mean dimension for d' in dynamical PS theorem}
  \limsup_{\varepsilon \to 0}
  \left\{ \lim_{R\to \infty} \left(\frac{\log \#\left(\Phi_{f',R}(\mathcal{X}), \norm{\cdot}_R,\varepsilon\right)}{R}\right) 
  \middle/ \log(1/\varepsilon)\right\}.
\end{equation}

\begin{claim} \label{claim: estimate of f_n(X)}
Let $0< \varepsilon < \varepsilon_{n-1}$ $(n\geq 2)$.
For sufficiently large natural numbers $R$
\begin{equation*}
   \#\left(\Phi_{f_n,R}(\mathcal{X}),\norm{\cdot}_R,\varepsilon\right) 
   < 2^{4R} \left(\frac{2}{\varepsilon}\right)^{\left(D+\frac{2}{n-1}\right)R + \frac{R}{n}}. 
\end{equation*}         
\end{claim}

\begin{proof}
In this proof $n\geq 2$ is fixed and $R$ is a sufficiently large natural number.
Let $x\in \mathcal{X}$.
We call $J = \{b, b+1, \dots, b+N_n-1\}$ $(b\in \mathbb{Z})$ \textbf{good for $x$} if 
there is $a\in \mathbb{Z}$ such that $b\equiv a \, (\mathrm{mod} N_n)$ and $(b-1, b+N_n) \subset I_{\varphi_n}(x,a)$.
If $J$ is good for $x$ then $\Phi_{f_n,J}(x)$ is contained in $g_n(P_n)$ in $V^{N_n}$.
We denote by $\mathcal{J}_x$ the union of $J \subset [R]$ which are good for $x$.
The number of possibilities of $\mathcal{J}_x$ (when $x\in \mathcal{X}$ varies) is bounded by 
$2^{R}$. 
Then $\#\left(\Phi_{f_n,R}(\mathcal{X}), \norm{\cdot}_R,\varepsilon\right)$ is bounded by 
\begin{equation*}
    2^{R} \cdot
    \underbrace{\#\left(g_n(P_n), \norm{\cdot}_{N_n}, \varepsilon\right)^{\frac{R}{N_n}}}_{\text{contribution over $\mathcal{J}_x$}} 
    \cdot
  \underbrace{\#\left(f_n(\mathcal{X}), 
  \norm{\cdot},\varepsilon\right)^{2N_n\sup_{x\in \mathcal{X}} 
  |\partial_{\varphi_n}(x)\cap [0,R]|+2N_n}}_{\text{contribution over $[0,R)\setminus \mathcal{J}_x$}}.
\end{equation*}   
Here ``$+2N_n$'' is the edge effect.
$f_n(\mathcal{X})$ is contained in the union of $g_{n,\lambda}*g'_n(P_n*Q_n)$ over $\lambda\in [N_n]$.
So 
\[  \#\left(f_n(\mathcal{X}), \norm{\cdot},\varepsilon \right) \leq N_n \left(\frac{1}{\varepsilon}\right)^{A(P_n*Q_n)}. \]
Using Condition \ref{condition: dynamical PS theorem} (2) and (3), we get the statement of the claim.
\end{proof}

We now return to the proof of Theorem~\ref{theorem: dynamical PS theorem}.

Let $0<\varepsilon < \varepsilon_1$.
Take $n> 1$ with $\varepsilon_n \leq \varepsilon < \varepsilon_{n-1}$.
Recall that $\norm{f'(x)-f_n(x)} < \varepsilon_n/4$. Hence
\begin{equation*}
   \begin{split}
       \#\left(\Phi_{f',R}(\mathcal{X}), \norm{\cdot}_R,\varepsilon\right)
    &  \leq \#\left(\Phi_{f_n,R}(\mathcal{X}), \norm{\cdot}_R,\varepsilon -\frac{\varepsilon_n}{2}\right)  \\
    &  \leq  \#\left(\Phi_{f_n,R}(\mathcal{X}), \norm{\cdot}_R,\frac{\varepsilon}{2}\right).  
   \end{split}
\end{equation*}      
From Claim \ref{claim: estimate of f_n(X)},
\begin{equation*}
  \lim_{R \to \infty}  \frac{\log \# \left(\Phi_{f',R}(\mathcal{X}), \norm{\cdot}_R, \varepsilon\right)}{R} 
          \leq 4  + \left(D+\frac{2}{n-1} + \frac{1}{n}\right) \log (4/\varepsilon).   
\end{equation*}    
Notice that $n\to \infty$ as $\varepsilon \to 0$.
Using (\ref{eq: metric mean dimension for d' in dynamical PS theorem}) we get 
$\overline{\mdim}_{\mathrm{M}}(\mathcal{X},T,d') \leq D$.

\vspace{0.2cm}

\textbf{Induction: Step 1.}
Now we start to construct the data.
Take $0<\tau_1<1$ such that 
\[  d(x,y) < \tau_1 \Longrightarrow \norm{f(x)-f(y)} < \frac{\eta}{2}. \]
From the definition of mean dimension, we can find $N_1>0$ and $\tau_1$-embeddings 
$\pi_1:(\mathcal{X},d_{N_1})\to P_1$ and $\pi'_1:(\mathcal{X},d)\to Q_1$ such that 
$P_1$ and $Q_1$ are simplicial complexes with $\dim P_1 < N_1 (D+1)$.
By subdividing $P_1$ and $Q_1$, we can assume that 
\[  \diam \left(\pi_1^{-1}\left(O_{P_1}(v)\right), d_{N_1}\right) < \tau_1, \quad
     \diam \left( (\pi'_1)^{-1}\left(O_{Q_1}(w)\right), d\right) < \tau_1 \]
     for all $v\in \ver(P_1)$ and $w\in \ver(Q_1)$.
By Lemma \ref{lemma: preparations on linear maps} (3), we can find 
linear maps $\tilde{g}_{1,\lambda}:P_1\to B_1^\circ(V)$ for $\lambda\in [N_1]$
and $\tilde{g}'_1:Q_1\to B_1^\circ(V)$ satisfying 
\[  \norm{f(T^\lambda x)-\tilde{g}_{1,\lambda}(\pi_1(x))} < \frac{\eta}{2}, \quad 
    \norm{f(x)-\tilde{g}'_1(\pi'_1(x))} < \frac{\eta}{2} \]
for all $x\in \mathcal{X}$.    
By Lemma \ref{lemma: preparations on linear maps} (2), we can replace 
$\tilde{g}_{1,\lambda}$ and $\tilde{g}'_1$ with 
linear embeddings $g_{1,\lambda}:P_1\to B_1^\circ(V)$ and $g'_1:Q_1\to B_1^\circ(V)$
satisfying Condition \ref{condition: dynamical PS theorem} (1) and 
\begin{equation} \label{eq: g_1 approximate f}
     \norm{f(T^\lambda x)-g_{1,\lambda}(\pi_1(x))} < \frac{\eta}{2}, \quad 
    \norm{f(x)-g'_1(\pi'_1(x))} < \frac{\eta}{2}.   
\end{equation}    
By Lemma \ref{lemma: preparations on linear maps} (1), we can find $0< \varepsilon_1 <1/2$ satisfying 
Condition \ref{condition: dynamical PS theorem} (2).

By Lemma \ref{lemma: dynamical tiling}, we can take a continuous function 
$\varphi_1:\mathcal{X}\to [0,1]$ satisfying Condition \ref{condition: dynamical PS theorem} (3).
By Condition \ref{condition: dynamical PS theorem} (1) (which has been already established for $n=1$),
the map $f_1:\mathcal{X}\to V$ becomes a $1$-embedding.
It also satisfies Condition \ref{condition: dynamical PS theorem} (5)
by (\ref{eq: g_1 approximate f}).
Since ``$1$-embedding'' is an open condition, we can find $0< \delta_1 <\eta/2$ satisfying
Condition \ref{condition: dynamical PS theorem} (4). 
The first step of the induction has been completed.

\vspace{0.2cm}

\textbf{Induction: Step $n$ $\Rightarrow$ Step $n+1$.}
Next we suppose that we have constructed the date for $n$.
We will construct the date for $n+1$.

 We subdivide $P_n * Q_n$ sufficiently fine so that 
            \begin{equation} \label{eq: modulus of continuity of g_n,lambda*g'_n}
                  \diam \left(g_{n,\lambda}*g'_n(\Delta), \norm{\cdot}\right) < \min\left(\frac{\varepsilon_n}{8}, \frac{\delta_n}{2}\right)
            \end{equation}   
            for any simplex $\Delta\subset P_n*Q_n$ and $\lambda\in [N_n]$.

We define a continuous map $q_n:\mathcal{X}\to P_n*Q_n$ as follows.
Let $x\in \mathcal{X}$ and take $a\in \mathbb{Z}$ with $0\in I_{\varphi_n}(x,a)$.
We take $b\in \mathbb{Z}$ such that $b\equiv a \, (\mathrm{mod} N_n)$ and $0\in b+[N_n]$.
Set 
\[  q_n(x) = \left\{1-\alpha\left(\dist(0,\partial_{\varphi_n}(x))\right)\right\} \pi_n(T^b x) \oplus
                \alpha\left(\dist(0,\partial_{\varphi_n}(x))\right) \pi'_n(x). \]
We take $0<\tau_{n+1}<1/(n+1)$ such that 
\begin{itemize}
    \item If $d(x,y)<\tau_{n+1}$ then $\norm{f_n(x)-f_n(y)}< \min(\varepsilon_n/8,\delta_n/2)$ and 
             \begin{equation} \label{eq: continuity of dist(0, partial) in the choice of tau}
                 \left|\dist\left(0, \partial_{\varphi_n}(x)\right) - \dist\left(0, \partial_{\varphi_n}(y)\right)\right| < \frac{1}{4}.                            
             \end{equation}
    \item If $d(x,y) < \tau_{n+1}$ and $(-1/4,1/4)\subset I_{\varphi_n}(x,a)$ 
            then 
            $I_{\varphi_n}(y,a)$ contains $0$ as an interior point.    
    \item $\tau_{n+1}$ is smaller than the Lebesgue number of the open cover 
            $\left\{q_n^{-1}\left(O_{P_n*Q_n}(v)\right)\right\}_{v\in \ver(P_n*Q_n)}$:
            \begin{equation*}
                 \tau_{n+1} < LN\left(\mathcal{X}, d, \{q_n^{-1}(O_{P_n*Q_n}(v))\}_{v\in \ver(P_n*Q_n)}\right).
            \end{equation*}    
\end{itemize}

Take a $\tau_{n+1}$-embedding $\pi'_{n+1}:(\mathcal{X},d) \to Q_{n+1}$ with a simplicial complex $Q_{n+1}$.
We can assume $\diam \left((\pi'_{n+1})^{-1}\left(O_{Q_{n+1}}(w)\right),d\right) < \tau_{n+1}$ for every $w\in \ver(Q_{n+1})$.
By Lemma \ref{lemma: preparations on linear maps} (3), we can take a linear map 
$\tilde{g}'_{n+1}:Q_{n+1} \to B_1^\circ(V)$ satisfying 
\begin{equation}  \label{eq: tilde{g}' approximate f_n}
   \norm{\tilde{g}'_{n+1}(\pi'_{n+1}(x))-f_n(x)} < \min\left(\frac{\varepsilon_n}{8}, \frac{\delta_n}{2}\right). 
\end{equation}

We can find $N_{n+1}> N_n$ such that 
\begin{itemize}
    \item There exists a $\tau_{n+1}$-embedding $\pi_{n+1}:(\mathcal{X},d_{N_{n+1}})\to P_{n+1}$ with a simplicial 
             complex $P_{n+1}$ of dimension less than $N_{n+1}\left(D+\frac{1}{n+1}\right)$.
    \item 
    \[   1 + \sup_{x\in \mathcal{X}} |\partial_{\varphi_n}(x)\cap [0,N_{n+1}]| 
         < \frac{N_{n+1}}{2n N_n^2 \cdot A(P_n*Q_n)}. \]
     Here we have used Condition \ref{condition: dynamical PS theorem} (3)\footnote{Here is a technical point.
     The number $A(P_n*Q_n)$ is defined by using the simplicial complex structure of $P_n*Q_n$.
     We use the natural simplicial complex structure of the join $P_n*Q_n$ here, not its subdivision introduced in 
     (\ref{eq: modulus of continuity of g_n,lambda*g'_n}).} for $\varphi_n$.
\end{itemize}
By subdividing $P_{n+1}$ sufficiently fine, we assume that 
for any two simplexes $\Delta, \Delta'\subset P_{n+1}$ with $\Delta\cap \Delta' \neq \emptyset$ 
\begin{equation} \label{eq: P_{n+1} is sufficiently fine}
   \diam \left(\pi_{n+1}^{-1}\left(O_{P_{n+1}}(\Delta)\right)\cup \pi_{n+1}^{-1}\left(O_{P_{n+1}}(\Delta')\right), d_{N_{n+1}}\right) < \tau_{n+1}.
\end{equation}
Moreover by Lemma \ref{lemma: essential map} we can assume that 
$\pi_{n+1}:\mathcal{X}\to P_{n+1}$ is essential.

We apply Lemma \ref{lemma: approximately commuting simplicial maps} (with $P=P_{n+1}$, $Q=P_n*Q_n$, $N=N_{n+1}$, 
and $Q' = P_n \text{ or } Q_n$) 
to continuous maps
$\pi_{n+1}:\mathcal{X}\to P_{n+1}$ and $q_n\circ T^\lambda: \mathcal{X}\to P_n*Q_n$, $\lambda\in [N_{n+1}]$.
(The assumption of Lemma \ref{lemma: approximately commuting simplicial maps} is satisfied because of the above 
Lebesgue number condition.)
Then we get simplicial maps $h_\lambda:P_{n+1}\to P_n*Q_n$, $\lambda\in [N_{n+1}]$, so that
\begin{itemize}
   \item For every $\lambda\in [N_{n+1}]$ and $x\in \mathcal{X}$, 
           the two points $h_\lambda(\pi_{n+1}(x))$ and $q_n(T^\lambda x)$ belong to the 
           same simplex of $P_n*Q_n$.
   \item  Let $\Delta\subset P_{n+1}$ be a simplex and $\lambda\in [N_{n+1}]$.
            If $\pi_{n+1}^{-1}(O_{P_{n+1}}(\Delta)) \subset T^{-\lambda} q_n^{-1}(P_n)$ then $h_\lambda(\Delta) \subset P_n$.
            Similarly, if $\pi_{n+1}^{-1}(O_{P_{n+1}}(\Delta)) \subset T^{-\lambda} q_n^{-1}(Q_n)$,
            then $h_\lambda(\Delta) \subset Q_n$.
   \item For $\lambda, \lambda'\in [N_{n+1}]$ and a simplex $\Delta\subset P_{n+1}$, if 
           $q_n\circ T^\lambda = q_n\circ T^{\lambda'}$ on $\pi_{n+1}^{-1}\left(O_{P_{n+1}}(\Delta)\right)$ then 
           $h_\lambda = h_{\lambda'}$ on $\Delta$.         
\end{itemize}

Let $u\in P_{n+1}$, and let $\Delta\subset P_{n+1}$ be a simplex containing $u$.
Since $\pi_{n+1} : \mathcal{X}\to P_{n+1}$ is essential, there exists $x\in \pi_{n+1}^{-1}\left(O_{P_{n+1}}(\Delta)\right)$.
Let $\lambda\in [N_{n+1}]$.
We take $a, b\in \mathbb{Z}$ such that $\lambda\in I_{\varphi_n}(x,a)$, $b\equiv a \, (\mathrm{mod}\, N_n)$ and 
$\lambda\in b+[N_n]$.
We set 
\[  \tilde{g}_{n+1,\lambda}(u) = g_{n,\lambda-b}*g'_n \left(h_\lambda(u)\right) \in B^\circ_1(V). \]
We will check that this is independent of the choices of $x$ and $a$; see Claim \ref{claim: tilde{g}_{n+1} is linear} below.
In this setting we have $f_n(T^\lambda x) = g_{n,\lambda-b}*g'_n(q_n(T^\lambda x))$.
It follows from (\ref{eq: modulus of continuity of g_n,lambda*g'_n}) and the first condition of $h_\lambda$ that
\begin{equation}  \label{eq: tilde{g}_{n+1,lambda} approximate f_n}
    \norm{\tilde{g}_{n+1,\lambda}(\pi_{n+1}(x)) - f_{n}(T^\lambda x)} < \min\left(\frac{\varepsilon_n}{8}, \frac{\delta_n}{2}\right). 
\end{equation}

\begin{claim} \label{claim: tilde{g}_{n+1} is linear}
   The map $\tilde{g}_{n+1,\lambda}: P_{n+1}\to V$ is a linear map.
\end{claim}

\begin{proof}
The point is that the above definition of $\tilde{g}_{n+1,\lambda}(u)$ is independent of the choices of $x$ and $a$.
Let $\Delta'\subset P_{n+1}$ be another simplex containing $u$ and pick $x'\in \pi_{n+1}^{-1}\left(O_{P_{n+1}}(\Delta')\right)$.
We take $a',b'\in \mathbb{Z}$ such that $\lambda\in I_{\varphi_n}(x',a')$, $b'\equiv a' \, (\mathrm{mod} N_n)$ and 
$\lambda\in b'+[N_n]$.

\textbf{Case 1:} Suppose $\dist\left(\lambda, \partial_{\partial_n}(x)\right) > 1/4$.
We have $d(T^\lambda x, T^\lambda x') < \tau_{n+1}$ by (\ref{eq: P_{n+1} is sufficiently fine}).
From the second condition of the choice of $\tau_{n+1}$, we have $a=a'$ and $b=b'$. 
Hence 
\[  g_{n,\lambda-b}* g'_n (h_\lambda(u)) = g_{n, \lambda-b'}* g'_n(h_\lambda(u)). \]

\textbf{Case 2:} Suppose $\dist\left(\lambda, \partial_{\varphi_n}(x)\right) \leq 1/4$.
Take an arbitrary $y\in \pi_{n+1}^{-1} \left(O_{P_{n+1}}(\Delta)\right) \cup \pi_{n+1}^{-1}\left(O_{P_{n+1}}(\Delta')\right)$.
We have $d(T^\lambda x, T^\lambda y) < \tau_{n+1}$ by (\ref{eq: P_{n+1} is sufficiently fine}) and hence
$\dist\left(\lambda, \partial_{\varphi_n}(y)\right) < 1/2$ by the condition 
(\ref{eq: continuity of dist(0, partial) in the choice of tau}) of the choice of $\tau_{n+1}$.
Then $q_n(T^\lambda y) = \pi'_n (T^\lambda y) \in Q_n$.
Hence $\pi_{n+1}^{-1} \left(O_{P_{n+1}}(\Delta)\right) \cup \pi_{n+1}^{-1}\left(O_{P_{n+1}}(\Delta')\right) \subset T^{-\lambda} q_n^{-1}(Q_n)$.
So $h_\lambda(\Delta)\cup h_\lambda(\Delta') \subset Q_n$ by the second condition of $h_\lambda$.
This implies 
\[  g_{n,\lambda-b}*g'_n(h_\lambda(u)) = g'_n(h_\lambda(u)) = g_{n,\lambda-b'}*g'_n(h_\lambda(u)). \]
\end{proof}

\begin{claim}  \label{claim: covering number of tilde{g}_{n+1}}
Set $\tilde{g}_{n+1}(u) = \left(\tilde{g}_{n+1,0}(u), \tilde{g}_{n+1,1}(u), \dots, \tilde{g}_{n+1,N_{n+1}-1}(u)\right)$.
Then for $0<\varepsilon < \varepsilon_n$
\[  \#\left(\tilde{g}_{n+1}(P_{n+1}), \norm{\cdot}_{N_{n+1}}, \varepsilon\right) < 
      4^{N_{n+1}} \left(\frac{1}{\varepsilon}\right)^{\left(D+\frac{2}{n}\right) N_{n+1}}. \]
\end{claim}

\begin{proof}
This is similar to the proof of Claim \ref{claim: estimate of f_n(X)}.
Let $\Delta\subset P_{n+1}$ be a simplex.
For $b\in \mathbb{Z}\cap [0,N_{n+1}-N_n]$, a discrete interval $J = \{b, b+1, \dots, b+N_n-1\}$ is said to be 
\textbf{good for} $\Delta$ if there exist $x\in \pi_{n+1}^{-1}\left(O_{P_{n+1}}(\Delta)\right)$ and $a\in \mathbb{Z}$
such that $b\equiv a \, (\mathrm{mod} N_n)$ and $(b-1,b+N_n)\subset I_{\varphi_n}(x,a)$.
This condition implies that every $y\in \pi_{n+1}^{-1}\left(O_{P_{n+1}}(\Delta)\right)$ satisfies 
$(b-3/4,b+N_n-1/4)\subset I_{\varphi_n}(y,a)$ by (\ref{eq: P_{n+1} is sufficiently fine})
and the first and second conditions of $\tau_{n+1}$.
In particular for every $\lambda\in J$ and $y\in \pi_{n+1}^{-1}\left(O_{P_{n+1}}(\Delta)\right)$
\[   q_n(T^\lambda y) = q_n(T^b y)  = \pi_n(T^b y) \in P_n. \]
Then the second and third conditions of $h_\lambda$ imply $h_\lambda(u) = h_b(u) \in P_n$ on $u\in \Delta$ and $\lambda\in J$.
Hence for $u\in \Delta$
\[  \left(\tilde{g}_{n+1,\lambda}(u) \right)_{\lambda\in J} = g_n\left(h_b(u)\right) \in g_n(P_n) \subset V^{N_n}. \]
As a conclusion, if $J$ is good for $\Delta$ then $\left(\tilde{g}_{n+1,\lambda}(u) \right)_{\lambda\in J} \in g_n(P_n)$ for all 
$u\in \Delta$.

Let $\mathcal{J}_\Delta$ be the union of $J = \{b,b+1,\dots, b+N_n-1\} \subset [N_{n+1}]$ which are good for 
$\Delta$.
The number of possibilities of $\mathcal{J}_\Delta$ (when $\Delta\subset P_{n+1}$ varies) is bounded by $2^{N_{n+1}}$.
Then $\#\left(\tilde{g}_{n+1}(P_{n+1}), \norm{\cdot}_{N_{n+1}}, \varepsilon\right)$ is bounded by 
\begin{equation*}
   \begin{split}
     &2^{N_{n+1}} \times \underbrace{\#\left(g_n(P_n),\norm{\cdot}_{N_n},\varepsilon\right)^{\frac{N_{n+1}}{N_n}}}_{\text{contribution 
     over $\mathcal{J}_\Delta$}}  \\
    & \times \underbrace{\#\left(\bigcup_{\lambda\in [N_n]} g_{n,\lambda}*g'_n(P_n*Q_n), \norm{\cdot},\varepsilon\right)^{2N_n
     \sup_{x\in \mathcal{X}} |\partial_{\varphi_n}(x)\cap [0,N_{n+1}]|+2N_n}}_{\text{contribution over $[N_{n+1}]\setminus \mathcal{J}_\Delta$}}.
   \end{split}
\end{equation*}
We have 
\[  \#\left(\bigcup_{\lambda\in [N_n]} g_{n,\lambda}*g'_n(P_n*Q_n), \norm{\cdot},\varepsilon\right)  \leq 
     N_n \left(\frac{1}{\varepsilon}\right)^{A(P_n*Q_n)}. \]
Then we get the claim by Condition \ref{condition: dynamical PS theorem} (2) and (3) for 
$g_n$ and $\varphi_n$.
\end{proof}

By Lemma \ref{lemma: preparations on linear maps} (1) and $\dim P_{n+1} < \left(D+\frac{1}{n+1}\right) N_{n+1}$, 
we can take $0<\varepsilon_{n+1}<\varepsilon_n/2$ such that 
for any linear map $g:P_{n+1}\to B_1^\circ(V)^{N_{n+1}}$ 
\[  \#\left(g(P_{n+1}),\norm{\cdot}_{N_{n+1}},\varepsilon\right) < \left(\frac{1}{\varepsilon}\right)^{\left(D+\frac{1}{n+1}\right) N_{n+1}}
     \quad (0<\varepsilon < \varepsilon_{n+1}).  \]

By lemma \ref{lemma: preparations on linear maps} (2) and the above (\ref{eq: tilde{g}' approximate f_n}) and
(\ref{eq: tilde{g}_{n+1,lambda} approximate f_n}), 
we can find linear embeddings $g'_{n+1}:Q_{n+1}\to B_1^\circ (V)$ and 
$g_{n+1,\lambda}:P_{n+1}\to B_1^\circ(V)$, $\lambda\in [N_{n+1}]$, such that 
they satisfy Condition \ref{condition: dynamical PS theorem} (1) and for any $x\in \mathcal{X}$ and $u\in P_{n+1}$
\begin{equation}  \label{eq: g and g' approximate f_n in dynamical PS theorem}
  \begin{split}
    \norm{g'_{n+1}(\pi'_{n+1}(x)) - f_n(x)} < \min\left(\frac{\varepsilon_n}{8}, \frac{\delta_n}{2}\right), \\
    \norm{g_{n+1, \lambda}(\pi_{n+1}(x)) - f_n(T^\lambda x)} <   \min\left(\frac{\varepsilon_n}{8}, \frac{\delta_n}{2}\right), 
  \end{split}
\end{equation}
\begin{equation} \label{eq: g is close to g' in dynamical PS theorem}
  \norm{g_{n+1,\lambda}(u) - \tilde{g}_{n+1,\lambda}(u)} < \frac{\varepsilon_{n+1}}{4}. 
\end{equation}
It follows from (\ref{eq: g is close to g' in dynamical PS theorem}) and Claim \ref{claim: covering number of tilde{g}_{n+1}}
that for $\varepsilon_{n+1} \leq \varepsilon < \varepsilon_n$
\begin{equation*}
   \begin{split}
     \#\left(g_{n+1}(P_{n+1}), \norm{\cdot}_{N_{n+1}}, \varepsilon \right) & \leq 
     \#\left(\tilde{g}_{n+1}(P_{n+1}), \norm{\cdot}_{N_{n+1}}, \varepsilon-\frac{\varepsilon_{n+1}}{2}\right) \\
     & \leq  \#\left(\tilde{g}_{n+1}(P_{n+1}), \norm{\cdot}_{N_{n+1}}, \frac{\varepsilon}{2}\right) \\
     & <   4^{N_{n+1}} \left(\frac{2}{\varepsilon}\right)^{\left(D+\frac{2}{n}\right) N_{n+1}}. 
   \end{split}  
\end{equation*}
From the choice of $\varepsilon_{n+1}$, for $0<\varepsilon<\varepsilon_{n+1}$
\[  \#\left(g_{n+1}(P_{n+1}),\norm{\cdot}_{N_{n+1}},\varepsilon\right) < \left(\frac{1}{\varepsilon}\right)^{(D+\frac{1}{n+1})N_{n+1}}. \]
Hence $g_{n+1,\lambda}$ satisfy Condition \ref{condition: dynamical PS theorem} (2).

From Lemma \ref{lemma: dynamical tiling} we can choose 
a continuous function $\varphi_{n+1}:\mathcal{X}\to [0,1]$ satisfying 
Condition \ref{condition: dynamical PS theorem} (3).
From (\ref{eq: g and g' approximate f_n in dynamical PS theorem}), $f_{n+1}$ satisfies Condition \ref{condition: dynamical PS theorem} (5).
Since $g_{n+1,\lambda}$ and $g'_{n+1}$ satisfy Condition \ref{condition: dynamical PS theorem} (1),
$f_{n+1}$ is a $1/(n+1)$-embedding with respect to $d$.
Since ``$1/(n+1)$-embedding'' is an open condition, we can choose $\delta_{n+1}>0$ satisfying 
Condition \ref{condition: dynamical PS theorem} (4).

We have established all the data for the $(n+1)$-th step.
\end{proof}

\section{Example: algebraic actions}  \label{section: example: algebraic actions}

We study an example in this section.
Probably the example below can be more generalized (e.g. more general group actions), 
but we restrict ourselves to a simple case because 
our purpose here is just to illustrate the concepts studied in the paper.
We plan to study more examples in future works.

Set $\mathbb{T} = \mathbb{R}/\mathbb{Z}$.
Let $r>0$ be an integer and consider the shift $\sigma :(\mathbb{T}^r)^{\mathbb{Z}}\to (\mathbb{T}^r)^{\mathbb{Z}}$
on the alphabet $\mathbb{T}^r = \mathbb{R}^r/\mathbb{Z}^r$.
This becomes a compact Abelian group under the component-wise addition.
A subset $\mathcal{X}\subset \left(\mathbb{T}^r\right)^{\mathbb{Z}}$ is called an \textbf{algebraic action} if 
it is a $\sigma$-invariant closed subgroup\footnote{Since $r$ is finite, this is more restricted than in the literatures 
\cite{Schmidt, Li--Liang}. They consider automorphisms of general compact Abelian groups.  
But we study only the restricted class here for simplicity.}.
Equivalently \cite[Definitions 3.7 and 4.1, Theorems 3.8 and 4.2]{Schmidt}, 
a subset $\mathcal{X}\subset \left(\mathbb{T}^r\right)^{\mathbb{Z}}$ is an algebraic action if and only 
if there exist a positive integer $a$ and a closed subgroup $H\subset \left(\mathbb{T}^r\right)^a$
such that 
\[  \mathcal{X} = \left\{ (x_n)_{n\in \mathbb{Z}}\in (\mathbb{T}^r)^\mathbb{Z}\middle|\, 
     (x_n,x_{n+1},\dots, x_{n+a-1}) \in H \> (\forall n\in \mathbb{Z})\right\}. \]

We define metrics $\rho$ and $\rho_r$ on $\mathbb{T}$ and $\mathbb{T}^r$ respectively by 
\[  \rho(t,t') = \min_{n\in \mathbb{Z}} |t-t'-n|, \]
\[ \rho_r\left((t_1,\dots, t_r), (t'_1,\dots, t'_r)\right) = \max_{1\leq i\leq r} \rho(t_i,t'_i). \]
We define a metric $d$ on $(\mathbb{T}^r)^\mathbb{Z}$ by 
\[ d(x,y) = \sum_{n\in \mathbb{Z}} 2^{-|n|} \rho_r(x_n, y_n), \quad (x_n,y_n\in \mathbb{T}^r). \]
Later we will use the fact that $d$ is \textit{homogeneous}, namely it is invariant under the addition
\[  d(x+z, y+z) = d(x,y), \quad \left(x,y,z\in \left(\mathbb{T}^r\right)^{\mathbb{Z}} \right). \]

For $N>0$ we denote by $\pi_N:(\mathbb{T}^r)^\mathbb{Z}\to (\mathbb{T}^r)^N$ the projection to the 
$\{0,1,2,\dots, N-1\}$-coordinates:
\[  \pi_N(x)= (x_0,\dots,x_{N-1}). \]
Let $\mathcal{X}\subset (\mathbb{T}^r)^{\mathbb{Z}}$ be an algebraic action.
Following Gromov \cite[\S 1.9]{Gromov} we define the \textbf{projective dimension} of $\mathcal{X}$ by 
\[  \mathrm{prodim}(\mathcal{X}) = \lim_{N\to \infty} \frac{\dim \pi_N(\mathcal{X})}{N}. \]
Here $\dim \pi_N(\mathcal{X})$ is the topological dimension of $\pi_N(\mathcal{X})$.
This limit always exists because $\dim \pi_N(\mathcal{X})$ is subadditive in $N$.
(Note that, a priori, the projective dimension may depend on the way of the embedding $\mathcal{X}\subset (\mathbb{T}^r)^\mathbb{Z}$.
So the notation $\mathrm{prodim}(\mathcal{X})$ might be misleading.
But we use it for simplicity.)

Li--Liang \cite[Theorem 4.1, Theorem 7.2]{Li--Liang} proved:
\begin{equation}  \label{eq: mean dimension of algebraic actions}
  \mdim(\mathcal{X},\sigma) = \mdim_{\mathrm{M}}(\mathcal{X},\sigma,d) = \mathrm{prodim}(\mathcal{X}).
\end{equation}
Indeed they proved more general results.
But we stick to this simple case.
Since mean Hausdorff dimension is bounded between mean dimension and metric mean dimension 
(Proposition \ref{prop: mean Hausdorff dimension dominates mean dimension}), we also have 
\[  \mdim_{\mathrm{H}}(\mathcal{X}, \sigma, d) = \mathrm{prodim}(\mathcal{X}). \]

The purpose of this section is to show:

\begin{proposition}  \label{prop: rate distortion dimension of algebraic actions}
Let $\mathcal{X}\subset (\mathbb{T}^r)^{\mathbb{Z}}$ be an algebraic action and $\mu$ the normalized Haar measure on it
($\mu(\mathcal{X}) = 1$). Then 
\[  \rdim(\mathcal{X},\sigma,d,\mu) = \mathrm{prodim}(\mathcal{X}). \]
Hence the rate distortion dimension with respect to the Haar measure coincides with the mean dimension, mean Hausdorff dimension and 
metric mean dimension.
\end{proposition}

Therefore algebraic actions provide natural examples where all the dynamical dimensions studied in this paper 
coincide with each other.

In the sequel, we also include the proof of (\ref{eq: mean dimension of algebraic actions})
for the completeness. The idea of the proof is the same as \cite{Li--Liang}.

The following lemma is a key estimate \cite[Lemma 4.2]{Li--Liang}.

\begin{lemma}  \label{lemma: widim and covering number of compact Abelian group}
Let $A$ be a compact Abelian group with a metric $\boldsymbol{d}$ and 
$f:A\to \mathbb{T}^r$ a continuous homomorphism satisfying $\rho_r(f(x),f(y)) \leq \boldsymbol{d}(x,y)$.
Then for any $0<\varepsilon<1/4$
  \begin{equation}  \label{eq: widim of compact Abelian group}
    \widim_\varepsilon(A,\boldsymbol{d}) \geq \dim f(A),
  \end{equation}
  \begin{equation}  \label{eq: covering number of compact Abelian group}
    \#(A,\boldsymbol{d},\varepsilon) \geq \left(\frac{1}{4\varepsilon}\right)^{\dim f(A)}.
  \end{equation}
  Here $\dim f(A)$ is the topological dimension of $f(A)$.
\end{lemma}

\begin{proof}
We can assume that $f(A)$ is connected. (If it is not, we replace $A$ with the inverse by $f$ of the connected component of 
$f(A)$ through the origin.)
Let $\pi:\mathbb{R}^r\to \mathbb{T}^r$ be the natural covering map and set $V = \pi^{-1}(f(A))$.
$V$ is a subvector space of $\mathbb{R}^r$ of dimension $\dim f(A)$.
We consider the $\ell^\infty$-norm $\norm{\cdot}_\infty$ on $\mathbb{R}^r$.

\begin{claim} \label{claim:  existence of inverse}
  There exists a continuous homomorphism $g:V\to A$ satisfying $f\circ g = \pi|_V$.
\end{claim}

\begin{proof}
Let $M$ be the Pontyagin dual of $f(A)$.
The dual group $\hat{V}$ of $V$ (we denote Pontragin duality by hat) is identified with $M\otimes \mathbb{R}$
and $(\pi|_V)\hat{}(m) = m\otimes 1$ for $m\in M$.
It is enough to construct a homomorphism $h:\hat{A}\to M\otimes \mathbb{R}$ satisfying 
$h\circ \hat{f} (m) = m\otimes 1$ for $m\in M$. 
(Note that every homomorphism defined on $\hat{A}$ automatically becomes continuous because its topology is discrete.)

Since the map $\hat{f}:M\to \hat{A}$ is injective,
\[  \hat{f}\otimes \mathrm{id}: M\otimes \mathbb{R} \to \hat{A}\otimes \mathbb{R}  \]
is also injective.
Take an $\mathbb{R}$-linear map $\varphi:\hat{A}\otimes \mathbb{R} \to M\otimes \mathbb{R}$
satisfying $\varphi \circ (\hat{f}\otimes \mathrm{id}) = \mathrm{id}_{M\otimes \mathbb{R}}$. 
Then the map $h:\hat{A}\to M\otimes \mathbb{R}$ defined by $h(a) = \varphi(a\otimes 1)$ satisfies 
$h\circ \hat{f}(m) = m\otimes 1$ for $m\in M$.
\end{proof}

Let $B_{1/4}(V)$ be the closed $1/4$-ball of $V$ around the origin.
Note that the map $\pi: (B_{1/4}(\mathbb{R}^r),\norm{\cdot}_\infty) \to (\mathbb{T}^r,\rho_r)$ is
an isometry.
So for $x,y\in B_{1/4}(V)$
\[  \norm{x-y}_\infty = \rho_r\left(f(g(x)), f(g(y)\right) \leq \mathbf{d}\left(g(x), g(y)\right). \]
This implies 
\[  \widim_\varepsilon\left(B_{1/4}(V), \norm{\cdot}_\infty\right) \leq \widim_\varepsilon(A,\mathbf{d}). \]
We have $\widim_\varepsilon \left(B_{1/4}(V),\norm{\cdot}_\infty \right) = \dim V = \dim f(A)$ for $0<\varepsilon <1/4$
by (\ref{eq: widim of ball}) in Example 
\ref{example: widim and covering number of ball} in \S \ref{subsection: topological and metric mean dimensions}.
This shows (\ref{eq: widim of compact Abelian group}).
We can prove (\ref{eq: covering number of compact Abelian group}) in the same way
by using (\ref{eq: covering number of ball}) in Example \ref{example: widim and covering number of ball}.
\end{proof}

\begin{proof}[Proof of (\ref{eq: mean dimension of algebraic actions})]
Let $\mathcal{X}\subset \left(\mathbb{T}^r\right)^{\mathbb{Z}}$ be an algebraic action.
First we prove $\mdim(\mathcal{X}, \sigma) \geq \mathrm{prodim}(\mathcal{X})$.
Consider the projection $\pi_N: \mathcal{X}\to \left(\mathbb{T}^r\right)^N = \mathbb{T}^{rN}$.
This satisfies $\rho_{rN}\left(\pi_N(x),\pi_N(y)\right) \leq d_N(x,y)$.
So we can use Lemma \ref{lemma: widim and covering number of compact Abelian group} and get 
\[ \widim_\varepsilon(\mathcal{X},d_N) \geq \dim \pi_N(\mathcal{X}), \quad (0<\varepsilon <1/4). \]
Divide this by $N$ and take the limits with respect to $N$ and then $\varepsilon$.
We get $\mdim(\mathcal{X},\sigma) \geq \mathrm{prodim}(\mathcal{X})$.

Next we prove $\overline{\mdim}_{\mathrm{M}}(\mathcal{X},\sigma,d) \leq \mathrm{prodim}(\mathcal{X})$.
This completes the proof of (\ref{eq: mean dimension of algebraic actions}) because metric mean dimension always
dominates mean dimension (Theorem \ref{theorem: metric mean dimension dominates mean dimension}).
Notice that the following map is an isometric embedding
\begin{equation*}
   \begin{split}
    \left(\pi_{M+N}(\mathcal{X}), \rho_{r(M+M)}\right) 
    &\to \left(\pi_M(\mathcal{X}) \times \pi_N(\mathcal{X}), \rho_{rM}\times \rho_{rN}\right), \\
     \pi_{M+N}(x) &\mapsto \left(\pi_M(x), \pi_N(\sigma^M x)\right),
   \end{split}
\end{equation*}     
where the metric of the right-hand side is given by 
\[   \rho_{rM}\times \rho_{rN}\left((x,y), (z,w)\right) = \max\left(\rho_{rM}(x,z), \rho_{rN}(y,w)\right). \]
It follows that $\#\left(\pi_N(\mathcal{X}), \rho_{rN},\varepsilon\right)$ is subadditive in $N$ and hence 
for any $\varepsilon >0$
\begin{equation}  \label{eq: subadditivity of covering numbers of projections}
   \lim_{N\to \infty} \frac{\log \#(\pi_N(\mathcal{X}),\rho_{rN},\varepsilon)}{N}
   =  \inf_{N>0} \frac{\log \#(\pi_N(\mathcal{X}),\rho_{rN},\varepsilon)}{N}.
\end{equation}

For $A\subset \mathbb{R}$, let $\pi_A:\mathcal{X} \to \left(\mathbb{T}^r\right)^{A\cap \mathbb{Z}}$ be
the projection to $A\cap \mathbb{Z}$-coordinates.
Let $\varepsilon >0$ and take $L=L(\varepsilon)>0$ satisfying $\sum_{|n|>L} 2^{-|n|} < \varepsilon/4$.
Then 
\begin{equation*}
    \begin{split}
     \#(\mathcal{X}, d_N,\varepsilon) &\leq \#(\pi_{[-L,N+L]}(\mathcal{X}), \rho_{r(N+2L+1)}, \varepsilon/4) \\
     &= \#(\pi_{N+2L+1}(\mathcal{X}), \rho_{r(N+2L+1)}, \varepsilon/4).
    \end{split}
\end{equation*}    
Noting the above (\ref{eq: subadditivity of covering numbers of projections}), we get 
\begin{equation*}
   \begin{split}
     S(\mathcal{X},\sigma,d,\varepsilon) & = \lim_{N\to \infty} \frac{\log  \#(\mathcal{X}, d_N,\varepsilon)}{N}  \\
     & \leq  \lim_{N\to \infty} \frac{\log \#(\pi_{N+2L+1}(\mathcal{X}), \rho_{r(N+2L+1)}, \varepsilon/4)}{N} \\
     & =  \inf_{N>0} \frac{\log \#(\pi_N(\mathcal{X}),\rho_{rN},\varepsilon/4)}{N}.
   \end{split}
\end{equation*}
Thus for any $N>0$
\begin{equation*}
   \begin{split}
     \overline{\mdim}_{\mathrm{M}}(\mathcal{X},\sigma,d) 
    & = \limsup_{\varepsilon \to 0} \frac{S(\mathcal{X},\sigma,d,\varepsilon)}{\log(1/\varepsilon)}  \\
    &  \leq \frac{1}{N} \limsup_{\varepsilon\to 0} \frac{\log \#(\pi_N(\mathcal{X}),\rho_{rN},\varepsilon/4)}{\log(1/\varepsilon)}.
   \end{split}
\end{equation*}      
$\pi_N(\mathcal{X})$ is a closed subgroup of $\mathbb{T}^{rN}$ and hence a smooth submanifold.
Then the upper Minkowski dimension 
\[   \overline{\dim}_{\mathrm{M}}\left(\pi_N(\mathcal{X}), \rho_{rN}\right) 
    = \limsup_{\varepsilon \to 0} \frac{\log\#(\pi_N(\mathcal{X}), \rho_{rN}, \varepsilon)}{\log(1/\varepsilon)} \]
is equal to the topological dimension $\dim \pi_N(\mathcal{X})$.    
Thus for any $N$
\[  \overline{\mdim}_{\mathrm{M}}(\mathcal{X},\sigma,d) \leq \frac{\dim \pi_N(\mathcal{X})}{N}. \]
Let $N\to \infty$. This proves $\overline{\mdim}_{\mathrm{M}}(\mathcal{X},\sigma,d) \leq \mathrm{prodim}(\mathcal{X})$.
\end{proof}

For any $N>0$ we define a distance $\bar{d}_N$ on $(\mathbb{T}^r)^{\mathbb{Z}}$ by (see \S \ref{subsection: L^1 mean Hausdorff dimension})
\[  \bar{d}_N(x,y) = \frac{1}{N} \sum_{n=0}^{N-1} d(\sigma^n x, \sigma^n y). \]

\begin{lemma}   \label{lemma: separating number of algebraic action}
Let $\mathcal{X} \subset (\mathbb{T}^r)^{\mathbb{Z}}$ be an algebraic action.
For any $0<\delta<1$ there exists $\varepsilon_0 = \varepsilon_0(\delta)>0$ such that 
for any $0<\varepsilon <\varepsilon_0$ and $N>0$ 
\[  \#_{\mathrm{sep}} (\mathcal{X},\bar{d}_N,\varepsilon) \geq 4^{-N} (1/\varepsilon)^{(1-\delta) \dim \pi_N(\mathcal{X})}. \]
Recall that $\#_{\mathrm{sep}} (\mathcal{X},\bar{d}_N,\varepsilon)$ is the maximum cardinality of 
$\{x_1,\dots, x_n\}\subset \mathcal{X}$ satisfying $\bar{d}_N(x_i,x_j) \geq \varepsilon$ for $i\neq j$.
\end{lemma}

\begin{proof}
We use (\ref{eq: covering and separating numbers}) in \S \ref{subsection: topological and metric mean dimensions}
and the estimate in Remark \ref{remark: metric mean dimension and tame growth of covering numbers}:
For any $L>0$
\begin{equation*}
   \begin{split}
     \log\#_{\mathrm{sep}}(\mathcal{X},\bar{d}_N,\varepsilon)  &\geq \log \#(\mathcal{X},\bar{d}_N,3\varepsilon) \\
     & \geq \log \#(\mathcal{X},d_N, 6L\varepsilon) -N -\frac{N}{L} \log\#(\mathcal{X},d,3\varepsilon). 
   \end{split}
\end{equation*}   
Let $L = (1/24) (1/\varepsilon)^\delta$. Then 
\[   \log\#_{\mathrm{sep}}(\mathcal{X},\bar{d}_N,\varepsilon)  \geq 
     \log \#\left(\mathcal{X},d_N,\frac{\varepsilon^{1-\delta}}{4}\right) - N 
     -24 N \varepsilon^\delta \log\#(\mathcal{X},d,3\varepsilon). \]
$(\mathcal{X},d)$ has the tame growth of covering numbers (see Example \ref{example: tame growth of covering numbers}).
So there exists $\varepsilon_0>0$ so that for any $0<\varepsilon<\varepsilon_0$
\[  \frac{\varepsilon^{1-\delta}}{4} < \frac{1}{4}, \quad 
     \varepsilon^\delta \log \#(\mathcal{X},d,3\varepsilon) < \frac{1}{24}. \]
By applying Lemma \ref{lemma: widim and covering number of compact Abelian group}
to $\pi_N: \mathcal{X}\to (\mathbb{T}^r)^{N}$, 
\[  \#\left(\mathcal{X},d_N,\frac{\varepsilon^{1-\delta}}{4}\right)  \geq \left(\frac{1}{\varepsilon^{1-\delta}}\right)^{\dim \pi_N(\mathcal{X})}
    \quad (0<\varepsilon <\varepsilon_0). \]
Combining these estimates we get 
\[   \log\#_{\mathrm{sep}}(\mathcal{X},\bar{d}_N,\varepsilon)  \geq (1-\delta) \dim\pi_N(\mathcal{X}) \log(1/\varepsilon) - 2N. \]
This is equivalent to the statement. (Recall that the base of the logarithm is two.)
\end{proof}

\begin{proof}[Proof of Proposition \ref{prop: rate distortion dimension of algebraic actions}]
We know from Proposition \ref{prop: mean Hausdorff dimension dominates mean dimension} and
 (\ref{eq: mean dimension of algebraic actions}) that 
 \[  \overline{\rdim}(\mathcal{X},\sigma, d,\mu) \leq \mdim_{\mathrm{M}}(\mathcal{X},\sigma,d) = \mathrm{prodim}(\mathcal{X}). \]
So it is enough to prove $\underline{\rdim}(\mathcal{X},\sigma,d,\mu) \geq \mathrm{prodim}(\mathcal{X})$.
We can assume $\mathrm{prodim}(\mathcal{X}) >0$.

Recall that the distance $d$ is homogeneous.
In particular the measure $\mu\left(B^\circ_r(x,\bar{d}_N)\right)$ of the open ball 
around $x\in \mathcal{X}$
is independent of $x$.
So we denote it by $\mu\left(B^\circ_r(\bar{d}_N)\right)$.

Let $\{x_1,\dots,x_K\}$ be a separated set of $(\mathcal{X}, \bar{d}_N)$ with 
$K = \#_{\mathrm{sep}}\left(\mathcal{X},\bar{d}_N,\varepsilon\right)$.
Since the balls $B_{\varepsilon/2}^\circ(x_i,\bar{d}_N)$ are disjoint with each other,
$K \mu\left(B_{\varepsilon/2}^\circ(\bar{d}_N)\right) \leq 1$.
Let $0<\delta<1/2$.
It follows from Lemma \ref{lemma: separating number of algebraic action} that for $0<\varepsilon < \varepsilon_0(\delta)$ 
\[  \mu\left(B_{\varepsilon/2}^\circ(\bar{d}_N)\right) \leq K^{-1} \leq  4^{N} 
    \varepsilon^{(1-\delta)\dim \pi_N(\mathcal{X})}. \]
Since $\dim \pi_N(\mathcal{X})$ is subadditive in $N$,
\[ \mathrm{prodim}(\mathcal{X}) = \lim_{N\to \infty} \frac{\dim \pi_N(\mathcal{X})}{N} 
   = \inf_{N>0}  \frac{\dim \pi_N(\mathcal{X})}{N}  \]
   and we assumed that this is positive.
Therefore there exists $\varepsilon_1 = \varepsilon_1(\delta)>0$ such that for any $0<\varepsilon <\varepsilon_1$
\[  \mu\left(B^\circ_{\varepsilon/2}(\bar{d}_N)\right) \leq (\varepsilon/2)^{N(1-2\delta)\mathrm{prodim}(\mathcal{X})}. \]
This implies that for all $E\subset \mathcal{X}$ with $\diam (E, \bar{d}_N) < \varepsilon_1/2$
\[  \mu(E) \leq \left(\diam  (E, \bar{d}_N)\right)^{N(1-2\delta)\mathrm{prodim}(\mathcal{X})}. \]
We use Lemma \ref{lemma: L^1 metric and rate distortion dimension} in \S \ref{subsection: L^1 mean Hausdorff dimension}
and get 
\[  \underline{\rdim}(\mathcal{X},\sigma, d,\mu) \geq (1-2\delta) \mathrm{prodim}(\mathcal{X}). \]
Let $\delta\to 0$.
This proves $\underline{\rdim}(\mathcal{X},\sigma, d,\mu) \geq  \mathrm{prodim}(\mathcal{X})$.
\end{proof}

\vspace{0.5cm}

\address{ Elon Lindenstrauss \endgraf
Einstein Institute of Mathematics, Hebrew University, Jerusalem 91904, Israel}

\textit{E-mail address}: \texttt{elon@math.huji.ac.il}

\vspace{0.5cm}

\address{ Masaki Tsukamoto \endgraf
Department of Mathematics, Kyoto University, Kyoto 606-8502, Japan}

\textit{E-mail address}: \texttt{masaki.tsukamoto@gmail.com}

\end{document}